\documentclass[12pt,reqno]{amsart}
\usepackage{soul}
\usepackage{multicol,amsmath,graphics, mathtools, cancel,soul,color,bbm,amsfonts,dsfont,tikz,mathrsfs,amssymb,bm,hyperref,comment,bm,stackengine,scalerel,enumitem,multirow,cite}

\usepackage[left=1in, right=1in, top=1.1in,bottom=1.1in]{geometry}
\usetikzlibrary{matrix,patterns,positioning}
\numberwithin{equation}{section}

\hypersetup{
    colorlinks=true, 
    linktoc=all,     
    linkcolor=blue,  
}

\newcommand\Pig[1]{\scalerel*[7.5pt]{\bigg#1}{%
  \ensurestackMath{\addstackgap[1.5pt]{\Big#1}}}}
\newcommand\Pigl[1]{\mathopen{\Pig{#1}}}
\newcommand\Pigr[1]{\mathclose{\Pig{#1}}}
\renewcommand{\le}{\leqslant}
\renewcommand{\leq}{\leqslant}
\renewcommand{\ge}{\geqslant}
\renewcommand{\geq}{\geqslant}

\newcommand{\p}{\Pb}
\newcommand{\cip}{\overset{\Pb}{\to}}
\newcommand*{\pdot}{\mathbin{\scalerel*{\boldsymbol\odot}{\circ}}}
\newcommand{\Oh}{\mathrm{O}}
\newcommand{\oh}{\mathrm{o}}
\DeclareMathOperator*{\argmin}{ar\smash[b]{\mathrm{g}}\,min}
\newcommand{\rnd}{\mathrm{rnd}}
\newcommand{\Det}{\mathrm{det}}
\newcommand{\ua}{\uparrow}
\newcommand{\br}[1]{[{#1}]}

\newcommand{\E}{\mathbb{E}}

\newcommand{\N}{\mathbb{N}}
\newcommand{\Pb}{\mathbb{P}}

\newcommand{\R}{\mathbb{R}}

\newcommand{\vertiii}[1]{{\left\vert\kern-0.25ex\left\vert\kern-0.25ex\left\vert #1
    \right\vert\kern-0.25ex\right\vert\kern-0.25ex\right\vert}}

\DeclareMathOperator\sgn{sgn}

\def\R{\mathbb{R}}

\def\Var{\mathbb{V}\mathrm{ar}}
\def\Cov{\mathbb{C}\mathrm{ov}}

\def\ve{\varepsilon}

\def\dd{\mathrm{d}}
\def\MAD{\mathrm{MAD}}
\def\Det{\mathrm{det}}
\def\rnd{\mathrm{rnd}}
\def\RV{\mathrm{RV}}
\def\Beta{\mathrm{Beta}}
\def\ua{\uparrow}
\def\da{\downarrow}
\def\Ld{\mathfrak{L}}
\def\Kd{\mathfrak{K}}
\def\Wd{\mathfrak{W}}
\def\Pd{\mathfrak{P}}
\def\cid{\overset{\dd}{\to}}
\def\eqd{\overset{\dd}{=}}

\def\ov{\overline}
\def\un{\underline}

\DeclareMathOperator\med{med}
\def\1{\mathbbm{1}}

\newtheorem{thm}{Theorem}[section]
\newtheorem{lemma}[thm]{Lemma}
\newtheorem{cor}[thm]{Corollary}
\newtheorem{prop}[thm]{Proposition}

\newtheorem{definition}[thm]{Definition}
\theoremstyle{rem}
\newtheorem{rem}[thm]{Remark}

\makeatletter
\@namedef{subjclassname@1991}{2020 Mathematics Subject Classification}
\makeatother

\begin{document}

\title[Robust Scale Estimation in Additive Noise]
{Robust Scale Estimation in Additive Noise via Weighted Order Statistics}

\author[J. I. Gonz\'alez C\'azares]{Jorge I. Gonz\'alez C\'azares}
\address{Instituto de Investigaciones en Matem\'aticas Aplicadas y en Sistemas, Universidad Nacional Aut\'onoma de M\'exico, Circuito Escolar S/N, Ciudad Universitaria, Coyoac\'an, C.P. 04510, Ciudad de M\'exico, M\'exico}
\email{jorge.gonzalez@iimas.unam.mx}

\author[A. Jaramillo Gil]{Arturo Jaramillo Gil}
\address{Centro de Investigaci\'on en Matem\'aticas, A.C., Jalisco S/N, Valenciana, Guanajuato, Gto. 36023, M\'exico}
\email{jagil@cimat.mx}

\date{\today}
\begin{abstract}
This manuscript develops a non-parametric and robust framework for estimating the scale of additive noise in weakly sparse systems. The method does not require independence, prescribed dependence, or temporal regularity of the noise sequence. We introduce a class of order-statistic estimators based on comparing the sorted observations with deterministic or random proxies generated from a reference noise distribution. This purely spatial approach avoids preliminary filtering or temporal decorrelation, and therefore preserves the sparsity structure of the latent signal. We establish non-asymptotic concentration inequalities for weighted loss functions, with bounds that separate the contribution of the signal from the discrepancy between the ordered noise and the proxy. We then control this proxy discrepancy in independent and correlated regimes, including heavy-tailed reference laws. Finally, we apply the method to high-frequency observations of continuous-time stochastic processes, obtaining scale estimators for fractional Brownian motion and stable L\'evy noise in the presence of lower-variation additive perturbations.
\end{abstract}



\maketitle

\section{Introduction}
\noindent Estimating the scale of the noise in high-dimensional observation vectors corrupted by additive perturbations is a central problem in statistical signal processing, non-parametric regression, high-frequency econometrics, and related areas. In many instances of these problems, the data are represented by a discrete collection $X_k$, for $1\leq k\leq n$, obeying an additive relation 
of the form
\begin{equation}
\label{eq:X=Y+Z}
X_{k} = Y_{k} + \sigma Z_{k},
\end{equation}
where the $Z_k$ are unobserved, identically distributed but not necessarily independent, with common distribution function $\Psi$. The variables $Y_k$ represent an underlying signal component, assumed to satisfy a suitable, possibly weak, notion of sparsity or low variation relative to the noise, and $\sigma$ is an unknown positive scale parameter. We write $\bm{X}$, $\bm{Y}$ and $\bm{Z}$ for the corresponding vectors in $\R^n$.\\

\noindent We study the problem of estimating $\sigma$ from the observed vector $\bm X$ alone. The signal $\bm Y$ and the noise realization $\bm Z$ are unobserved. The estimator uses the reference marginal law $\Psi$ and structural control of the signal component $\bm Y$, but is otherwise agnostic to the joint dependence structure of the noise sequence.\\

\subsection*{Scale estimation in signal and stochastic models}
\noindent In models of the form \eqref{eq:X=Y+Z}, an estimate of $\sigma$ is often needed before deciding which fluctuations should be treated as background noise and which ones should be attributed to the signal. This problem appears naturally in several settings. We mention a few of them, without aiming for exhaustiveness.\\

\noindent The first one is wavelet denoizing, where the noise level fixes the threshold used to separate signal from background fluctuations, as in the universal Donoho-Johnstone threshold \cite{wavelet_shrinkage}.  Another setting is speech enhancement and adaptive beamforming, where one uses filtering to reduce background noise while preserving the signal of interest. In this setting, scale estimation enters noise-tracking and robust filtering procedures, including methods related to Improved Minima Controlled Recursive Averaging and Generalized Sidelobe Canceller architectures, where one also needs to limit the leakage of active signal components into the estimated background~\cite{VOROBYOV2014503}. There is already substantial work on robust versions of these procedures. For instance, recent distributionally robust beamforming methods handle target mismatch through worst-case SINR criteria over uncertainty sets for the interference-plus-noise covariance matrix~\cite{MR4556035,MR4931381}. These methods are effective in that setting, but they remain covariance-based, which can be limiting under very heavy-tailed bursts or infinite-variation interference, where covariance estimates may become unstable. The present work instead targets additive contamination under possible dependence, heavy tails, or weak sparsity.\\

\noindent A similar issue appears in high-frequency econometrics and stochastic analysis. For high-frequency observations of an It\^o semimartingale, estimating the continuous volatility component in the presence of jumps often relies on truncation-based thresholding or multipower variations~\cite{MR3235234}. These methods are very effective, but their implementation usually requires some preliminary calibration, such as an initial volatility estimate, a threshold adapted to the scale, or information on the jump activity \cite{MR3224283,MR4140022,MR4774175}. This is precisely the type of situation in which one wants to isolate a scalar volatility parameter from additional jump or signal components.\\

\subsection*{Classical robust scale estimators}
\noindent There is already a good collection of standard robust alternatives for scale estimation. Among them, it is worth mentioning the median absolute deviation (MAD), rescaled by a consistency factor~\cite{hampel1986robust}, and the estimators of Rousseeuw and Croux~\cite{MR1245360}.\\

\noindent These classical estimators are useful in a variety of scenarios, but they are not fully agnostic to the structure of the data in the sense needed here. In particular, their standard interpretation is tied to settings where the noise behaves essentially as an independent sample, and where the contamination can be handled through central order information. This interpretation becomes restrictive, and may lead to unstable scale calibration, when the noise sequence $\bm Z$ has significant dependence, for instance in long-range dependent or fractional regimes, or when the reference law $\Psi$ is heavy-tailed or multimodal. We instead allow dependence in the noise sequence $\bm Z$ and only assume that the signal $\bm Y$ satisfies a weak sparsity or low-variation condition. In this regime, an estimator based only on a few central quantiles, or on a preliminary transformation of the data such as a wavelet transform used to reduce dependence, may miss part of the information contained in the sample. Such transformations can also spread localized spikes or discontinuities across several coefficients, thereby weakening the sparsity structure of $\bm Y$ and introducing additional oscillations.\\

\subsection*{Main contribution}
\noindent The main contribution of this paper is to develop and analyze a purely spatial, permutation-based scale estimator based on the ordered observations. We sort the entries of $\bm X$ and align them with a reference profile determined by the noise distribution $\Psi$. In this way, the estimator avoids preliminary temporal filtering, which may spread localized signal components, and does not rely on covariance-type quantities, predefined uncertainty sets, or iterative matrix relaxations, which may be unstable or difficult to specify under dependence, heavy tails, or additive signal contamination.\\

\noindent This viewpoint leads to a family of $L^r$-type estimators based on weighted order statistics. We compare the sorted vector $\bm X^\uparrow$ with a proxy profile $\bm\psi$, constructed either from the quantiles of $\Psi$ or from an independent sample with law $\Psi$, through weighted $L^r$ losses, with the quadratic case taking an inner-product form. The weights, chosen according to the reference profile, allow us to emphasize the central part of the ordered sample, where the noise bulk is expected to remain visible, and downweight the extremes, where sparse signal perturbations are expected to concentrate after sorting. The analysis then exploits the stability of the sorting map in $\ell^p$ norms to obtain non-asymptotic bounds that separate the contribution of the noise profile from the variation of the signal.

\subsection*{Organization of the paper}
\noindent The rest of the paper is organized as follows. In Section~\ref{sec:weighted_estimators}, we introduce the weighted order-statistic estimators, describe the deterministic and random proxy constructions, and prove the main non-asymptotic bounds for $r\in\{1,2\}$. We also discuss robustness and breakdown properties of the weighted median estimator. Section~\ref{sec:weighted-proxies} studies the coefficients that appear in the main deterministic error bound. Section~\ref{sec:proxy_control} studies the distance between sorted noise samples and their proxies, including both deterministic and random proxy constructions, in the i.i.d.\ and correlated settings. Finally, Section~\ref{sec:sde_applications} applies the framework to high-frequency observations of stochastic processes and to scalar volatility estimation.

\section{Weighted order-statistic estimators}
\label{sec:weighted_estimators}

\noindent Throughout this section, we write $\br{n}=[1,n]\cap\mathbb{Z}$. Our estimators of the scale parameter $\sigma$ are based on the ordered representation of the sample. Before continuing further with the formal construction, we first describe the intuition behind it.

\subsection{Heuristic considerations}
\label{subsec:heurcons}

Let us denote by 
\[
\bm{X}^\ua=\{X^\ua_k;\ k\in\br{n}\}
\]
the non-decreasing rearrangement of $\bm X$. The main idea behind our estimate consists of working  with $\bm X^\ua$ instead of $\bm X$, based on the heuristic that after sorting, the leading term is the ordered noise $\sigma\bm Z^\ua$, while the signal contributes a perturbation whose size can still be controlled. More precisely, if $\|\cdot\|_p$ denotes the usual $p$-norm in $\R^n$, then we can write
\[
\bm X^\ua
=
\sigma\bm Z^\ua+\widetilde{\bm Y},
\]
for some vector $\widetilde{\bm Y}$ satisfying
\[
\|\widetilde{\bm Y}\|_p
\leq
\|\bm Y\|_p,
\]
as explained in more detail in Section \ref{se:NAEBFE}. Thus, after sorting, the signal contribution remains controlled by the original size of $\bm Y$ measured in a suitable metric. In the applications considered in this manuscript, the relevant normalization makes this term vanish asymptotically. The previous discussion suggests the heuristic approximation
\[
\bm X^\ua\approx
\sigma\bm Z^\ua.
\]
It remains to approximate the unobserved ordered noise vector $\bm Z^\ua$. If $\Psi^{-1}$ denotes the generalized quantile function of the reference law $\Psi$, a natural deterministic proxy is 
\[
\psi_k=\Psi^{-1}\left(k/(n+1)\right).
\]
This choice is motivated by the fact that order statistics fluctuate around their corresponding quantiles; the estimates below make this approximation quantitative in the regimes considered here. More generally, the proxy vector $\bm\psi$ will be either deterministic, as above, or generated from an independent sample with law $\Psi$. This leads to the approximation
\[
\bm X^\ua
\approx
\sigma\bm\psi.
\]
The estimator is then obtained by choosing the value of $\sigma$ that gives the best fit between the ordered observations $\bm X^\ua$ and the scaled proxy $\sigma\bm\psi$. The weights introduced below allow us to decide which parts of the ordered sample should have more influence in this comparison, and in particular to reduce the effect of the extremes when the signal contamination is expected to be concentrated there.

\subsection{Deterministic and random target proxies}
\label{subsec:proxies}
In the previous discussion, $\bm\psi$ was introduced as a proxy for the unobserved sorted noise vector $\bm Z^\ua$, serving as an ordered reference profile for the rescaled observations. Beyond this role, no specific structure is imposed. This leaves enough flexibility to adapt the proxy to the features of the underlying noise. Deterministic quantiles are natural when the reference law $\Psi$ is explicit and easy to handle, but other choices are also possible: in some regimes it is useful to replace the deterministic profile by a random one, generated independently from a suitable reference distribution. This alternative can be particularly convenient for heavy-tailed or less tractable laws. We distinguish between these two possibilities below.

\subsubsection{Deterministic proxies}
\label{subsec:deterministic-proxies}

The deterministic regime corresponds to choosing the proxy directly from the theoretical quantiles of the reference law $\Psi$. We use the generalized inverse
\[
\Psi^{-1}(u)\coloneqq \inf\{x\in\R:\Psi(x)\ge u\},
\]
for $u\in(0,1)$, and set
\begin{align}\label{eq:proxies_deterministic}	
\bm\psi=\bm\psi^\Det,
\qquad
\psi_k^\Det\coloneqq \Psi^{-1}\Pigl(\frac{k}{n+1}\Pigr),
\end{align}

for $1\le k\le n$. This choice is most useful when these quantiles can be evaluated explicitly or efficiently.

\subsubsection{Random proxies}
\label{subsec:random-proxies}

In some regimes, the deterministic quantile profile is not well suited to the assumptions required by our estimates. This may happen, for instance, when the quantile function of the reference law has poor integrability properties, so that the convergence guarantees for the resulting estimator (to be presented later in the paper) are no longer available. A second difficulty is that the quantile function $\Psi^{-1}$ may be mathematically intractable or expensive to evaluate; this can occur, for example, for stable laws or for multimodal distributions such as mixtures. In these situations, we replace the deterministic quantile profile by an auxiliary sample drawn directly from the reference law.\\

\noindent More precisely, let $\bm\xi=(\xi_1,\dots,\xi_n)$ be an auxiliary i.i.d. sample with common law $\Psi$, generated independently of $\bm X$. We define the random proxy by
\begin{align}\label{eq:proxies_random}
\bm\psi=\bm\psi^\rnd\coloneqq \bm\xi^\ua,	
\end{align}

where $\bm\xi^\ua$ denotes the increasing rearrangement of $\bm\xi$.  Using the usual quantile representation, we may write
\[
\bm\xi=\Psi^{-1}(\bm U),
\qquad
\bm\xi^\ua=\Psi^{-1}(\bm U^\ua),
\]
where $\bm U=(U_1,\dots,U_n)$ is an i.i.d. sample of standard uniform random variables. The advantage of this representation is that the randomness is transferred to the uniform order statistics. In particular, for each $1\le k\le n$, $U_k^\ua$ has distribution $\mathrm{Beta}(k,n-k+1)$ and is concentrated around its mean $k/(n+1)$.\footnote{We follow the standard convention that, for a map $\phi:\R\to\R$ and a vector $\bm x=\{x_k;\, k\in\br n\}\in\R^n$, $\phi(\bm x)$ denotes the vector $\{\phi(x_k);\, k\in\br n\}$. Thus, for instance, $|\bm x|^r=\{|x_k|^r;\, k\in\br n\}$ for $r>0$.}\\

\noindent As will be seen below, the estimators based on $\bm\psi^\rnd$ may contain fluctuation terms with twice the asymptotic variance of their deterministic counterparts. This extra variability can be reduced by averaging over several independent auxiliary samples. Concretely, one may generate $m$ independent copies of $\bm\xi$, compute the corresponding estimator for each copy, and average the resulting values; see Section~\ref{subsec:sim_mechanics}. This amounts to a Monte Carlo approximation of the conditional expectation, given $\bm X$, of the estimator based on $\bm\psi^\rnd$, at the cost of the corresponding increase in computation.

\subsection{Weighted \texorpdfstring{$L^r$}{Lr}-loss function and estimators}
\label{subsec:weighted-Lr-loss}
The proxy $\bm\psi$ gives the ordered reference profile against which $\bm X^\ua$ will be compared. For the results below, this comparison is refined by adding weights: the boundary coordinates are downweighted, since this is where the sorted signal contribution may be largest and where the quantile profile may have poorer regularity. More precisely, we introduce a weight vector
\[
\bm w=(w_1,\dots,w_n)
\]
with $w_i\geq 0$ and use it to define a weighted loss between $\bm X^\ua$ and $s\bm\psi$, for each candidate scale $s$ in $\R$. The decay of $\bm w$ near the boundary is chosen according to the tail behavior of the reference law $\Psi$. Given $r\ge 1$, we define the weighted $L^r$-loss by
\begin{equation}
\label{eq:loss_general}
\ell^r_{\bm w}(s)
\coloneqq
\sum_{k=1}^n w_k |X^\ua_k-s\psi_k|^r,
\end{equation}
for $s$ in $\R$. The map $s\mapsto \ell^r_{\bm w}(s)$ is convex, so the estimator is obtained from a one-dimensional convex minimization problem. We choose
\[
\Sigma_r\in\argmin_{s\in\R}\ell^r_{\bm w}(s),
\]
interpreting $\Sigma_r$ as the unique minimizer whenever uniqueness holds. For $r>1$, uniqueness holds as soon as a non-degenerate condition on $\langle \bm w,|\bm\psi|^r\rangle$ is imposed. In the sequel, we denote by $\bm a\pdot\bm b$ the vector in $\R^n$ with coordinates $a_kb_k$. For the case $r$ equal to two, the loss is a one-dimensional weighted least-squares criterion, so that
\[
\Sigma_2
=
\frac{\langle \bm w,\bm\psi\pdot \bm X^\ua\rangle}
{\langle \bm w,\bm\psi^2\rangle}.
\]
We write $\Sigma_r^\Det$ and $\Sigma_r^\rnd$ for the estimators obtained from $\bm\psi^\Det$ and $\bm\psi^\rnd$, respectively. For notational convenience, we do not make explicit the dependence on the weight vector $\bm w$. \\

\noindent In the case $r=1$, the minimization reduces to a weighted median and can be carried out by the usual Powersort (or any other sorting algorithm such as Timsort) procedure, which can be computed in $\Oh(n\log n)$ computational time with $\Oh(n)$ memory. The accuracy of the resulting estimators depends on the size and sparsity of $\bm Y$, on the reference law $\Psi$, and on the dependence structure of $\bm Z$. The relation between $\bm Y$ and $\bm Z$, however, will play no direct role in the concentration estimates below.

\subsection{Non-asymptotic error bounds for the estimators}\label{se:NAEBFE}
We next present deterministic error bounds for the estimators introduced above, focusing on the cases $r=1,2$. These correspond, respectively, to the weighted-median and weighted least-squares estimators.

\begin{thm}
\label{thm:CI_r12}
Let $p\ge 1$ and let $1\leq q\leq \infty$ be its conjugate exponent. For any proxy $\bm\psi\in\R^n$, any weight $\bm w\in[0,\infty)^n$ and any $r=1,2$ such that $\bm\psi\pdot\bm w\ne\bm0$, we have
\begin{align}
\label{eq:bound_r12}
|\Sigma_r - \sigma| 
&\le 2^{2-r}\frac{\|\bm w\pdot |\bm\psi|^{r-1}\|_q}{\langle\bm w,|\bm\psi|^r\rangle} \|\bm Y\|_p 
+ 2^{2-r} \sigma\bigg(\frac{\langle\bm w,|\bm Z^\ua - \bm\psi|^r\rangle}{\langle\bm w,|\bm\psi|^r\rangle}\bigg)^{1/r}.
\end{align}
\end{thm}

\noindent The bound \eqref{eq:bound_r12} is deterministic in the sense that it does not rely on any distributional assumption and holds once the vectors $\bm Y$, $\bm Z^\ua$, $\bm\psi$ and $\bm w$ are fixed. Consequently, probabilistic estimates for $\Sigma_r$ reduce to controlling two deterministic quantities: the signal contribution from $\bm Y$ and the proxy error between $\bm Z^\ua$ and $\bm\psi$.\\

\noindent The proof of Theorem~\ref{thm:CI_r12} requires the use of a basic stability property of increasing rearrangements. This result    can be found for instance in \cite[Lem.~4.2]{MR4028181} and  states that sorting two vectors in the same order cannot increase their $L^p(\Omega)$-distance. More precisely, for any $p\ge 1$ and any $\bm x,\bm y\in\R^n$,
\begin{equation}
\label{eq:non-exp-Lp}
\|\bm{x}^\ua-\bm{y}^\ua\|_p
\le \|\bm{x}-\bm{y}\|_p.
\end{equation}

\begin{proof}[Proof of Theorem~\ref{thm:CI_r12}]
The proof is split into the two cases $r=1$ and $r=2$.\\

\noindent We first handle the case $r=1$. By definition of $\Sigma_1$, note that
\[
\langle \bm{w}, |\bm{X}^\ua - \Sigma_1 \bm\psi|\rangle
\le \langle \bm{w}, |\bm{X}^\ua - \sigma \bm\psi|\rangle.
\]
Thus, the triangle inequality yields
\begin{align*}
|\Sigma_1-\sigma|
\langle \bm{w}, |\bm\psi|\rangle
&\le
\langle \bm{w},|\bm{X}^\ua - \Sigma_1 \bm\psi|\rangle
+\langle \bm{w}, |\bm{X}^\ua - \sigma \bm\psi|\rangle
\le 2\langle \bm{w}, |\bm{X}^\ua - \sigma \bm\psi|\rangle.	
\end{align*}

\noindent Thus, the triangle inequality gives
\begin{equation*}
|\Sigma_1-\sigma|
\langle \bm{w}, |\bm\psi|\rangle
\le 2 \langle\bm{w},|\bm{X}^\ua - \sigma \bm{Z}^\ua|\rangle + 2 \sigma\langle\bm{w},|\bm{Z}^\ua - \bm\psi|\rangle.
\end{equation*}
To bound the first term, we simply use H\"older's inequality and~\eqref{eq:non-exp-Lp}:
\[
\langle\bm{w},|\bm{X}^\ua - \sigma \bm{Z}^\ua|\rangle
\le \|\bm{w}\|_q\|\bm{X}^\ua - \sigma \bm{Z}^\ua\|_p
\le \|\bm{w}\|_q\|\bm{X} - \sigma \bm{Z}\|_p
=\|\bm{w}\|_q\|\bm{Y}\|_p.
\]

\noindent We now handle the case $r=2$ following a similar strategy. Note that
\begin{equation*}
(\Sigma_2 - \sigma) \langle\bm{w},\bm\psi^{2}\rangle 
= \langle\bm{w},\bm\psi\pdot\bm{X}^\ua\rangle - \sigma \langle\bm{w},\bm\psi^{2}\rangle 
= \langle\bm{w},\bm\psi\pdot(\bm{X}^\ua - \sigma \bm\psi)\rangle
= \langle\bm{w}\pdot\bm\psi,\bm{X}^\ua - \sigma \bm\psi\rangle.
\end{equation*}
Next, we add and subtract $\sigma \bm{Z}^\ua$ inside the parentheses, take absolute value and apply the triangle inequality to get
\begin{equation}
\label{eq:proof_r2_decoupled}
|\Sigma_2 - \sigma| \langle\bm{w},\bm\psi^{2}\rangle 
\le | \langle\bm{w}\pdot\bm\psi,\bm{X}^\ua - \sigma \bm{Z}^\ua\rangle| + \sigma |\langle\bm{w}\pdot\bm\psi,\bm{Z}^\ua - \bm\psi\rangle|. 
\end{equation}

\noindent We control the first term in \eqref{eq:proof_r2_decoupled} by applying H\"older's inequality together with \eqref{eq:non-exp-Lp}:
\[
| \langle\bm{w}\pdot\bm\psi,\bm{X}^\ua - \sigma \bm{Z}^\ua\rangle| 
\le \|\bm{w} \pdot \bm\psi\|_q \|\bm{X}^\ua - \sigma \bm{Z}^\ua\|_p
\le \|\bm{w} \pdot \bm\psi\|_q \|\bm{X} - \sigma \bm{Z}\|_p
=\|\bm{w} \pdot \bm\psi\|_q \|\bm{Y}\|_p.
\]
For the second term in~\eqref{eq:proof_r2_decoupled}, we apply the Cauchy-Schwarz inequality: 
\[
\big|\big\langle\sqrt{\bm{w}}\pdot\bm\psi,\sqrt{\bm{w}}\pdot(\bm{Z}^\ua - \bm\psi)\big\rangle\big|^2 
\le \langle\bm{w},\bm\psi^2\rangle \langle\bm{w},(\bm{Z}^\ua - \bm\psi)^2\rangle.
\]
Substituting both bounds into \eqref{eq:proof_r2_decoupled} and dividing by $\langle\bm{w},\bm\psi^2\rangle$ gives the bound for $\Sigma_2$.
\end{proof}

\begin{rem}
\label{rem:CI-r-not-12}
While similar concentration inequalities can be developed for the cases $r \ne 1,2$, these estimators will not be robust, unlike $r=1$, nor explicit, unlike $r=2$. Since our goal here is to develop the basic non-asymptotic theory in the two most relevant cases, we restrict the analysis to the weighted-median and weighted least-squares estimators. Extensions to other values of $r$ are possible, but they require a separate treatment and will not be addressed  in this work.
\end{rem}

\noindent Theorem~\ref{thm:CI_r12} decomposes the estimation error $|\Sigma_r-\sigma|$ into two terms: one coming from the signal component $\bm Y$, and one from the mismatch between the sorted noise $\bm Z^\ua$ and the proxy $\bm\psi$. Before continuing further, we would like to emphasize the different roles played by these two contributions.\\

\begin{itemize}[leftmargin=2em, topsep=0em]
\item[(I)] The first term, hereafter called ``signal contamination'', is governed by the contamination magnitude $\|\bm{Y}\|_p$ and represents the full contamination by the sparse signal. The magnitude of this contamination is primarily suppressed by the weighted mass $\langle\bm{w},|\bm\psi|^r\rangle$ of the proxies in the denominator. If the weights $\bm{w}$ are chosen to sufficiently penalize the boundaries, the full contamination term will vanish polynomially as $n\to\infty$. Corollary~\ref{cor:CI_r12} below presents a simplified version of Theorem~\ref{thm:CI_r12} with explicit asymptotic coefficients in the case where $\bm{w}$ is a continuous bounded function $\omega(\bm\psi)$ of the proxy~$\bm\psi$, for appropriate functions $\omega:\R\to(0,\infty)$ described in Sections~\ref{sec:weighted-proxies} and~\ref{sec:proxy_control} below.\\

\item[(II)] The second term, hereafter called ``empirical-to-proxy discrepancy'', is governed by the weighted disagreement between $\bm{Z}^\ua$ and the proxy~$\bm\psi$ and is harder to control in full generality. The main objective of Section~\ref{sec:proxy_control} is to establish rigorous probabilistic control over this discrepancy under both, the deterministic and random proxy regimes, respectively, accommodating for complex dependence structures and heavy-tailed reference distributions $\Psi$ in both cases.\\
\end{itemize}

\noindent We now use the previous bounds to derive three asymptotic results for $\Sigma_r$. In the sequel, we will write $\zeta_n=\Oh_\p(\beta_n)$, for a positive deterministic sequence $ \beta_n $, if $\zeta_n/\beta_n$ is bounded in probability. Similarly, we write $\zeta_n=\oh_\p(\beta_n)$ if $\zeta_n/\beta_n$ converges to zero in probability  as $n$ tends to infinity.

\begin{cor}
\label{cor:CI_r12}
Fix $p\geq 1$, let $1\leq q\leq \infty$ be its conjugate exponent, and let $r=1,2$. Set $\bm{w}=\omega(\bm\psi)$, for some $\omega:\R\to\R_{+}$ continuous and bounded with
\[
\sup_{x\in\R}|x|^r\omega(x)<\infty
\qquad\text{and}\qquad
\E[|Z|^r\omega(Z)]>0,
\]
where $Z$ has distribution function $\Psi$.
If the proxy~$\bm\psi$ is constructed via either~\eqref{eq:proxies_deterministic} or~\eqref{eq:proxies_random}, then, 
\begin{align*}
|\Sigma_r - \sigma| 
&\le
4\frac{\|\bm Y\|_p}{n^{1/p}}
\bigg(
\frac{\|Z^{r-1}\omega(Z)\|_{L^q}}
{\|Z^r\omega(Z)\|_{L^1}}
+\oh_\p(1)
\bigg)\\
&+
4\frac{
\langle\bm w,|\bm Z^\ua-\bm\psi|^r\rangle^{1/r}
}{n^{1/r}}
\bigg(
\frac{\sigma}
{\|Z\omega(Z)^{1/r}\|_{L^r}}
+\oh_\p(1)
\bigg).
\end{align*}
\end{cor}

\begin{rem}
Besides requiring the coordinates of the noise vector $\bm Z$ to have marginal distribution $\Psi$, no further condition is imposed on the dependence structure among them.
\end{rem}

\begin{rem}
As a consequence of Corollary~\ref{cor:CI_r12}, the estimator $\Sigma_r$ is consistent provided that
\[
\|\bm{Y}\|_p^p=\oh_\p(n)
\qquad\text{and}\qquad
\langle\bm{w},|\bm{Z}^\ua-\bm\psi|^r\rangle=\oh_\p(n).
\]
Moreover, if
\[
\|\bm{Y}\|_p^p=\Oh_{\p}(n^\delta)
\qquad\text{and}\qquad
\langle\bm{w},|\bm{Z}^\ua-\bm\psi|^r\rangle=\Oh_{\p}(n^\gamma),
\]
then
\[
|\Sigma_r-\sigma|
=
\Oh_\p\bigl(n^{-\min\{(1-\gamma)/r,(1-\delta)/p\}}\bigr).
\]
\end{rem}

\begin{proof}[Proof of Corollary \ref{cor:CI_r12}]
The result is a consequence of Theorem~\ref{thm:CI_r12} and the forthcoming Corollary~\ref{cor:weighted-proxies}. When $p=1$ and hence $q=\infty$, we additionally use
\[
\|\bm w\pdot|\bm\psi|^{r-1}\|_\infty
\longrightarrow
\|Z^{r-1}\omega(Z)\|_{L^\infty},
\]
deterministically for $\bm\psi^\Det$ and in probability for $\bm\psi^\rnd$,
which follows directly from the deterministic quantile approximation and
from the convergence of sample maxima, respectively.
\end{proof}

\noindent We next specialize the discussion to the case where the coordinates of $\bm Z$ are independent and identically distributed. In this setting, the discrepancy terms appearing in the above corollary can be controlled explicitly. More precisely,  Corollary~\ref{cor:CI_r12} together with Theorems~\ref{thm:weak_lim_proxy_det} and~\ref{thm:weak_lim_proxy_rand} below yields the following consequence.

\begin{cor}
\label{cor:sigma-iid}
Assume the hypotheses of Corollary~\ref{cor:CI_r12}. Suppose moreover that $\bm Z$ has i.i.d. coordinates and that the corresponding assumptions of Theorem~\ref{thm:weak_lim_proxy_det} or Theorem~\ref{thm:weak_lim_proxy_rand} hold, according to the chosen proxy. Then,
\[
|\Sigma_r - \sigma|
=
\Oh_{\p}\big( n^{-1/p} \|\bm{Y}\|_p + n^{-1/2}\big),
\]
as $n$ tends to infinity.
\end{cor}
\begin{rem}
As a consequence of Corollary \ref{cor:sigma-iid},  we have that $\Sigma_r$ is consistent whenever $\|\bm{Y}\|_p^p = \oh_\p(n)$. Moreover, if $\|\bm{Y}\|_p^p =\Oh_\p(n^\delta)$, then 
$$|\Sigma_r - \sigma| = \Oh_{\p}(n^{-\min\{(1-\delta)/p, 1/2\}}).$$	
\end{rem}

\noindent The setting in which $\bm Z$ has independent coordinates provides the cleanest illustration of how Corollary~\ref{cor:CI_r12} can be used. Nevertheless, the same approach applies whenever the discrepancy between the ordered noise and the proxy can be controlled. In particular, in  a dependent setting, this control is obtained from bounds on the correlations of indicator functions of half-lines. The general case is treated in the forthcoming Proposition~\ref{prop:expected_dependent_discrepancy} below, while we now present the simplified result for stationary Gaussian processes.\\

\noindent Before going further, we make explicit the form of dependence that we allow for the vector $\bm Z$. We start from a centered, unit-variance Gaussian process $\bm W$, whose dependence is encoded by its covariance function. Applying the standard Gaussian distribution function $\Phi$ coordinatewise yields a dependent sequence with uniform marginals. The  noise distribution to be considered in the sequel is then assumed to be obtained by applying the quantile function $\Psi^{-1}$ to the aforementioned procedure, where $\Psi$ denotes the distribution function of the noise. Namely, we assume that
\begin{align}\label{eq:Zndef}
\bm{Z}=\Psi^{-1}\circ\Phi(\bm{W}).
\end{align}
Thus, the dependence of $\bm Z$ is inherited from the Gaussian process $\bm W$, and the quantity
\begin{align}\label{eq:Dndef}
D_n=\sum_{i=1}^n\sum_{j=1}^n |\Cov[W_i,W_j]|
\end{align}
provides a measurement of the dependence entering the bound. The previous construction essentially means that the system follows a Gaussian copula with arbitrary marginal law $\Psi$. The adjustment of Corollary \ref{cor:sigma-iid} to the dependent case then reads as follows

\begin{cor}
\label{cor:sigma-dependent-noise}
Let $p\geq 1$ and $r= 1, 2$. Assume $\bm{Z}$ is given by \eqref{eq:Zndef} and that
\[
\int_{\R}
|x|^{r-1}
\sqrt{\Psi(x)(1-\Psi(x))}\,\dd x
<\infty.
\]
Suppose $\bm\psi$ is generated via either~\eqref{eq:proxies_deterministic} or~\eqref{eq:proxies_random}. Let $\bm{w} = \omega(\bm\psi)$ where $\omega:\R\to [0,1]$ is a continuous bounded function such that 
$$\sup_{x\in\R}\omega(x)|x|^r<\infty.$$
Assume moreover that
\[
\E[|Z|^r\omega(Z)]>0.
\]
Then, we have
\[
|\Sigma_r - \sigma| = \Oh_{\p}\big( n^{-1/p} \|\bm{Y}\|_p + n^{-1/r}D_n^{1/(2r)}\big),
\]
as $n$ tends to infinity, where $D_n$ is given by \eqref{eq:Dndef}. 
\end{cor}

\begin{rem}
As a consequence of Corollary \ref{cor:sigma-dependent-noise},  we have that $\Sigma_r$ is consistent if $\|\bm{Y}\|_p^p = \oh_\p(n)$ and $D_n=\oh(n^2)$. Moreover, if  $\|\bm{Y}\|_p^p =\Oh_\p(n^\delta)$ and $D_n=\Oh(n^\gamma)$, then 
$$|\Sigma_r - \sigma| = \Oh_{\p}(n^{-\min\{(1-\delta)/p,(1-\gamma/2)/r\}}).$$	
\end{rem}

\begin{proof}
The result follows from Corollaries~\ref{cor:CI_r12} and~\ref{cor:expected_dependent_discrepancy} and Remark~\ref{rem:expected_dependent_discrepancy}(a) below.
\end{proof}

\begin{rem}
\label{rem:fBm}
\noindent Corollary~\ref{cor:sigma-dependent-noise} reduces the control of the empirical-to-proxy discrepancy to the growth of the correlation sum associated with the underlying Gaussian process. As a concrete benchmark, we consider the case where the components of $\bm Z$ are given by the standardized increments of a fractional Brownian motion. Concretely, for a given fractional Brownian motion $B^{H}$ with $0<H<1$, we let 
$$Z_i:=B_{i}^H-B_{i-1}^H.$$ 
for $i\leq n$. The associated covariance function satisfies
\[
\rho_H(k)
\coloneqq
\frac{(k+1)^{2H}-2k^{2H}+(k-1)^{2H}}{2}.
\]
For $H\neq\tfrac12$,
\[
\rho_H(k)
\sim H(2H-1)k^{2H-2},
\]
as $k$ tends to infinity, while for $H=\tfrac12$ we have $\rho_H(k)=0$ for every $k\geq1$. Consequently,
\[
D_n
=
n+\sum_{k=1}^{n-1}2(n-k)|\rho_H(k)|
=
\Oh\big(n^{\max\{1,2H\}}\big).
\]
Thus, under the assumptions of Corollary~\ref{cor:sigma-dependent-noise}, we have
\[
|\Sigma_r-\sigma|
=
\Oh_{\p}\big(
n^{-1/p}\|\bm Y\|_p
+
n^{-\min\{1-H,1/2\}/r}
\big).
\]
\end{rem}

\noindent We end this section by taking a closer look at the weighted median estimator $\Sigma_1$, with the aim of making explicit the robustness properties that are specific to the case $r$ equal to one in our analysis.

\subsection{Robustness and Breakdown Point Analysis for \texorpdfstring{$r=1$}{the weighted median}}

\noindent In order to make the robustness properties of $\Sigma_1$ precise, we consider situations where a subset of the observed data may be arbitrarily corrupted, either due to measurement errors \cite{fuller1987measurement}, outliers \cite{hampel1986robust}, or adversarial perturbations \cite{MR161415}. Even a small fraction of such contamination can significantly distort classical estimators, potentially driving them to arbitrarily large values, so that they no longer provide reliable information about the underlying parameter. This raises the fundamental question of how much contamination an estimator can tolerate before it completely loses reliability.\\

\noindent A natural way to formalize this notion is through the concept of the breakdown point, which quantifies the largest proportion of arbitrary corruption that a dataset can sustain while keeping the estimator bounded, as defined below. Consider a sample $\bm{\theta} = \{\theta_k;\, k\in\br{n}\}$, and let $T_n:\R^n\to\R$ be a given statistic. For $m\le n$, define
\[
\Theta_m(\bm\theta)
\coloneqq
\big\{
\tilde{\bm\theta}\in\R^n:
\bigl|\{k\in\br{n}:\tilde{\theta}_k\neq\theta_k\}\bigr|\le m
\big\}.
\]
That is, $\Theta_m(\bm\theta)$ is the collection of samples obtained by replacing at most $m$ entries of $\bm\theta$ with arbitrary real values. To determine whether this level of contamination breaks the statistic, we consider the largest value that $|T_n|$ can attain over $\Theta_m(\bm\theta)$. This leads to the following definition.

\begin{definition}
The finite-sample breakdown point of the statistic $T_n$ is defined as
\[
\ve(T_n)
\coloneqq
\frac{1}{n}
\inf_{\bm\theta\in\R^n}\ve(T_n,\bm\theta),
\quad\text{where}\quad
\ve(T_n,\bm\theta)
\coloneqq\max\Pigl\{
m\in\{0,\ldots,n\} : \sup_{\tilde{\bm\theta}\in\Theta_m(\bm\theta)} |T_n(\tilde{\bm\theta})| < \infty
\Pigr\}.
\]
\end{definition}
\noindent When $T_n$ is taken to be the usual median of a real-valued sample, its breakdown point is approximately one half, meaning that it remains stable under arbitrary contamination of up to nearly half of the observations. A comprehensive treatment of this topic can be found in \cite{MR161415}. This notion extends naturally to the weighted median setting. Given weights $\bm{\gamma}\in\R_{+}^n$, consider any weighted median $\med_{\bm \gamma}(\bm\theta)$, that is, any minimizer of the $L^1$ loss function\footnote{As explained before, $|\bm{x}|^r=\{|x_k|^r;\, k\in\br{n}\}$ for $\bm{x}=\{x_k;\, k\in\br{n}\}\in\R^n$ and $r>0$.} 
\[
\ell^1_{\bm{\gamma},\bm\theta}(\vartheta) = \langle \bm{\gamma}, |\bm{\theta}-\vartheta|\rangle.
\]
Equivalently, $\med_{\bm \gamma}(\bm\theta)$ may be chosen as any point in the interval
\[
\Pigl[\sup\Pigl\{\vartheta:\frac{\dd}{\dd\vartheta^-}\ell^1_{\bm{\gamma},\bm\theta}(\vartheta)<0\Pigr\},\,\inf\Pigl\{\vartheta:\frac{\dd}{\dd\vartheta^+}\ell^1_{\bm{\gamma},\bm\theta}(\vartheta)>0\Pigr\}\Pigr].
\]
Typically, one considers the middle point of the interval above. It is known that weighted medians are robust and that if the weights of the corrupted indices accumulate less than half of the total mass, then the weighted median remains bounded (see, e.g.,~\cite[Prop.~2]{MR2412551}). In our context, however, this result is not directly applicable, as our estimator operates with the ordered values of the data. In this case, a single perturbation to the unordered sample can result in all the ordered samples to be considered as perturbed, though the estimator remains bounded. For this reason, we adapt the methods of~\cite[Prop.~2]{MR2412551} to our context.

In the $r=1$ regime, the loss can be written as
\[
\ell^1_{\bm w}(s)
=
\sum_{k=1}^n w_k|\psi_k|
\bigg|
\frac{X_k^\ua}{\psi_k}-s
\bigg|,
\]
where the terms with $\psi_k=0$ and $w_k|\psi_k|=0$ do not affect the minimization. Thus, $\Sigma_1$ is a weighted median of the values $X_k^\ua/\psi_k$, with weights
\[
\gamma_k=w_k|\psi_k|.
\]
The next lemma records the robustness of this order-statistic estimator, showing that arbitrary contamination of up to $m$ elements in the unordered sample $X$ cannot drive $\Sigma_1$ outside a bounded range as long as the worst-case weight mass of the affected order statistics remains below one half of the total weight.

\begin{lemma}
\label{lem:weighted_breakdown}
Let $\bm{X} \in \R^n$ and $\widetilde{\bm{X}} \in \Theta_m(\bm{X})$. Define the quantities 
\begin{gather*}
\un{X}\coloneqq\min_{i\in\br{n}}X_i,
\quad
\ov{X}\coloneqq\max_{i\in\br{n}}X_i,
\quad
M_-\coloneqq
\min_{i,k\in\br{n}, \psi_k\ne 0}\frac{X_i}{\psi_k}
\quad\text{and}\quad
M_+\coloneqq
\max_{i,k\in\br{n}, \psi_k\ne 0}\frac{X_i}{\psi_k}.
\end{gather*}
Set $\gamma_k \coloneqq w_k |\psi_k|$ for $k \in \br{n}$ and let $C \subset [n]$ denote the set of indices $k$ for which $\widetilde X_k^\ua$ lies outside the range $[\un{X}, \ov{X}]$. If $\sum_{k \in C} \gamma_k < \frac{1}{2} \|\bm\gamma\|_1$, then $\Sigma_1(\widetilde{\bm{X}})\in[M_-,M_+]$. Consequently,
\begin{equation}
\label{eq:vesigmaone}
\ve(\Sigma_1)
\geq
\frac{1}{n}
\max\bigg\{
m\in\{0,\ldots,n\}\ ; 
\sum_{k=n-m+1}^n\gamma^\ua_k
<
\frac{1}{2}\|\bm\gamma\|_1
\bigg\},
\end{equation}
\end{lemma}

\begin{proof}
Since $\bm{X} \in \R^n$, $\widetilde{\bm{X}} \in \Theta_m(\bm{X})$, at least $n-m$ entries of $\widetilde{\bm{X}}$ are identical to those of $\bm{X}$ and thus lie in $[\un{X},\ov{X}]$. Thus, the same is true of $\widetilde{\bm{X}}^\uparrow$, so $|C|\le m$. For any $k \notin C$, 
\[
M_-
\le\widetilde{\bm{X}}^\uparrow_k/\psi_k 
\le M_+.
\]
Note that $M_-$ and $M_+$ depend solely on the uncorrupted sample $\bm X$ and the proxy $\bm\psi$. For any $\vartheta > M_+$, the right derivative of the loss function $\ell_{\bm w}^1(\vartheta)$ satisfies
\[
\frac{\dd}{\dd\vartheta} \ell_{\bm w}^1(\vartheta) 
= \sum_{k \notin C} \gamma_k + \sum_{k \in C} \gamma_k \sgn\big(\vartheta - \widetilde{X}^\ua_k/\psi_k\big) \ge \|\bm\gamma\|_1 - 2 \sum_{k \in C} \gamma_k > 0.
\]
Similarly, the left derivative is strictly negative for any $\vartheta < M_-$. Thus, any minimizer $\Sigma_1(\widetilde{\bm{X}})$ must lie in $[M_-, M_+]$, giving the first claim.   Since $|C|\leq m$, we have
\[
\sum_{k\in C}\gamma_k
\leq
\sum_{k=n-m+1}^n\gamma_k^\ua.
\]
Therefore, if the sum of the $m$ largest weights is strictly smaller
than one half of the total weight, then the estimator remains bounded
under every replacement of at most $m$ original observations. Taking
the largest such $m$ gives the lower bound in~\eqref{eq:vesigmaone}.
\end{proof}

We now derive an asymptotic version of the lower bound
in~\eqref{eq:vesigmaone} when the weights are generated from the proxy
through $w_k=\omega(\psi_k)$, for some bounded function $\omega$.

\begin{prop}
\label{prop:breakdown}
Suppose $\bm{w}=\omega(\bm\psi)$ for some bounded function
$\omega:\R\to\R_+$. Let $W_n$ denote the empirical distribution
function of the weights
\[
\gamma_k=\omega(\psi_k)|\psi_k|.
\]
Let $\zeta$ be a random variable distributed according to $\Psi$ and
define
\[
W(x)\coloneqq
\p\big[\omega(\zeta)|\zeta|\le x\big],
\]
for $x\ge 0$. Assume that
\[
0<
\int_0^1W^{-1}(u)\dd u
<\infty
\]
and that
\begin{equation}
\label{eq:weight-quantile-convergence}
\lim_n\int_0^1
\big|W_n^{-1}(u)-W^{-1}(u)\big|
\dd u
=0.
\end{equation}
Then,
\begin{align}
\label{eq:breakdown-limit}
\liminf_{n\to\infty}\ve(\Sigma_1)
\geq
1-
\inf\bigg\{
p\in(0,1):
\int_p^1 W^{-1}(u)\dd u
<
\frac{1}{2}\int_0^1W^{-1}(u)\dd u
\bigg\}.
\end{align}
\end{prop}

\begin{proof}
By Lemma~\ref{lem:weighted_breakdown}, if the sum of the $m$ largest
weights is strictly smaller than one half of the total weight mass, then
$\Sigma_1$ remains bounded under every contamination of at most $m$
observations. We then use this sufficient condition to obtain the asymptotic
lower bound. Fix $p\in(0,1)$ such that
\[
\int_p^1 W^{-1}(u)\dd u
<
\frac{1}{2}\int_0^1W^{-1}(u)\dd u,
\]
and set $m_n\coloneqq\lfloor n(1-p)\rfloor$. Since $W_n^{-1}$ is the empirical quantile function of the weights
$\bm\gamma$, we have
\[
\frac{1}{n}
\sum_{k=n-m_n+1}^n\gamma_k^\ua
=
\int_{1-m_n/n}^1W_n^{-1}(u)\dd u
\quad\quad\quad
and
\quad\quad\quad
\frac{1}{n}\|\bm\gamma\|_1
=
\int_0^1W_n^{-1}(u)\dd u.
\]
Using \eqref{eq:weight-quantile-convergence}, as well as the fact that 
$1-m_n/n$ converges to $p$, it follows that 
\[
\lim_{n\to\infty}\frac{1}{n}
\sum_{k=n-m_n+1}^n\gamma_k^\ua
=
\int_p^1W^{-1}(u)\dd u
<\frac{1}{2}\int_0^1W^{-1}(u)\dd u
=\lim_{n\to\infty}\frac{1}{2n}\|\bm\gamma\|_1.
\]
Consequently, for all sufficiently large $n$,
\[
\sum_{k=n-m_n+1}^n\gamma_k^\ua
<
\frac{1}{2}\|\bm\gamma\|_1.
\]
The finite-sample lower bound~\eqref{eq:vesigmaone} therefore gives
\[
\ve(\Sigma_1)
\geq
\frac{m_n}{n}.
\]
Since $m_n/n$ converges to $1-p$, we obtain
\[
\liminf_{n\to\infty}\ve(\Sigma_1)
\geq
1-p.
\]
The conclusion follows by taking the supremum over all $p\in(0,1)$
satisfying the strict integral inequality above.
\end{proof}

\begin{rem}
\label{rem:Gaussian-breakdown}

In the case of the standard normal distribution $\Psi=\Phi$, the right-hand side of~\eqref{eq:breakdown-limit} can be approximated numerically. The following table reports Monte Carlo approximations of the asymptotic breakdown point for several choices of the weight function $\omega$.

\begin{table}[ht]
\centering
\begingroup
\renewcommand{\arraystretch}{1.35}
\setlength{\tabcolsep}{9pt}
\begin{tabular}{|c|c|c|c|c|}
\hline
Weight function & $\omega_1(x)=1$ & $\omega_2(x)=\frac{1}{1+|x|}$ & $\omega_3(x)=\frac{1}{.01+|x|}$ & $\omega_4(x)=\exp(-x^2/2)$ \\
\hline
Breakdown point & 23.9\% & 32.6\% & 48.6\% & 34.3\% \\
\hline
\end{tabular}
\endgroup
\vspace{4pt}
\caption{Asymptotic breakdown point for the standard normal law $\Psi=\Phi$.}
\label{tab:Gaussian-breakdown}
\end{table}

\noindent It is worth noting that the breakdown point can be made arbitrarily close to the $50\%$ benchmark of the usual median by taking $\omega$ close to $x\mapsto 1/|x|$. This limit, however, is unstable: the weights then concentrate around the central quantiles, where $|\psi_k|$ is close to zero, giving those observations a disproportionate influence on the estimator. One possible way to avoid this degeneracy would be to consider regularized versions of the weight function, possibly depending on $n$, such as $x\mapsto 1/(n^{-1/2}+|x|)$, or data-dependent variants such as $x\mapsto 1/(\min_{k\in\br{n}}|X_k|+|x|)$. We leave such extensions for future work.
\end{rem}

\section{Median Absolute Deviation Estimator}
\label{sec:MAD}

\noindent In view of the robustness discussion in the previous section, it is natural to ask how our estimators compare with classical median-based scale estimators. The most immediate reference point in this direction is the median absolute deviation (MAD). The MAD was studied by Hampel and Huber~\cite{MR362657,MR606374}, with Hampel attributing the idea to Gauss~\cite{Gauss1816}. In the usual location-scale setting, it has a $50\%$ asymptotic breakdown point and has also been analysed under sparse contamination and dependent samples~\cite{MR395039,MR418179}.\\

\noindent The additive-noise model considered in this manuscript does not fall exactly within the classical location-scale setting for which the MAD is usually defined. Consequently, the standard MAD theory does not directly yield a scale estimator adapted to the present framework. The underlying median-based principle; however, can be reformulated naturally in our setting. We develop this reformulation explicitly and use the resulting estimator as a benchmark for the estimators introduced above. We then establish deterministic and probabilistic guarantees under assumptions of the same type as those considered in the previous sections. These bounds are of independent interest and also provide the basis for a fair comparison with our estimators.\\

\noindent In order to give the precise formulation of this adapted MAD estimator, we proceed as follows. Let $Z$ be a random variable with distribution function $\Psi$, and define  $m:=\Psi^{-1}(\tfrac{1}{2})$. We denote by $\Psi_*$ the distribution function of $|Z-m|$, namely
\[
\Psi_*(x)
\coloneqq
\Psi(m+x)-\Psi((m-x)-),
\]
for $x\ge 0$. The corresponding deterministic MAD estimator is defined by
\begin{equation}
\label{eq:MAD}
\Sigma_\MAD^\Det
\coloneqq
\frac{\med(|\bm{X}-\med(\bm{X})|)}
{\Psi_*^{-1}(\tfrac{1}{2})}.
\end{equation}
Here and throughout this section, $\med(\bm{\xi})$ denotes the empirical quantile $\iota[\bm\xi]^{-1}(\tfrac12)$.  Note that if $\Psi$ is symmetric about its median (so that $\Psi(\Psi^{-1}(\tfrac{1}{2})+x)=1-\Psi(\Psi^{-1}(\tfrac{1}{2})-x)$ for $x\in\R$), then 
\[
\Psi_*^{-1}\big(\tfrac{1}{2}\big)
= \Psi^{-1}\big(\tfrac{3}{4}\big)
- \Psi^{-1}\big(\tfrac{1}{2}\big).
\]
\noindent As with $\bm\psi^\Det$, it is often difficult to compute the quantiles necessary to compute the MAD, especially when $\Psi$ is not symmetric. However, if sampling algorithms for $\Psi$ are available, we may consider the randomized MAD estimator
\begin{equation}
\label{eq:MAD-rnd}
\Sigma_\MAD^\rnd
\coloneqq \frac{\med(|\bm{X}-\med(\bm{X})|)}{\med(|\bm{\xi}-\med(\bm{\xi})|)},
\end{equation}
where $\bm\xi$ is an auxiliary sample\footnote{The estimators $\Sigma^\rnd_r$ and $\Sigma^\rnd_\MAD$ both use auxiliary samples of size $n$, keeping the complexity of computing them within $\Oh(n\log n)$. However, other sizes are also reasonable when $n$ is moderately small.} of $n$ i.i.d. copies with law $\Psi$.\\

\noindent The randomized estimator \eqref{eq:MAD-rnd} differs from \eqref{eq:MAD} only through the normalizing factor: the deterministic constant $\Psi_*^{-1}(\tfrac{1}{2})$ is replaced by the empirical MAD of an auxiliary i.i.d. sample with law $\Psi$. Since this auxiliary sample is drawn directly from the noise distribution and has no signal component, this empirical normalization typically converges to $\Psi_*^{-1}(\tfrac{1}{2})$ at rate $\Oh_\p(n^{-1/2})$. Indeed, by Remark~\ref{rem:MAD-rnd-det} below, \eqref{eq:MAD-rnd} and \eqref{eq:MAD} have the same leading behaviour:
\begin{equation*}
\Sigma_\MAD^\rnd=\Sigma_\MAD^\Det (1+\Oh_\p(n^{-1/2})),
\end{equation*}
with the additional fluctuation coming from the auxiliary sample. This fluctuation is asymptotically negligible under the assumptions below, but may be visible for moderate values of~$n$. In the sequel, we write $\Sigma_\MAD$ when a statement applies to either version.\\

\begin{rem}
\label{rem:sn_qn_extensions}
While the MAD estimator provides a clean benchmark, it is natural to consider whether the higher-efficiency robust scale estimators introduced by~\cite{MR1245360}, denoted by $S_n$ and $Q_n$ therein, can circumvent the convergence bottlenecks identified in Theorem~\ref{thm:CI-MAD}. A non-asymptotic concentration analysis for these estimators could potentially be developed by mapping the problem from classical empirical processes to the theory of \textit{$U$-empirical processes}.

Specifically, because $Q_n$ is defined as a localized quantile functional of the pairwise absolute differences, its concentration properties are fundamentally governed by the uniform fluctuations of the pairwise $U$-empirical distribution function:
\begin{equation}
H_n(t) 
\coloneqq \binom{n}{2}^{-1} \sum_{1 \le i < j \le n} \1_{\{|X_i - X_j| \le t\}}.
\end{equation}
In the independent regime of Corollary~\ref{cor:MAD-iid}, concentration inequalities for $\sup_{t>0} |H_n(t) - H(t)|$ can be established using standard $U$-process decoupling inequalities and uniform bounds for bounded canonical kernels; see, e.g.~\cite{MR1370308}. More generally, we conjecture that both $S_n$ and $Q_n$ may enjoy guarantees similar to those obtained here for MAD. While the proofs would require more work and careful calculations, we expect them to be of the same nature as those obtained for MAD in Theorem~\ref{thm:CI-MAD} and its corollaries.
\end{rem}

\noindent Since the MAD is defined through a median, its analysis will be phrased in terms of distances between distribution functions. We shall use the L\'evy distance, whose definition we recall next. Given two distribution functions $F$ and $G$ on $\R$, their L\'evy distance is defined by
\[
\Ld(F,G)
\coloneqq
\inf\left\{
\ve>0:
G(x-\ve)-\ve\le F(x)\le G(x+\ve)+\ve
\ \text{for all } x\in\R
\right\}.
\]
It is well known that $\Ld$ metrises weak convergence and is dominated by the Kolmogorov distance
\[
\Kd(F,G)\coloneqq \sup_{x\in\R}|F(x)-G(x)|.
\]
Moreover, if either $F$ or $G$ is Lipschitz continuous with constant $L$, then
\[
\Kd(F,G)\le (1+L)\Ld(F,G),
\]
see, e.g.,~\cite[\S3]{ProbMetrics}.

\noindent In order to state the concentration estimates in a compact form, we first introduce the notation used for empirical laws. Given a vector $\bm v=(v_1,\ldots,v_n)\in\R^n$, we write
\[
\iota[\bm v]
\coloneqq
\frac{1}{n}\sum_{k=1}^n \delta_{v_k}
\]
for the empirical measure associated with $\bm v$. We use the same notation for its distribution function, namely
\[
\iota[\bm v](x)
=
\frac{1}{n}\sum_{k=1}^n \1_{\{v_k\le x\}},
\]
for a real number $x$. Thus, $\iota$ embeds vectors in $\R^n$ into probability measures on $\R$. With this notation, our first objective is to provide concentration inequalities, in L\'evy distance, for $\iota[\bm Z]$ and $\iota[\bm X]$ around $\Psi$ and $\Psi(\cdot/\sigma)$, respectively.

\begin{prop}
\label{prop:CI-Levy}
For any $\bm{Y}\in\R^n$ and $p > 0$, we have
\begin{equation}
\Ld(\iota[\bm{X}], \Psi(\cdot/\sigma))
\le (\sigma\vee 1)\Ld(\iota[\bm{Z}], \Psi) 
+ \big( n^{-1}\|\bm{Y}\|_p^p\big)^{\frac{1}{p+1}}.
\end{equation}
\end{prop}

\begin{proof}
For any $x \in \R$, $\ve > 0$ and $k\in\br{n}$, the event inclusion 
\[
\{X_k\le x\}
=\{Y_k+\sigma Z_k\le x\}
\subset 
\{\sigma Z_k\le x + \ve\}
\cup\{|Y_k|>\ve\},
\]
yields the inequality
\[
\iota[\bm{X}](x) 
= \frac{1}{n}\sum_{k=1}^n 
    \1_{\{Y_k + \sigma Z_k \le x\}}
\le \frac{1}{n}\sum_{k=1}^n 
    \1_{\{\sigma Z_k \le x + \ve\}} 
+ \frac{1}{n}\sum_{k=1}^n 
    \1_{\{|Y_k| > \ve\}}.
\]
Then, using the inequality $\1_{\{|Y_k| > \ve\}}\le \ve^{-p}|Y_k|^p$, we have 
$$\sum_{k=1}^n \1_{\{|Y_k| > \epsilon\}} \le \ve^{-p}\|\bm{Y}\|_p^p.$$ 
An analogous lower bound follows similarly, yielding the sandwich
\[
\iota[\sigma \bm{Z}](x - \ve) 
- n^{-1}\ve^{-p}\|\bm{Y}\|_p^p
\le \iota[\bm{X}](x) 
\le \iota[\sigma\bm{Z}](x + \ve) 
+ n^{-1}\ve^{-p}\|\bm{Y}\|_p^p.
\]Taking $\ve = (n^{-1}\|\bm{Y}\|_p^p)^{1/(p+1)}$, the previous display gives 
$$\Ld(\iota[\bm{X}],\iota[\sigma\bm{Z}])\le \ve.$$ 
The triangle inequality and the bound 
$$\Ld(\iota[\sigma\bm{Z}], \Psi(\cdot/\sigma))\le (\sigma\vee 1)\Ld(\iota[\bm{Z}], \Psi),$$ 
then give the desired claim.
\end{proof}

\noindent We now turn the previous L\'evy-distance control into a concentration inequality for $\Sigma_\MAD$. 

\begin{thm}
\label{thm:CI-MAD}
Fix $p>0$. Assume that $\Psi$ and $\Psi_*$ are continuous at their respective medians
$m=\Psi^{-1}(\tfrac{1}{2})$ and $m_*=\Psi_*^{-1}(\tfrac{1}{2})$, and that,
for some $\kappa,\delta>0$ and all $\eta\in(0,\delta]$,
\begin{equation}
\label{eq:Dini-LB-F}
\kappa\eta
\le \min\big\{\tfrac{1}{2}-\Psi(m-\eta), 
\Psi(m+\eta)-\tfrac{1}{2},
\tfrac{1}{2}-\Psi_*(m_*-\eta), 
\Psi_*(m_*+\eta)-\tfrac{1}{2}\big\}.
\end{equation}
Define
\[
\Delta_\MAD
\coloneqq
\max\{\sigma,1\}\Ld(\iota[\bm{Z}],\Psi)
+
\big( n^{-1}\|\bm{Y}\|_p^p \big)^{\frac{1}{p+1}}.
\]
There exists $\varepsilon_\MAD>0$, depending only on $\sigma,\kappa$ and $\delta$,
such that, if $\Delta_\MAD<\varepsilon_\MAD$, then, almost surely,
\begin{equation}
\label{eq:CI-MAD}
|\Sigma_\MAD^\Det - \sigma| 
\le 
\frac{(1+\sigma/\kappa)(2+\sigma/\kappa)}
{\Psi_*^{-1}(\tfrac{1}{2})}
\Delta_\MAD.
\end{equation}
\end{thm}

\begin{rem}
\label{rem:MAD-explicit-threshold}
The threshold $\varepsilon_\MAD$ in Theorem~\ref{thm:CI-MAD} may be chosen explicitly as
\[
\varepsilon_\MAD
=
\frac{\delta\kappa}{2+\sigma/\kappa}.
\]
For this choice, the condition $\Delta_\MAD<\varepsilon_\MAD$ guarantees that the median perturbations remain in the neighbourhood where the local growth estimate \eqref{eq:Dini-LB-F} applies.
\end{rem}

\begin{rem}
\label{rem:MAD-rnd-det}
Under the assumption~\eqref{eq:Dini-LB-F}, applying Theorem~\ref{thm:CI-MAD} to the i.i.d. auxiliary sample $\bm\xi$ and the Dvoretzky-Kiefer-Wolfowitz (DKW) inequality imply that
\begin{equation}
\label{eq:MAD-rnd-det}
\frac{\Sigma_\MAD^\rnd}{\Sigma_\MAD^\Det}= 1+\Oh_\p\big(n^{-1/2}\big).
\end{equation}
\end{rem}

\begin{rem}
\label{rem:Dini-LB-F}
Condition \eqref{eq:Dini-LB-F} is a local lower-growth assumption (related to Dini derivatives) at the two medians involved in the MAD construction. It requires that both $\Psi$ and $\Psi_*$ put mass at least linearly, from the left and from the right, around their respective medians. In particular, it is satisfied whenever $\Psi$ and $\Psi_*$ are absolutely continuous in neighbourhoods of $m$ and $m_*$, respectively, with densities bounded away from zero there. 
\end{rem}

\begin{rem}
\label{rem:MAD-consistency-rate}
The bound \eqref{eq:CI-MAD} immediately gives the consistency of the deterministic MAD estimator under the assumptions
\[
\|\bm{Y}\|_p^p=\oh_\p(n)
\qquad\text{and}\qquad
\Ld(\iota[\bm{Z}],\Psi)=\oh_\p(1).
\]
Indeed, under these conditions, $\Delta_\MAD=\oh_\p(1)$ and therefore, by~\eqref{eq:MAD-rnd-det},
$\Sigma_\MAD\cip\sigma$. More quantitatively, if for some
$\alpha\in[0,1)$ and $\gamma>0$,
\[
\|\bm{Y}\|_p^p=\Oh_\p(n^\alpha)
\qquad\text{and}\qquad
\Ld(\iota[\bm{Z}],\Psi)=\Oh_\p(n^{-\gamma}),
\]
then
\[
|\Sigma_\MAD-\sigma|
=
\Oh_\p\left(
n^{-\min\left\{\gamma,\frac{1-\alpha}{p+1}\right\}}
\right).
\]
\end{rem}

\begin{proof}[Proof of Theorem \ref{thm:CI-MAD}]
The proof has two parts. We first prove a deterministic stability estimate for the MAD functional. This is done in three steps. More precisely, we show that there exists a constant $C>0$ such that 
\[
|\Sigma_\MAD^\Det-\sigma|
\le
C\Ld\bigl(\iota[\bm{X}],\Psi(\cdot/\sigma)\bigr).
\]
Once this stability estimate is proved, Proposition~\ref{prop:CI-Levy} can be applied to control the right-hand side in terms of
\[
\Ld(\iota[\bm Z],\Psi)
\qquad\text{and}\qquad
\bigl(n^{-1}\|\bm Y\|_p^p\bigr)^{1/(p+1)}.
\]
The desired estimate will then follow.\\

\noindent \textbf{Step 1}\\
\noindent We claim that, if $\ve=\Ld(F,G)$ is sufficiently small, namely if
$\ve<\kappa\delta$, then the medians of both distributions are
separated by at most a constant multiple of $\ve$. In fact, we claim that
\begin{equation}
\label{eq:median-diff}
\min\big\{|\hat m-\sigma m|,
\sigma\delta+\ve\big\}
\le \ve (1+\sigma/\kappa),
\end{equation}
where $\hat{m} \coloneqq G^{-1}(1/2)$ is the empirical median. Let
$F(x) \coloneqq \Psi(x/\sigma)$ and $G(x) \coloneqq \iota[\bm{X}](x)$, for
$x\in\R$. The condition~\eqref{eq:Dini-LB-F} implies that
\begin{equation}
\label{eq:Dini-LB-F-scaled}
F(\sigma m-\eta) + \eta \kappa/\sigma 
\le F(\sigma m)
= \tfrac{1}{2}
\le F(\sigma m+\eta) - \eta \kappa/\sigma,
\end{equation}
for $0<\eta\leq \sigma\delta$. By the definition of the L\'evy distance, for any real $x$,
\begin{equation}
\label{eq:FG-Levy}
F(x - \ve) - \ve 
\le G(x) 
\le F(x + \ve) + \ve.
\end{equation}
If $|\hat m-\sigma m|\le \ve$, the claim follows, so it remains to establish the claim when $|\hat m-\sigma m|>\ve$. Suppose first that $\hat m-\ve >\sigma m$. By the definition of $\hat{m}$, we have $G(x)<1/2$ for $x<\hat{m}$, and consequently,
\begin{equation*}
F(x - \ve) - \ve 
\le G(x) < 1/2 
= F(\sigma m).
\end{equation*}
Taking $x$ converging to $\hat m$ from below, we obtain
\[
F(\hat{m} - \ve-) \le F(\sigma m) + \ve.
\]
Since $\hat{m} - \ve > \sigma m$, relation~\eqref{eq:Dini-LB-F-scaled} gives
\[
F(\hat{m} - \ve-) 
\ge 
F(\sigma m) 
+ 
\min\{\hat{m} - \ve - \sigma m,\sigma\delta\}\kappa/\sigma.
\]
Thus,
\[
\ve 
<\min\{\hat{m} - \sigma m,\sigma\delta + \ve\} 
\le \ve(1 + \sigma/\kappa).
\]
Since $\delta>0$ is fixed, the above inequality gives the desired bound as soon as $\ve<\delta\kappa$.

Suppose now that $\hat m+\ve <\sigma m$. Since $G(\hat{m}) \ge 1/2$ by right-continuity, an application of \eqref{eq:FG-Levy} with $x = \hat{m}$ gives 
\begin{equation*}
F(\sigma m)
= 1/2 
\le G(\hat{m}) 
\le F(\hat{m} + \ve) + \ve.
\end{equation*}
Since $\hat{m} + \varepsilon < \sigma m$, relation~\eqref{eq:Dini-LB-F-scaled} gives
\[
F(\hat{m} + \varepsilon) 
\le 
F(\sigma m) 
- 
\min\{\sigma m - \hat{m} - \varepsilon,\sigma\delta\}\kappa/\sigma.
\]
Combining both inequalities yields
\[
\min\{\sigma m - \hat{m},\sigma\delta+\ve\} 
\le 
\varepsilon(1 + \sigma/\kappa).
\]
Relation~\eqref{eq:median-diff} follows.\\

\noindent \textbf{Step 2}\\
\noindent Define, for $x\geq 0$,  
\[
G_*(x)\coloneqq \iota[|\bm X-\hat m|](x),
\qquad
F_*(x)\coloneqq F(\sigma m+x)-F((\sigma m-x)-)
=\Psi_*(x/\sigma).
\]
These are the distribution functions of the absolute deviations from the corresponding medians. Using~\eqref{eq:FG-Levy} at the two endpoints of the intervals defining $F_*$ and $G_*$, we obtain
\[
F_*(x-\ve-|\hat m-\sigma m|)-2\ve
\leq
G_*(x)
\leq
F_*(x+\ve+|\hat m-\sigma m|)+2\ve,
\]
for every $x\geq0$. Consequently,
\[
\Ld(F_*,G_*)
\leq
\max\{2\ve,\ve+|\hat m-\sigma m|\}.
\]

Using Step~1, we obtain
\begin{equation}
\label{eq:Levy-absolute-deviations}
\Ld(F_*,G_*)
\le
\ve(2+\sigma/\kappa),
\end{equation}
whenever $\ve<\delta\kappa$.\\

\noindent \textbf{Step 3}\\ 
\noindent We now apply the same median argument to the laws of the absolute deviations.
The median of $F_*$ is $\sigma m_*$, where $m_*=\Psi_*^{-1}(\tfrac12)$, while
the median of $G_*$ is
\[
G_*^{-1}(\tfrac12)
=
\med(|\bm X-\hat m|).
\]
Set $\ve_* \coloneqq \Ld(F_*,G_*)$. By \eqref{eq:Levy-absolute-deviations}, if
\[
\ve(2+\sigma/\kappa)<\delta\kappa,
\]
then $\ve_*<\delta\kappa$. Therefore the estimate of Step~1, applied now
to $F_*$ and $G_*$ and using the local growth condition for $\Psi_*$, gives
\[
\big|G_*^{-1}(\tfrac12)-F_*^{-1}(\tfrac12)\big|
\le
\ve_*(1+\sigma/\kappa).
\]
Since
\[
F_*^{-1}(\tfrac12)=\sigma\Psi_*^{-1}(\tfrac12),
\qquad
\Sigma_\MAD^\Det
=
\frac{G_*^{-1}(\tfrac12)}{\Psi_*^{-1}(\tfrac12)},
\]
we conclude that
\[
|\Sigma_\MAD^\Det-\sigma|
\le
\frac{1+\sigma/\kappa}{\Psi_*^{-1}(\tfrac12)}
\ve_*.
\]
Combining this with \eqref{eq:Levy-absolute-deviations} yields
\[
|\Sigma_\MAD^\Det-\sigma|
\le
\frac{(1+\sigma/\kappa)(2+\sigma/\kappa)}
{\Psi_*^{-1}(\tfrac12)}
\ve.
\]
Finally, Proposition~\ref{prop:CI-Levy} gives the stated bound in terms of
$\Delta_\MAD$.
\end{proof}

\noindent Theorem~\ref{thm:CI-MAD} reduces the control of the MAD estimator to the size of the signal perturbation and to the empirical discrepancy $\Ld(\iota[\bm Z],\Psi)$ of the noise sample. We now state resulting rates in two basic situations: i.i.d. noise, where the empirical discrepancy is controlled by the Dvoretzky-Kiefer-Wolfowitz (DKW) inequality, and a second one, which considers Gaussian dependence, where the rate is expressed in terms of the covariance quantity $D_n$.

\begin{cor}
\label{cor:MAD-iid}
Assume $\bm{Z}$ consists of i.i.d. random variables and suppose~\eqref{eq:Dini-LB-F} holds. Then,
\begin{equation*}
|\Sigma_\MAD - \sigma| 
= 
\Oh_\p\big(
n^{-1/(p+1)}\|\bm{Y}\|_p^{p/(p+1)}
+
n^{-1/2}
\big),
\end{equation*}
as $n$ tends to infinity.
\end{cor}

\begin{rem}
\label{rem:MAD-iid-consistency}
As a consequence of Corollary~\ref{cor:MAD-iid}, we have that $\Sigma_\MAD^\Det$ is consistent whenever $\|\bm{Y}\|_p^p=\oh_\p(n)$.  Moreover, if $\|\bm{Y}\|_p^p=\Oh_\p(n^\alpha)$ for some $\alpha\in[0,1)$, then
\[
|\Sigma_\MAD - \sigma|
=
\Oh_\p\bigl(
n^{-\min\{(1-\alpha)/(p+1),1/2\}}
\bigr).
\]
\end{rem}

\begin{proof}[Proof of Corollary \ref{cor:MAD-iid}]
The L\'evy distance $\Ld(\iota[\bm{Z}],\Psi)$ is bounded by the Kolmogorov distance $\Kd(\iota[\bm{Z}],\Psi)$, which is $\Oh_\p(n^{-1/2})$ by the DKW inequality. Theorem~\ref{thm:CI-MAD} then gives the claim.
\end{proof}

\begin{cor}
\label{cor:mad_correlated}
Assume $\bm{Z}=\Psi^{-1}\circ\Phi(\bm{W})$, where $\bm{W}$ is a  stationary zero mean unit variance Gaussian process such that
\[
\lim_{k\to\infty}|\Cov[W_1,W_{k+1}]|=0,
\]
and suppose that~\eqref{eq:Dini-LB-F} holds. Set
\[
D_n \coloneqq \sum_{i=1}^n\sum_{j=1}^n |\Cov[W_i, W_j]|\ge n.
\]
Then
\begin{equation}
|\Sigma_\MAD - \sigma| 
= \Oh_{\p}\big(
n^{-1/(p+1)}\|\bm{Y}\|_p^{p/(p+1)}
+ n^{-1}\sqrt{D_n}
\big),
\end{equation}
as $n$ tends to infinity.
\end{cor}

\begin{proof}
By Theorem~\ref{thm:CI-MAD},~\eqref{eq:MAD-rnd-det}, the inequality $\Ld\le\Kd$ and the fact that 
$$\Kd(F\circ h^{-1},G\circ h^{-1})\le \Kd(F,G)$$ 
for any monotone $h$ and distributions $F$ and $G$, it suffices to show that 
$$\Kd(\iota[\bm{W}],\Phi)=\Oh_\p(n^{-1}D_n^{1/2}).$$ 
However, this is precisely the content of Theorem~\ref{thm:Kol-Gauss-dependent} below, completing the proof.
\end{proof}

\subsubsection*{Comparison with the ordered-quantile estimators}

\noindent We now compare the guarantees for the estimators $\Sigma_r$ in Corollary~\ref{cor:CI_r12} with those for the MAD estimator in Theorem~\ref{thm:CI-MAD}. The MAD bound is less robust with respect to high signal activity: the signal contribution is controlled by the term $\bigl(n^{-1}\|\bm Y\|_p^p\bigr)^{1/(p+1)}$, rather than linearly by $n^{-1/p}\|\bm Y\|_p$, as for $\Sigma_r$. This discrepancy is also visible in the numerical experiments of the forthcoming Section~\ref{sec:numerical_experiments}.\\

\noindent This difference is consistent with the way the estimators use the ordered sample. The MAD is determined by two median operations, namely the empirical median and the median of the absolute deviations, while the estimators $\Sigma_r$ use the full collection of ordered observations. Thus, in the present framework, the latter estimators can distribute the effect of the signal over all sample quantiles, whereas the MAD remains more sensitive to local perturbations around the relevant medians. In particular, the estimators $\Sigma_r$ will generally lead to different estimates when using a single sample and different reference law $\Psi$ (because $\Psi$ is fully specified by its quantile function $\Psi^{-1}$), while the MAD estimator for different reference laws $\Psi$ may give the same result if these specific quantiles agree. For instance, we may take $\Psi=\Phi$ and modify its tails arbitrarily on $(-\infty,\Phi^{-1}(1/4))\cup(\Phi^{-1}(3/4),\infty)$ without affecting the MAD estimator.

\section{Functionals of weights and proxies}
\label{sec:weighted-proxies}

\noindent In this section we study the asymptotic behaviour of the deterministic quantities involving $(\bm w,\bm\psi)$ that appear in Corollary~\ref{cor:CI_r12}. In particular, we control the normalizing terms involving the weights and the empirical-to-proxy discrepancy
\[
\langle \bm w,|\bm Z^\ua-\bm\psi|^r\rangle.
\]
The analysis uses elementary Riemann and Poisson-sum approximations, together with standard weak limits for empirical and quantile processes.

\subsection{Riemann and Poisson sums}
\label{subsec:Riemann-Poisson-sums}

\noindent We begin with two elementary approximation results for Riemann and mean Poisson sums. They will be used to control the normalizing quantities associated with the deterministic and random proxies. Before stating them, we recall the notation and endpoint regularity assumptions used below.\\ 

\subsubsection*{Notation and conventions}

\noindent We begin with two elementary approximation results for Riemann and mean Poisson sums. They will be used to control the normalizing quantities associated with the deterministic and random proxies. Before stating them, we recall the notation and endpoint regularity assumptions used below.\\

\noindent In the sequel, we will write $f\sim g$ as $x$ tends to $a$, for some real number $a$,  if $f(x)/g(x)$ tends to one as $x$ tends to $a$. A positive measurable function $f$ is said to be regularly varying at $a\in\{0,\infty\}$ with index $\rho\in\R$, denoted by $f\in\RV_a^\rho$, if $f(\lambda t)/f(t)\to \lambda^\rho$ as $t\to a$, for every $\lambda>0$. Similarly, we say that $f:(0,1)\to\R$ varies regularly at $1$ if the mapping $t\mapsto f(1-t)$ varies regularly at $0$. If $\rho=0$, we say that $f\in\RV_a^0$ is slowly varying.\\

\noindent Recall from the Lebesgue-Vitali theorem that a function $f:D\to\R$ is locally Riemann integrable, i.e., Riemann integrable on compact subintervals of $D$, if and only if $f$ is locally bounded and the set of discontinuities of $f$ has zero Lebesgue measure. In particular, all monotone functions and all c\`adl\`ag functions are locally Riemann integrable, since they are locally bounded and have at most countably many discontinuities. Despite this terminology, all integrals below are understood in the Lebesgue sense.

\subsubsection*{Riemann and Poisson approximations}
\noindent We next record two approximation results adapted to the deterministic and random proxies, respectively. The assumptions allow for the endpoint singularities naturally produced by quantile functions.

\begin{lemma}
\label{lem:Riemann-sums}
Let $\{f_n\}_{n\in\N}$ be measurable functions on $(0,1)$ converging to a locally Riemann integrable function $f:(0,1)\to\R$ uniformly on compact subsets of $(0,1)$. Suppose that for a Lebesgue integrable function $g:(0,1)\to[0,\infty)$ we have $|f_n| \leqslant g$ on $(0,1)$ for all $n\in\N$. If $g \in \RV_0^{-\alpha_0}$ and $g(1-\cdot) \in \RV_0^{-\alpha_1}$ for some $\alpha_0, \alpha_1 \le 1$, then
\[
\lim_{n\to\infty} \frac{1}{n} \sum_{k=1}^n f_n\left(\frac{k}{n+1}\right) 
= \int_0^1 f(x)\dd x.
\]
\end{lemma}

\begin{lemma}
\label{lem:Poisson-sums}
Let $\{f_n\}_{n\in\N}$ be measurable functions converging a.e. to a Lebesgue integrable $f:(0,1)\to\R$, and suppose that $|f_n|\le g$ for some Lebesgue integrable $g$. Consider i.i.d. standard exponential variables $\{E_k;\ k\in\br{n}\}$, set $\Gamma_k\coloneqq\sum_{j=1}^k E_j$ and $V_{n,k}\coloneqq\Gamma_k/\Gamma_{n+1}$. Then
\[
\lim_{n\to\infty}\frac{1}{n}\sum_{k=1}^n\E[f_n(V_{n,k})]
=\int_0^1 f(x)\dd x.
\]
\end{lemma}

\noindent To simplify the exposition, the proofs of Lemmas~\ref{lem:Riemann-sums} and~\ref{lem:Poisson-sums} are relegated to Appendix~\ref{app:Riemann-Poisson} below. We now apply these lemmas to the weight and proxy functionals appearing in Theorem~\ref{thm:CI_r12}, and in particular to the asymptotic coefficients in Corollary~\ref{cor:CI_r12}. When the weights are generated from the proxies, these quantities become Riemann-type averages along either the deterministic quantile grid or its random analogue. The two lemmas allow us to replace these averages by the corresponding integrals under the reference law $\Psi$.

\begin{cor}
\label{cor:weighted-proxies}
Fix $r,q\ge 1$, and let $\omega:\R\to[0,\infty)$ be a continuous bounded function with 
$$\sup_{x\in\R}|x|^r\omega(x)<\infty.$$
Set $\bm{w}\coloneqq\omega(\bm\psi)$ and let $Z\sim \Psi$. If $\bm\psi$ is constructed via~\eqref{eq:proxies_deterministic} (resp.~\eqref{eq:proxies_random}), then the following deterministic (resp. $L^1$) limits hold as $n$ tends to infinity
\[
\frac{\|\bm{w}\|_q^q}{n}
\to\|\omega(Z)\|_{L^q}^q,
\qquad 
\frac{\|\bm{w}\pdot\bm\psi\|_q^q}{n}
\to\|Z\omega(Z)\|_{L^q}^q,
\qquad
\frac{\langle\bm{w},|\bm\psi|^r\rangle}{n}
\to\|Z^r\omega(Z)\|_{L^1}.
\]
\end{cor}

\begin{proof}
The statements for the deterministic proxy follow directly from Lemma~\ref{lem:Riemann-sums}. Consider now the case of the random proxy. By Scheff\'e's lemma, it suffices to show convergence in probability and in mean. Let $\bm{U}$ be as before, so that $\bm\xi^\ua=\Psi^{-1}(\bm{U}^\ua)$, where $U^\ua_k\sim\Beta(k,n-k+1)$ for $k\in\br{n}$. Hence, the result follows from Lemma~\ref{lem:Poisson-sums}.
\end{proof}


\subsection{Empirical-to-proxy discrepancy}
\label{sec:proxy_control}

The concentration inequality established in Corollary~\ref{cor:CI_r12} decouples the estimation error into two algebraically independent components. The objective of the present section is to control the weighted distance $\langle\bm{w},|\bm{Z}^\ua - \bm\psi|^r\rangle$ for $r=1,2$.\\

\noindent We first relate the random and deterministic proxy discrepancies. The elementary inequality
\[
(a+b)^r\le 2^{r-1}(a^r+b^r),
\]
valid for $a,b\ge 0$ and $ r\ge 1$, implies that
\[
\langle\bm{w},|\bm{Z}^\ua - \bm\psi^\rnd|^r\rangle
\le
2^{r-1}
\big[
\langle\bm{w},|\bm{Z}^\ua - \bm\psi^\Det|^r\rangle
+
\langle\bm{w},|\bm\psi^\rnd - \bm\psi^\Det|^r\rangle
\big],
\]
where $\bm\psi^\Det$ and $\bm\psi^\rnd$ are the proxies defined in
\eqref{eq:proxies_deterministic} and \eqref{eq:proxies_random}, respectively.
Both weighted empirical-to-proxy discrepancies are of similar natures (although the random proxy case originates from i.i.d. samples and the weights are a function of $\bm\psi^\rnd$), most of the analysis of this section is essentially reduced to the case of the deterministic proxy $\bm\psi^\Det$ in~\eqref{eq:proxies_deterministic}. Indeed, the case of random proxies essentially follows from the deterministic case, by virtue of the fast convergence of $\bm\psi^\rnd$ to its quantile function in the sense of Donsker (see also~\cite{MR815960} and~\cite[Ch.~18]{MR3396731}).\\

\noindent We first examine the case where the latent noise variables $\bm{Z}$ are i.i.d., deriving exact asymptotic integrals for the weighted distance. Subsequently, we find a stabilizing weight function $\omega$. Finally, we relax the independence assumption, extending our controls to complex, highly correlated sample regimes.

\subsubsection{Empirical-to-proxy discrepancy for i.i.d. samples}
\label{subsec:iid-empirical-proxy}

Throughout this subsection, we assume that the latent noise variables
$\bm Z=\{Z_k;\,k\in\br n\}$ are i.i.d. with distribution function $\Psi$. We also assume that $\Psi$ is absolutely continuous, with a continuous density $\Psi'$ which is strictly positive in the interior of the support of $\Psi$.  Consider the weights 
\begin{equation}
\label{eq:proxy_dependent_weights}
\bm w\coloneqq \omega(\bm\psi),
\qquad
w_k\coloneqq \omega(\psi_k),
\end{equation}
for $k\in\br n$, where $\omega:\R\to[0,\infty)$ is bounded and continuous. The size of the empirical-to-proxy discrepancy is governed by two competing effects. The fluctuations of the uniform order statistics vanish near the endpoints, while the derivative of the quantile map,
\[
\eta\coloneqq(\Psi^{-1})'
=
1/\Psi'\circ\Psi^{-1},
\]
may explode there. The next result quantifies this balance for the deterministic proxy.\\

\noindent Let $r\ge 1$ and let $\kappa:(0,1)\to[1,\infty)$ be a measurable function satisfying
\[
\lim_{u\da 0}
\frac{\min\{\kappa(u),\kappa(1-u)\}}
{\sqrt{\log\log(1/u)}}
=
\infty.
\]
We say that the pair $(\omega,\Psi)$ satisfies condition $\mathbf{H}_{\kappa,r}$ if the following assumptions hold.

\medskip

\noindent\textbf{Condition $(\mathbf{H}_{\kappa,r})$.}
\begin{itemize}[leftmargin=2em]
\item[-] The distribution function $\Psi$ is absolutely continuous, with a continuous density $\Psi'$ which is strictly positive in the interior of the support of $\Psi$.

\item[-] Either $\omega\circ\Psi^{-1}$ is supported on a compact subset of $(0,1)$, or the function
\[
\eta\coloneqq\frac{1}{\Psi'\circ\Psi^{-1}}
\]
varies regularly at both endpoints $0$ and $1$.

\item[-] The function
\[
u
\longmapsto
\omega(\Psi^{-1}(u))
\bigl[
\kappa(u)\eta(u)\sqrt{u(1-u)}
\bigr]^r
\]
is bounded on $(0,1)$ by a Lebesgue integrable function which varies regularly at both endpoints $0$ and $1$.
\end{itemize}

\begin{thm}[Weak limit of the empirical-to-proxy discrepancy]
\label{thm:weak_lim_proxy_det}
Let the notation of Sections~\ref{subsec:proxies} and~\ref{subsec:weighted-Lr-loss} prevail. Assume that $\bm\psi=\bm\psi^\Det$ and that $\bm Z=\{Z_k;\,k\in\br n\}$ is i.i.d. with distribution function $\Psi$. If the pair $(\omega,\Psi)$ satisfies condition $(\mathbf{H}_{\kappa,r})$, then
\begin{equation}
\label{eq:weak_lim_proxy_det}
n^{\frac{r}{2}-1}
\langle\bm{w},|\bm{Z}^\ua-\bm\psi|^r\rangle
\cid
\int_0^1
\omega(\Psi^{-1}(u))
\eta(u)^r
|\beta_u|^r
\dd u,
\end{equation}
as $n\to\infty$, where $\beta$ is a standard Brownian bridge on $[0,1]$.
\end{thm}

\begin{rem}
\label{rem:Riemann-sums}
By~\cite{Denisov2006}, if $f:(0,1]\to[0,\infty)$ is eventually monotone at $0$ and Lebesgue-integrable on $(0,1]$, then there exists some $g\in\RV_0^{-1}$ with $f\le g$ and $\int_0^1 g(x)\dd x<\infty$.
\end{rem}

Under a slightly stronger domination assumption, we may establish the corresponding result for the random proxy.

\begin{thm}
\label{thm:weak_lim_proxy_rand}
Let the notation of Sections~\ref{subsec:proxies} and~\ref{subsec:weighted-Lr-loss} prevail. Assume that $\bm\psi=\bm\psi^\rnd$ and that $\bm Z=\{Z_k;\,k\in\br n\}$ is i.i.d. with distribution function $\Psi$. Let $\bm\xi=\{\xi_k;\,k\in\br n\}$ be the auxiliary i.i.d. sample used to generate $\bm\psi^\rnd$, independent of $\bm Z$. Suppose that $(\omega,\Psi)$ satisfies condition $\mathbf{H}_{\kappa,r}$.

Assume moreover that there exists a function $\varpi:(0,1)\to[0,\infty)$ such that
\[
\omega\circ\Psi^{-1}\le \varpi
\]
on $(0,1)$, where $\varpi$ is either supported on a compact subset of $(0,1)$ or varies regularly at both endpoints $0$ and $1$. Suppose also that the function
\[
u
\longmapsto
\varpi(u)
\bigl[
\kappa(u)\eta(u)\sqrt{u(1-u)}
\bigr]^r
\]
is bounded on $(0,1)$ by a Lebesgue integrable function which varies regularly at both endpoints $0$ and $1$. Then
\begin{equation}
\label{eq:weak_lim_proxy_rand}
n^{\frac{r}{2}-1}
\langle\omega(\bm\psi^\rnd),|\bm{Z}^\ua-\bm\psi^\rnd|^r\rangle
\cid
2^{r/2}
\int_0^1
\omega(\Psi^{-1}(u))
\eta(u)^r
|\beta_u|^r
\dd u,
\end{equation}
as $n$ tends to infinity, where $\beta$ is a standard Brownian bridge on $[0,1]$.
\end{thm}

\noindent The proofs of Theorems~\ref{thm:weak_lim_proxy_det} and~\ref{thm:weak_lim_proxy_rand} are found in Appendix~\ref{app:Riemann-Poisson} below. 

\begin{rem}
\label{rem:weak_lim_proxy}
(i) Notably, if the derivative $\Psi'$ is accessible, one may take $\omega(x)=\Psi'(x)^r$ in Theorems~\ref{thm:weak_lim_proxy_det} and~\ref{thm:weak_lim_proxy_rand} to obtain a distribution invariant limiting law $\int_0^1 |\beta_u|^r\dd u$. In the case of the random proxy, this choice does \emph{not} require access to $\Psi^{-1}$. Moreover, the regularly varying assumptions in both theorems are immediately satisfied as soon as $1/\eta$ is bounded by a regularly varying function at $0$ and $1$, in which case we may take $\kappa(u)=\sqrt{|\log(u(1-u))|}$, for instance.\\

(ii) Another sensible choice of weight function $\omega$ is a continuous modification of $\1_{[-c,c]}$ for some constant $c\in(0,\infty)$, essentially trimming the tails of $\Psi$. This aggressive weight function removes any concern for the behaviour of possible ill behaviour of the functions in Theorems~\ref{thm:weak_lim_proxy_det} and~\ref{thm:weak_lim_proxy_rand} at $0$ or $1$.
\end{rem}

In case a user is interested in other weight functions, it is valuable to understand when $\eta=1/\Psi'\circ\Psi^{-1}$ is regularly varying at $0$ and $1$. The following result identifies two important families of distributions with this property. Namely, those whose tails or log-tails are regularly varying. For simplicity, we only consider the positive tails, with a similar analysis being valid for the negative tails.\\

\noindent Before presenting the result, we recall that if a differentiable function $f\in\RV_a^\rho$, then its derivative $f'$ is in $\RV_a^{\rho-1}$ whenever it is eventually monotone~\cite[Thm~1.7.2]{MR1015093}, however, the converse is stronger. Namely, if $f'\in\RV_a^{\rho-1}$ is regularly varying, then its antiderivative $f\in\RV_a^\rho$ by virtue of Karamata's theorem~\cite[Thm~1.5.11]{MR1015093}.

\begin{prop}
\label{prop:RV-tails}
Suppose $\Psi$ is absolutely continuous with Radon-Nikodym derivative $\Psi'$. Let $g$ be an element of $\RV_0^{-1}$ with $G(x)\coloneqq\int_x^1g(y)\dd y$ tending  to infinity as $x$ tends to zero. Then, the following statements hold.
\begin{itemize}[leftmargin=2em]
\item[{\normalfont(a)}] If the derivative $\Psi'(x)=x^{-\beta-1}\ell(x)$ for some $\beta>0$ and $\ell\in\RV_\infty^0$, then 
\[
\Psi^{-1}(1-u)
=u^{-1/\beta}\tilde\ell(u)
\quad\text{and}\quad
\eta(1-u)
\sim \beta^{-1}u^{-1-1/\beta}\tilde\ell(u)
\]
as $u$ tends to zero, where $\tilde\ell$ is slowly varying at $0$. In particular, $\eta$ belongs to $\RV_1^{-1-1/\beta}$.
\item[{\normalfont(b)}] If the derivative 
$$\frac{\dd}{\dd x}G(1-\Psi(x))=g(1-\Psi(x))\Psi'(x)$$ 
belongs to $\RV_\infty^{\beta-1}$ for some $\beta>0$, then $\eta$ belongs to $\RV_1^{-1}$.
\end{itemize}
\end{prop}

\begin{proof}[Proof of Proposition~\ref{prop:RV-tails}]
Part (a). Note that
$$1-\Psi(x)\sim \beta^{-1}x^{-\beta}\ell(x)$$
as $x$ tends to infinity by Karamata's theorem~\cite[Thm~1.5.11]{MR1015093}. Moreover, by~\cite[Thm~1.5.12]{MR1015093}, the generalized inverse is also regularly varying: $\Psi^{-1}(1-u)=u^{-1/\beta}\tilde\ell(u)$ as $u$ tends to zero,  for some $\tilde\ell\in\RV_0^0$. Thus, using that
$\Psi'(x)\sim\beta(1-\Psi(x))/x$
as $x\to\infty$, we obtain
\[
\eta(1-u)
=\frac{1}{\Psi'\circ\Psi^{-1}(1-u)}
\sim \frac{\Psi^{-1}(1-u)}
{\beta\bigl(1-\Psi\circ\Psi^{-1}(1-u)\bigr)}
=\beta^{-1}u^{-1-1/\beta}\tilde\ell(u),
\]
as $u$ tends to zero.\\

\noindent Part (b). Karamata's theorem gives 
\begin{equation}
\label{eq:G-RV}
G\circ(1-\Psi(x))
\sim
\beta^{-1}xg(1-\Psi(x))\Psi'(x),
\end{equation}
as $x$ tends to infinity,  and is thus in $\RV_\infty^\beta$. Moreover, the positivity and regular variation of $g$ implies that $G$ is a slowly varying function (by Karamata's theorem~\cite[Thm~1.5.11]{MR1015093}), strictly decreasing and continuous and hence a bijection from $(0,1)$ to $(0,\infty)$. Let $G^{-1}$ denote the inverse of $G$. Then, the inverse map of $G\circ(1-\Psi)$ satisfies $\Psi^{-1}\circ(1-G^{-1})\in\RV_\infty^{1/\beta}$, where $\tilde\ell\in\RV_\infty^0$. Thus, $\Psi^{-1}(1-u)=G(u)^{1/\beta}\ell(G(u))\in\RV_0^0$ for some $\ell\in\RV_\infty^0$. Hence, solving for $1/\Psi'$ in~\eqref{eq:G-RV} and substituting $x=\Psi^{-1}(1-u)$ gives, as $u$ tends to zero,
\[
\eta(1-u)
= \frac{1}{\Psi'(\Psi^{-1}(1-u))}
\sim
\beta^{-1}
\frac{\Psi^{-1}(1-u)g(u)}{G(u)}
\in\RV_0^{-1}.\qedhere
\]
\end{proof}

\begin{rem}
The distributions considered in part (a) of Proposition~\ref{prop:RV-tails} encompass most power tail laws (i.e., in the domain of attraction of stable laws), while part (b) encompasses most stretched exponential tails ($g(x)=1/x$), as well as lighter tails (e.g., $g(x)=1/[x\log(1/x)]$) or heavier tails (e.g., $g(x)=\exp(\sqrt{\log(1/x)})/x$) that are still lighter than any power tail law.
\end{rem}

\subsubsection{Dependent Noise}
\label{subsec:dependent_noise}

In numerous physical and econometric applications, the underlying noise process exhibits persistent temporal dependence. We now relax the assumption of mutual independence among the latent noise variables $\bm{Z} = \{Z_k;k\in\br{n}\}$. Since the deterministic concentration inequalities established in Theorem~\ref{thm:CI_r12} remain valid; the structural impact of the noise dependence is entirely absorbed by the empirical-to-proxy discrepancy $\langle\bm{w},|\bm{Z}^\ua - \bm\psi|^r\rangle$, affecting only its asymptotic behaviour.\\

\noindent To control the empirical-to-proxy discrepancy, we will borrow techniques from~\cite{MR4028181} to control optimal transport distances. More precisely, we use~\cite[Thm~2.9]{MR4028181} to obtain the formula
\begin{equation}
\label{eq:W1}
\Wd_1(F,G)
=\int_{-\infty}^\infty |F(x)-G(x)|\dd x,
\end{equation}
where $\Wd_p(F,G)$ denotes the $L^p$-Wasserstein distance between the distribution functions $F$ and $G$ on $\R$, given by the formula~\cite[Thm~2.10]{MR4028181}
\begin{equation}
\label{eq:Wp}
\Wd_p(F,G)^p
\coloneqq\int_0^1|F^{-1}(u)-G^{-1}(u)|^p\dd u,
\end{equation}
for $ p\ge 1.$ In the case $r=2$, we use the following upper bound~\cite[Prop.~7.14]{MR4028181}:
\begin{equation}
\label{eq:W2}
\Wd_2(F,G)^2
\le 4\int_{-\infty}^\infty|x|\,|F(x)-G(x)|\dd x.
\end{equation}

Since $\omega$ is bounded, the empirical-to-proxy discrepancy is directly bounded by the product between $n^{-1}\|\omega\|_\infty$ and the $L^p$-Wasserstein distance in~\eqref{eq:Wp} between the discrete distributions 
\[
\iota[\bm{Z}](x)\coloneqq \frac{1}{n}\sum_{k=1}^n\1_{\{Z_k\le x\}}
\quad\text{and}\quad
\iota[\bm{\psi}](x)\coloneqq \frac{1}{n}\sum_{k=1}^n\1_{\{\psi_k\le x\}}.
\]
However, the proofs for the formulae~\eqref{eq:W1} and~\eqref{eq:W2} cannot be easily adapted to the general weighted case without introducing more random elements that complicate the results without substantially increasing their generality. In this sense, the choice of weight function $\omega$ only serves to control the finite-sample variations and the breakdown points in the case $r=1$. 

\begin{prop}[Expected empirical-to-proxy discrepancy]
\label{prop:expected_dependent_discrepancy}
Let $\bm{Z}=\{Z_k;k\in\br{n}\}$ be stationary with marginal law $\Psi$. Let $\bm\psi$ be generated via either~\eqref{eq:proxies_deterministic} or~\eqref{eq:proxies_random}. Assume that some non-negative $\{\rho_{ij};i,j \in\br{n}\}$ satisfy
\begin{equation}
\label{eq:indicator_covariance_bound}
|\Cov[ \1_{\{Z_i \le x\}}, \1_{\{Z_j \le x\}} ]| 
\le \rho_{ij} \Psi(x) (1 - \Psi(x)),
\end{equation}
for $x\in\R$. Set $W\coloneqq \|\omega\|_\infty$ and
\[
D_n \coloneqq \sum_{i=1}^n \sum_{j=1}^n \rho_{ij}.
\]
Then, the following inequalities hold
\begin{align*}
\E\big[\langle\bm{w}, 
|\bm{Z}^\ua-\bm\psi|\rangle\big] 
&\le W\big(\sqrt{D_n}+\sqrt{n}\big)
\int_{-\infty}^\infty \sqrt{\Psi(x)(1-\Psi(x))}\dd x,\\
\E\big[\langle\bm{w}, 
|\bm{Z}^\ua-\bm\psi|^2\rangle\big] 
&\le 4W\big(\sqrt{D_n}+\sqrt{n}\big)
\int_{-\infty}^\infty |x|\sqrt{\Psi(x)(1-\Psi(x))}\dd x.
\end{align*}
\end{prop}

\begin{rem}
\label{rem:expected_dependent_discrepancy}
(a) Note that the coefficients $\rho_{ij}$ (and hence $D_n$) are independent of the distribution $\Psi$ as long as it is continuous, and only depend on the dependence structure. Indeed, if $\Psi$ is continuous, then the variables $U_i=\Psi(Z_i)$ are uniform on $(0,1)$ and
\[
\1_{\{Z_i\le x\}}
=
\1_{\{U_i\le \Psi(x)\}}.
\]
Thus,~\eqref{eq:indicator_covariance_bound} is equivalent to
\begin{equation}
\label{eq:indicator_covariance_bound_unif}
|\Cov[ \1_{\{U_i \le u\}}, \1_{\{U_j \le u\}} ]| 
\le \rho_{ij}\, u (1 - u),
\end{equation}
for $0\le u\le 1.$\\

\noindent(b) From the proof of Proposition~\ref{prop:expected_dependent_discrepancy}, the following inequality can be recovered 
\[
\E\left[
\Wd_r(\iota[\bm{Z}],\Psi)^r
\right]
\le
r^2 n^{-1}\sqrt{D_n}
\int_{\R}
|x|^{r-1}
\sqrt{\Psi(x)(1-\Psi(x))}
\,\dd x,
\]
for $r=1,2$. Since $\Ld=\Pd\le\sqrt{\Wd_1}$ by~\cite[Thm~2]{ProbMetrics}, we can also derive a bound for $\Ld$ and hence, establish control for the MAD estimator (via Theorem~\ref{thm:CI-MAD}); however, the convergence rate would necessarily deteriorate, because of the square root. Hence, in Corollary~\ref{cor:mad_correlated}, we opted for stricter conditions that yield improved rates.
\end{rem}

\begin{proof}[Proof of Proposition~\ref{prop:expected_dependent_discrepancy}]
By virtue of~\eqref{eq:W1} and~\eqref{eq:Wp}, we have
\[
\langle\bm{w},|\bm{Z}^\ua-\bm\psi|\rangle
\le nW\int_{-\infty}^\infty |\iota[\bm{Z}](x)-\iota[\bm\psi](x)|\dd x.
\]
Since $\E[\1\{Z_k\le x\}]=\Psi(x)$ for $x\in\R$, $k\in\br{n}$, Jensen's inequality and our assumption yield
\[
\E\big[|\iota[\bm{Z}](x)-\Psi(x)|\big]^2
\le \frac{1}{n^2}\sum_{i=1}^n\sum_{j=1}^n
    \Cov\big[\1_{\{Z_i\le x\}},\1_{\{Z_j\le x\}}\big]
\le \frac{D_n}{n^2} \Psi(x)(1-\Psi(x)).
\]

Suppose for the moment that $\bm\psi$ is generated via~\eqref{eq:proxies_deterministic} and that we establish the inequality
\begin{equation}
\label{eq:CDF-vs-quantile-CDF}
|\iota[\bm\psi](x)-\Psi(x)|
\le \sqrt{\Psi(x)(1-\Psi(x))/n},
\end{equation}
for $x$ in $\R$. Then, the previous two displays and the triangle inequality yield
\[
\E[|\iota[\bm{Z}](x)-\iota[\bm\psi](x)|]
\le \big(n^{-1}\sqrt{D_n} +n^{-1/2}\big)\sqrt{\Psi(x)(1-\Psi(x))}.
\]
Integrating the previous display over $\R$ would then yield the first inequality in the statement. The second one would follow similarly, via~\eqref{eq:W2}. The case of random proxies, generated via~\eqref{eq:proxies_random}, is completely analogous. Indeed, in that case, the independence gives 
$$\Cov[\1_{\{\xi_i\le x\}},\1_{\{\xi_j\le x\}}]=\1_{\{i=j\}}\Psi(x)(1-\Psi(x)),$$
whose sum over all $i,j\in\br{n}$ is $n\Psi(x)(1-\Psi(x))$, so the previous procedure yields
\begin{align*}
\E[|\iota[\bm{Z}](x)-\iota[\bm\xi](x)|]
&\le \E[|\iota[\bm{Z}](x)-\Psi(x)|]
+\E[|\iota[\bm{\xi}](x)-\Psi(x)|]\\
&\le \big(n^{-1}\sqrt{D_n} +n^{-1/2}\big)\sqrt{\Psi(x)(1-\Psi(x))}.
\end{align*}

To complete the proof, it remains to establish~\eqref{eq:CDF-vs-quantile-CDF} when the proxies are deterministic and generated via~\eqref{eq:proxies_deterministic}. First, note that $\psi_k\le x$ is equivalent to $k\le (n+1)\Psi(x)$, so 
$$\iota[\psi](x) 
= \min\{1, \lfloor (n+1)\Psi(x) \rfloor/n\}$$ for $x$ in $\R$. Let $y=\Psi(x)$, then for $y\in [0,1/(n+1)]$, we have $\iota[\bm\psi](x)=0$ and thus 
$$|\iota[\bm\psi](x)-\Psi(x)|=y\le \sqrt{y(1-y)/n}.$$ 
The case $y\in[n/(n+1),1]$ follows symmetrically. Finally, consider $y\in(1/(n+1),n/(n+1))$ and let $k=\lfloor(n+1)y\rfloor\in\{1,\ldots,n-1\}$. Then $\iota[\bm\psi](x)=k/n$ and hence,
\[
|\iota[\bm\psi](x)-\Psi(x)|
=|\iota[\bm\psi](x)-y|
\le\max\Big\{\frac{k}{n}-\frac{k}{n+1},\frac{k+1}{n+1}-\frac{k}{n}\Big\}
\leq \frac{\max\{y,1-y\}}{n}.
\]
Suppose $y\ge 1-y$ without loss of generality. Both values lie in $(1/(n+1),n/(n+1))$, so
\[
\sqrt{(1-y)/y}
>\sqrt{1/n}
>\sqrt{n}/(n+1)
\]
which yields  $ y/(n+1)<\sqrt{y(1-y)/n}$.
\qedhere
\end{proof}

While Proposition \ref{prop:expected_dependent_discrepancy} provides a universally valid topological bound, directly computing the covariance of indicator functions for arbitrary joint distributions is analytically formidable. However, when the underlying noise process belongs to the Gaussian family, we can dramatically simplify this covariance structure by invoking a profound functional inequality.

\begin{lemma}[Gebelein's Inequality~\cite{MR7220}]
\label{lem:gebelein}
Let $(\zeta_1,\zeta_2)$ be a Gaussian vector with correlation 
$$\rho \coloneqq 
\Cov[\zeta_1,\zeta_2]/\sqrt{\Var(\zeta_1)\Var(\zeta_2)}.$$ 
Then, for any Borel functions $g_1,g_2 : \R \to \R$, 
\begin{equation}
\big| \Cov[g_1(\zeta_1), g_2(\zeta_2)] \big| 
\le |\rho| \sqrt{\Var(g_1(\zeta_1)) \Var(g_2(\zeta_2))}.
\end{equation}
\end{lemma}

\begin{cor}
\label{cor:expected_dependent_discrepancy}
Suppose $\bm{Z}$ is a stationary Gaussian process with $\Psi=\Phi$. Let $\bm\psi$ be generated via either~\eqref{eq:proxies_deterministic} or~\eqref{eq:proxies_random}. Set $\rho_{ij}\coloneqq|\Cov[Z_i,Z_j]|$ for $i,j\in\br{n}$ and $W\coloneqq \max_{k\in\br{n}}|w_k|$. Then
\begin{gather*}
\E\big[\langle\bm{w}, 
|\bm{Z}^\ua-\bm\psi|^r\rangle\big] 
\le C_rW(\sqrt{D_n}+\sqrt{n}),
\quad\text{where}\quad
D_n \coloneqq \sum_{i=1}^n \sum_{j=1}^n \rho_{ij},
\end{gather*}
for $r=1,2$, where 
\begin{gather*}
C_1
\coloneqq\int_{-\infty}^\infty \sqrt{ \Phi(x) (1 - \Phi(x)) } \dd x
\le 1.6148,
\quad
C_2
\coloneqq 4\int_{-\infty}^\infty |x|\sqrt{\Phi(x) (1-\Phi(x))} \dd x
\le 6.7378.
\end{gather*}
\end{cor}

\begin{proof}
The result follows from Proposition~\ref{prop:expected_dependent_discrepancy} as well as Gebelein's inequality (Lemma~\ref{lem:gebelein}), as it establishes~\eqref{eq:indicator_covariance_bound}.
\end{proof}

We say that random variables $\{\xi_n;n\in\N\}$ are $\Oh_{L^p}(a_n)$ for some $p>0$ if $\|\xi_n/a_n\|_{L^p}=\Oh(1)$. Note that, in that case, $\xi_n=\Oh_\p(a_n)$ by Markov's inequality. The following elementary result seems to be unavailable in the literature at this level of generality, despite the abundance of related works under more structured assumptions; see, e.g.,~\cite{MR471045,MR1331224,MR1311984,MR3383341,MR4059190,MR1026312}.

\begin{thm}
\label{thm:Kol-Gauss-dependent}
Let $\bm{W}$ be a stationary zero mean unit variance Gaussian process with covariance function converging to zero at infinity $\lim_{k\to\infty}|\Cov[W_1,W_{k+1}]|=0$. Then
\[
\Kd(\iota[\bm W],\Phi)
=
\Oh_\p\bigl(\sqrt{D_n}/n\bigr),
\quad\text{where}\quad
D_n=\sum_{i=1}^n\sum_{j=1}^n|\Cov[W_i,W_j]|.
\]
\end{thm}

\section{Applications to High-Frequency Stochastic Processes}
\label{sec:sde_applications}

\noindent We now discuss a further application of the preceding results to high-frequency observations of continuous-time stochastic processes. This is a natural setting for the present framework, since the continuous-time decomposition introduced below in~\eqref{eq:continuous_model} leads, after a suitable rescaling of the increments, to an additive array of the form studied in the previous sections.\\

\noindent In high-frequency econometrics, turbulence modelling and signal processing, one is often interested in the scale, or volatility, of the main stochastic driver in the presence of additional perturbations. In the notation of~\eqref{eq:continuous_model}, these perturbations are collected in the process $\mathcal Y$ and may represent drift terms, microstructure effects, or jump components. The relevant feature for our purposes is that, after rescaling, their increments can be controlled through a discrete variation condition. The order-statistic estimators introduced above can therefore be applied directly to the rescaled increments, without requiring the full model to be a semimartingale or the driving noise to be Gaussian.

\subsection{Continuous-Time Model Setup}
Consider a one-dimensional continuous-time stochastic process
$\mathcal{X}=(\mathcal{X}_t;\, t\geq 0)$, observed at the regular high-frequency times
$iT/n$, for $0\leq i\leq n$, over a fixed time horizon $T>0$, with $n\geq 1$. Assume that $\mathcal{X}$ evolves according to the dynamic
\begin{equation}
\label{eq:continuous_model}
\mathcal{X}_t 
= \mathcal{Y}_t 
+ \sigma \mathcal{Z}_t, 
\end{equation}
with $ 0\leq t\leq T,
$ where $\sigma > 0$ is the unknown constant scalar multiple that we seek to estimate. We next specify the nature of each component in a way that connects the continuous-time model with the framework of the previous sections.

\subsubsection{The driving noise process $\mathcal{Z}$} 
The principal stochastic driver $\mathcal{Z}=(\mathcal{Z}_t;\, t\geq 0)$ is assumed to be stochastically continuous, with stationary increments, and $H$-self-similar for some $H>0$. Namely, for every $c>0$,
\[
(\mathcal{Z}_{ct};\, t\geq 0)\eqd (c^H\mathcal{Z}_t;\, t\geq 0),
\]
in the sense of finite-dimensional distributions. We also assume that $\mathcal{Z}_1$ has known distribution function $\Psi$.\\

\noindent The class above includes two important examples. If $\mathcal{Z}$ is a fractional Brownian motion with Hurst index $0<H<1$, then its increments have Gaussian marginal distribution, but they are not independent unless $H=1/2$. On the other hand, if $\mathcal{Z}$ is an $\alpha$-stable L\'evy process, with $0<\alpha<2$, then $H=1/\alpha$ and the increments are independent, although the marginal law is heavy-tailed and the paths have jumps. In particular, for $\alpha\geq 1$, the paths have infinite variation.

\subsubsection{The signal $\mathcal{Y}$}

The additive signal $\mathcal{Y}=(\mathcal{Y}_t;\, t\geq 0)$ represents the perturbative part of the model. Since our estimators apply to discrete additive observations of the form~\eqref{eq:X=Y+Z}, we need a condition ensuring that the increments of $\mathcal{Y}$ become negligible after the high-frequency rescaling. We impose this condition through the variation of $\mathcal{Y}$ over uniform partitions. More precisely, we assume that there exists $p\geq 1$ such that
\begin{equation}
\label{eq:limsup_Delta_Y}
V_p \coloneqq \limsup_{n\to\infty} \sum_{k=1}^{n} \big| \mathcal{Y}_{Tk/n} - \mathcal{Y}_{T(k-1)/n} \big|^p < \infty
\end{equation}
with probability one. We will also require $Hp<1$, which ensures that the signal part has lower activity than the $H$-self-similar noise once the increments are rescaled.

\subsection{Discretization Scheme and Scaled Increments}\label{sec:disc}

\noindent As above, we assume that $\mathcal{X}$ is observed only at the regular times
$iT/n$, with $i$ a natural number and $n$ large. Define the  rescaled discrete increments
\begin{equation}
\label{eq:rescaled_increments}
X_{k} 
\coloneqq (T/n)^{-H} (\mathcal{X}_{Tk/n}-\mathcal{X}_{T(k-1)/n}),
\end{equation}
for $k\leq n $. Then we can write 
\begin{equation}
\label{eq:rescaled_increments_sum}
X_{k} = Y_{k} + \sigma Z_{k},
\end{equation}

where 
\begin{equation}\label{eq:rescaled_increments_sum2}
\begin{gathered}
Y_{k} \coloneqq (T/n)^{-H} (\mathcal{Y}_{Tk/n} - \mathcal{Y}_{T(k-1)/n})
\quad\quad \quad
Z_{k} \coloneqq (T/n)^{-H} (\mathcal{Z}_{Tk/n} - \mathcal{Z}_{T(k-1)/n}).
\end{gathered}
\end{equation}
The self-similarity of $\mathcal{Z}$ ensures that, for each $n$, the rescaled noise increments
$\bm{Z}=(Z_k;\,1\leq k\leq n)$ all have distribution function $\Psi$.

\subsection{Asymptotic Consistency and Optimal Convergence Rates}

\noindent To use the bounds from Section~\ref{sec:weighted_estimators}, it remains to control the size of the rescaled signal vector $\bm{Y}=(Y_k;\,1\leq k\leq n)$ in $\ell^p$. To this end, we first write 
\begin{equation}
\label{eq:signal_variation_expansion}
\|\bm{Y}\|_p 
= (T/n)^{-H} \left( \sum_{k=1}^n \big| \mathcal{Y}_{Tk/n} - \mathcal{Y}_{T(k-1)/n} \big|^p \right)^{1/p}. 
\end{equation}
Hence, assumption~\eqref{eq:limsup_Delta_Y} implies that $\limsup_{n\to\infty}n^{-H}\|\bm{Y}\|_p$ is finite, yielding
\begin{equation}
\label{eq:discrete_signal_bound}
\|\bm{Y}\|_p=\Oh_\p(n^H).
\end{equation}

\noindent The estimate above shows where the condition $Hp<1$ enters. Although $\|\bm{Y}\|_p$ may grow like $\Oh_\p(n^H)$ after rescaling, the bounds of Corollary~\ref{cor:CI_r12} involve the signal through $n^{-1/p}\|\bm{Y}\|_p$. Hence this contribution is negligible when $H-1/p<0$, that is, when $Hp<1$.\\

\noindent In the terminology introduced after Theorem~\ref{thm:CI_r12}, the bound~\eqref{eq:discrete_signal_bound} controls the signal contamination term.  It remains to control the empirical-to-proxy discrepancy appearing in Corollary~\ref{cor:CI_r12}, namely
\[
n^{-1/r}\langle\bm w,|\bm Z^\ua-\bm\psi|^r\rangle^{1/r}.
\]
This term depends on the law and dependence structure of the increments $Z_k$. For this reason, we now treat separately two standard examples: fractional Brownian motion, whose increments have Gaussian marginals but may be dependent, and $\alpha$-stable L\'evy processes, whose increments are independent but heavy-tailed.\\

\noindent  The two cases considered below, fractional Brownian motion and $\alpha$-stable L\'evy processes, are not meant to exhaust all possible examples. They also suggest how to treat related models, such as some generalized tempered stable processes~\cite{MR2798857} or Gaussian mean-reverting models: whenever the driving component can be decomposed into one of these noises plus an additive remainder of lower variation, the latter may be absorbed into $\mathcal Y$ without changing the scale parameter $\sigma$.\\

Corollaries~\ref{cor:fbm}-\ref{cor:stable-mad} illustrate how the same estimators apply in two different high-frequency regimes. In the fractional Brownian case, the main slowdown comes from dependence among increments. In the stable case, the increments are independent but the marginal law is heavy-tailed. In both cases, the estimators act on the rescaled increments and do not require a semimartingale or Markov structure for the full model, nor a Gaussian driving noise.

\subsubsection{Volatility estimation for fractional Brownian motion}

Suppose that the driving noise $\mathcal{Z}$ is a standard fractional Brownian motion with Hurst index $H$, where $0<H<1$. Then $\mathcal{Z}$ is $H$-self-similar and its rescaled increments have standard Gaussian distribution, so that $\Psi=\Phi$. The increments, however, are not independent unless $H=1/2$. Under the weight condition $|\psi_k|^r w_k\leq 1$, Corollary~\ref{cor:sigma-dependent-noise} and Remark~\ref{rem:fBm} give the following estimate.

\begin{cor}[Volatility estimation for fBm]
\label{cor:fbm}
Assume the continuous-time model given in~\eqref{eq:continuous_model}, where $\mathcal{Z}$ is a standard fractional Brownian motion with Hurst index $0<H<1$. Suppose that $\mathcal{Y}$ satisfies~\eqref{eq:limsup_Delta_Y} for some $1\leq p<1/H$. Let $\bm{X}$, $\bm{Y}$ and $\bm{Z}$ be  given as in \eqref{eq:rescaled_increments}-\eqref{eq:rescaled_increments_sum}. Suppose also that $\bm\psi$ is generated either by~\eqref{eq:proxies_deterministic} or by~\eqref{eq:proxies_random}, and that $\bm w=\omega(\bm\psi)$, where $\omega:\R\to[0,1]$ is continuous and satisfies
\[
\sup_{x\in\R} x^2\omega(x)<\infty.
\]
Then, for $r=1,2$, the estimator $\Sigma_r$ is consistent and satisfies
\begin{equation}
\label{eq:rate_fbm}
|\Sigma_r-\sigma|
=
\Oh_\p\bigl(
n^{-\min\{(1/p-H)r,\,1-H,\,1/2\}/r}
\bigr).
\end{equation}
\end{cor}

\subsubsection{Scale estimation for stable L\'evy processes}

Suppose that the driving noise $\mathcal{Z}$ is an $\alpha$-stable L\'evy process, with $1<\alpha<2$. Then $\mathcal{Z}$ has stationary independent increments and is $1/\alpha$-self-similar, so in this case $H=1/\alpha$. The restriction $\alpha>1$ is imposed only to allow a choice of $p\geq 1$ satisfying $p<\alpha$.\\

\noindent The reference distribution $\Psi$ is heavy-tailed, and therefore the choice of weights becomes more relevant than in the Gaussian case. One could tune the decay of $\omega$ according to the tail behaviour of the quantile density, as in Theorems~\ref{thm:weak_lim_proxy_det} and~\ref{thm:weak_lim_proxy_rand}. This, however, would require a more detailed discussion of the tail asymmetry of the stable law. To keep the statement simple, we impose instead that the weight function has compact support.

\begin{cor}[Scale estimation under stable noise]
\label{cor:stable}
Assume the continuous-time model in~\eqref{eq:continuous_model}, where $\mathcal{Z}$ is an $\alpha$-stable L\'evy process with $1<\alpha<2$. Suppose that $\mathcal{Y}$ satisfies~\eqref{eq:limsup_Delta_Y} for some $1\leq p<\alpha$. Let $\bm{X}$, $\bm{Y}$ and $\bm{Z}$ be given as in~\eqref{eq:rescaled_increments}-\eqref{eq:rescaled_increments_sum}. Suppose also that $\bm\psi$ is generated either by~\eqref{eq:proxies_deterministic} or by~\eqref{eq:proxies_random}, and that
\[
\bm w=\omega(\bm\psi),
\]
where $\omega:\R\to[0,1]$ is continuous and satisfies
\[
\sup_{x \in \R} |x|^{2r} \omega(x) < \infty.
\] 
Then, for $r=1,2$, the estimator $\Sigma_r$ is consistent and satisfies
\begin{equation}
\label{eq:rate_stable}
|\Sigma_r-\sigma|
=
\Oh_\p\bigl(n^{1/\alpha-1/p}\bigr).
\end{equation}
\end{cor}

\begin{proof}
Since the increments of $\mathcal{Z}$ are independent, the result follows from Corollary~\ref{cor:sigma-iid} and~\eqref{eq:discrete_signal_bound}, provided that the empirical-to-proxy discrepancy is of order $\Oh_{\p}(n^{-1/2})$ (which we show in the next paragraph). Here $H=1/\alpha$, so the signal term is of order $\Oh_\p(n^{1/\alpha-1/p})$. This is the dominant term, since $p\geq 1$ and $1<\alpha<2$ imply
\[
1/p-1/\alpha<1/2.
\]

Thus, it remains to show that the empirical-to-proxy discrepancy is of order $\Oh_{\p}(n^{-1/2})$. By Theorems~\ref{thm:weak_lim_proxy_det} and~\ref{thm:weak_lim_proxy_rand}, it suffices to check that the pair $(\omega, \Psi)$ satisfies condition $\mathbf{H}_{\kappa,r}$ for $\kappa(u)\coloneqq\max\{\log\log(e+1/u),\log\log(e+1/(1-u))\}$. Because $Z_1$ follows an $\alpha$-stable law with $\alpha \in (1, 2)$, its density exhibits power-law tail decay. By Proposition~\ref{prop:RV-tails}(a) (with $\beta = \alpha$), the quantile Jacobian $\eta = 1 / \Psi' \circ \Psi^{-1}$ varies regularly at the endpoints and, moreover, $\eta(u) \sim c' |\Psi^{-1}(u)|^{\alpha+1}$ and $\sqrt{u(1-u)} \sim c''  |\Psi^{-1}(u)|^{-\alpha/2}$ as $u \to 0$ or $1$. Consequently,
\[
\eta(u)\sqrt{u(1-u)} 
= \Oh\big( |x|^{1 + \alpha/2} \big),
\]
as $|x| \to \infty$. Since $\alpha < 2$, it follows that $r(1 + \alpha/2) < 2r$ for any $r \ge 1$. The assumption $\sup_{x \in \R} |x|^{2r} \omega(x) < \infty$ ensures that the function $u \mapsto \omega(\Psi^{-1}(u))[\kappa(u)\eta(u)\sqrt{u(1-u)}]^r$ is bounded on $(0,1)$ by a regularly varying integrable function, satisfying condition $\mathbf{H}_{\kappa,r}$ and completing the proof. 
\end{proof}

\subsubsection{MAD estimation for fractional Brownian motion and stable L\'evy processes}

We now establish the corresponding high-frequency guarantees for the MAD estimator under the two continuous-time regimes considered above. These results provide the theoretical basis for the convergence rate comparisons in Section~\ref{sec:numerical_experiments}.

\begin{cor}[MAD volatility estimation for fBm]
\label{cor:fbm-mad}
Assume the continuous-time model in~\eqref{eq:continuous_model}, where $\mathcal{Z}$ is a standard fractional Brownian motion with Hurst index $0<H<1$. Suppose that $\mathcal{Y}$ satisfies~\eqref{eq:limsup_Delta_Y} for some $1\leq p<1/H$. Let $\bm{X}$, $\bm{Y}$ and $\bm{Z}$ be given as in~\eqref{eq:rescaled_increments}-\eqref{eq:rescaled_increments_sum}. Then, the estimators $\Sigma_\MAD^\Det$ and $\Sigma_\MAD^\rnd$ are consistent and satisfy
\begin{equation}
\label{eq:rate_fbm_mad}
|\Sigma_\MAD - \sigma|
=
\Oh_\p\Big(
n^{-\min\left\{\frac{1/p - H}{1 + 1/p},\, 1 - H,\, 1/2\right\}}
\Big).
\end{equation}
\end{cor}

\begin{proof}
We invoke Corollary~\ref{cor:mad_correlated} and the signal variation bound~\eqref{eq:discrete_signal_bound}. First, note that since $\mathcal{Z}_1 \sim \Phi$, the reference distribution $\Psi = \Phi$ and its absolute deviation distribution $\Psi_*$ are continuously differentiable with densities strictly positive around their respective medians. Thus, the local Dini condition~\eqref{eq:Dini-LB-F} is satisfied by Remark~\ref{rem:Dini-LB-F}. 

By Remark~\ref{rem:fBm}, the covariance sum for standard fractional Brownian motion satisfies $D_n = \Oh(n^{\max\{1, 2H\}})$. Consequently, the empirical discrepancy term in Corollary~\ref{cor:mad_correlated} contributes
\[
n^{-1}\sqrt{D_n} 
= \Oh\big(n^{-\min\{1-H,\, 1/2\}}\big).
\]
On the other hand, substituting $\|\bm{Y}\|_p = \Oh_\p(n^H)$ from~\eqref{eq:discrete_signal_bound} into $(n^{-1}\|\bm{Y}\|_p^p)^{\frac{1}{p+1}}$ yields
\[
\big(n^{-1}\|\bm{Y}\|_p^p\big)^{\frac{1}{p+1}}
=
\Oh_\p\Big(
\big(n^{Hp - 1}\big)^{\frac{1}{p+1}}
\Big)
=
\Oh_\p\Big(
n^{-\frac{1/p - H}{1 + 1/p}}
\Big).
\]
Combining both error rates via the triangle inequality gives the stated bound for $\Sigma_\MAD^\Det$. Finally, as noted in Section~\ref{sec:MAD}, the auxiliary sample in $\Sigma_\MAD^\rnd$ introduces an $\Oh_\p(n^{-1/2})$ normalization error that is absorbed by the leading terms, completing the proof.
\end{proof}

\begin{cor}[MAD scale estimation under stable noise]
\label{cor:stable-mad}
Assume the continuous-time model in~\eqref{eq:continuous_model}, where $\mathcal{Z}$ is an $\alpha$-stable L\'evy process with $1<\alpha<2$. Suppose that $\mathcal{Y}$ satisfies~\eqref{eq:limsup_Delta_Y} for some $1\leq p<\alpha$. Let $\bm{X}$, $\bm{Y}$ and $\bm{Z}$ be given as in~\eqref{eq:rescaled_increments}-\eqref{eq:rescaled_increments_sum}. Then, the estimators $\Sigma_\MAD^\Det$ and $\Sigma_\MAD^\rnd$ are consistent and satisfy
\begin{equation}
\label{eq:rate_stable_mad}
|\Sigma_\MAD - \sigma|
=
\Oh_\p\Big(
n^{\frac{1/\alpha - 1/p}{1 + 1/p}}
\Big).
\end{equation}
\end{cor}

\begin{proof}
Since the increments of $\mathcal{Z}$ are independent, we apply Corollary~\ref{cor:MAD-iid}. Since $\alpha>1$, then $\Psi$ has a $C^\infty$ positive density, making $\Psi$ and $\Psi_*$ satisfy condition~\eqref{eq:Dini-LB-F}. Thus, the empirical-to-proxy discrepancy is bounded by the Dvoretzky-Kiefer-Wolfowitz rate $\Oh_\p(n^{-1/2})$. 

Because $\mathcal{Z}$ is $1/\alpha$-self-similar, we have $H = 1/\alpha$, and the discrete variation bound~\eqref{eq:discrete_signal_bound} gives $\|\bm{Y}\|_p = \Oh_\p(n^{1/\alpha})$. Substituting this into Corollary~\ref{cor:MAD-iid} yields
\[
\big(n^{-1}\|\bm{Y}\|_p^p\big)^{\frac{1}{p+1}}
=
\Oh_\p\Big(
\big(n^{p/\alpha - 1}\big)^{\frac{1}{p+1}}
\Big)
=
\Oh_\p\Big(
n^{\frac{1/\alpha - 1/p}{1 + 1/p}}
\Big).
\]
It remains to compare this signal contamination rate with the $\Oh_\p(n^{-1/2})$ DKW rate. Because $1 < \alpha < 2$ and $1 \leq p < \alpha$ by assumption, we have
\[
\frac{1 - p/\alpha}{p+1} < \frac{1}{p+1} \leq \frac{1}{2}.
\]
Multiplying by $-1$ shows that $\frac{p/\alpha - 1}{p+1} > -1/2$. Hence, the polynomial decay of the signal contamination term is strictly slower than $n^{-1/2}$ and dominates the asymptotic error. This establishes the rate~\eqref{eq:rate_stable_mad} for both deterministic and randomized versions.
\end{proof}

\section{Numerical Experiments}
\label{sec:numerical_experiments}

In this section, we present a series of controlled numerical experiments designed to validate the non-asymptotic concentration properties and empirical robustness of the proposed order-statistics framework. All simulations are implemented in the Julia programming language and the code is available in the repository~\cite{GitHub}. For simplicity, throughout we consider the setting of scale parameter estimation in high-frequency observations of a continuous-time process of Section~\ref{sec:sde_applications}, which leave room for dependence between signal and noise as well as across the noise. For most experiments, except for Subsection~\ref{subsec:asymptotic_rates}, where the convergence rates (as  a function of $n$) are analysed, $n$ is fixed at $10^4$.

\subsection{Simulation Mechanics and Signal-Noise Profiles}
\label{subsec:sim_mechanics}

Consider the additive model  \eqref{eq:rescaled_increments_sum}, where $\mathcal{Z}$ is a standard Brownian motion, making $\bm{Z}$ an i.i.d. sequence of standard normal random variables. The term $\bm{Y}$, constituted by i.i.d. $\alpha$-stable variables with scale $n^{1/2-1/\alpha}$, for some stability parameter $\alpha\in(0,2)$ and skewness parameter $\beta\in[-1,1]$. Throughout, we will consider a few couplings between $\bm{Z}$ and $\bm{Y}$:
\begin{itemize}[leftmargin=2em]
\item[(I)] $\bm{Y}$ is independent of $\bm{Z}$, 
\item[(II)] $Y_k=\Phi_\alpha^{-1}(\Phi(\pm Z_k))$ for $k\in\br{n}$, where $\Phi_\alpha$ is the distribution function of $Y_k$,
\item[(III)] the elements of $\bm{Y}$ are independent of those of $\bm{Z}$, but their ranks are either matched or inversely matched.\\
\end{itemize}

The second and third couplings stress the dependence between $\bm{Y}$ and $\bm{Z}$, with the second one forcing full positive or negative dependence, while the third coupling is an approximation of the second one, in cases where $\Phi_\alpha^{-1}$ is not numerically accessible (i.e., for $\alpha\ne 1$). Furthermore, to stress the resilience of our estimators, since $\mathcal{Y}$ has paths of finite $p$-variation with finite $p$-moment for $p>\alpha$, but not $p\le \alpha$, we will consider $\alpha=1$ and $\beta=0$ (Cauchy process) as well as $\alpha=1.75$ and $\beta=0.5$. For ease of reference, the case $\alpha=1$ is referred to as Cauchy signal (for which $\Phi_1^{-1}(x)=\tan(\pi x-\pi/2)$, $x\in(0,1)$) and the case $\alpha=1.75$ is simply called $\alpha$-stable signal in the remainder of this experiment.\\

\noindent We produce $M=10^5$ samples of the previous setting with $\sigma=2\pi$ and compute the estimators $\Sigma_r^\Det$ and $\Sigma_r^\rnd$ for $r\in\{1,2\}$ (throughout this section, we use the weight function $\omega_4$ from Table~\ref{tab:Gaussian-breakdown}) as well as the MAD estimators $\Sigma_\MAD^\Det$ and $\Sigma_\MAD^\rnd$. The mean and standard deviation of these estimators are displayed in Table~\ref{tab:Gaussian-breakdown}, while the estimated densities displayed in Figure~\ref{fig:sampling_distributions} (which exclude $r=2$ for visibility) shows a remarkable similarity in behaviour between the MAD estimator $\Sigma_\MAD^\Det$ and $\Sigma_1^\rnd$, both having similar variances, doubling the variance of the estimators $\Sigma_1^\Det$ and $\E[\Sigma_1^\rnd|\bm{X}]$ (estimated via Monte Carlo, averaging over $K=100$ samples of the random proxy $\bm\psi^\rnd$) and having half as much as $\Sigma_\MAD^\rnd$, as predicted by Theorems~\ref{thm:weak_lim_proxy_det} and~\ref{thm:weak_lim_proxy_rand}. All estimators have comparable biases, with the estimators $\Sigma_\MAD^\Det$, $\Sigma_1^\rnd$ and $\E[\Sigma_1^\rnd|\bm{X}]$ attaining slightly smaller biases, see Table~\ref{tab:sampling_distributions}.\\

\begin{figure}[ht]
\centering
\includegraphics[width=0.48\textwidth]{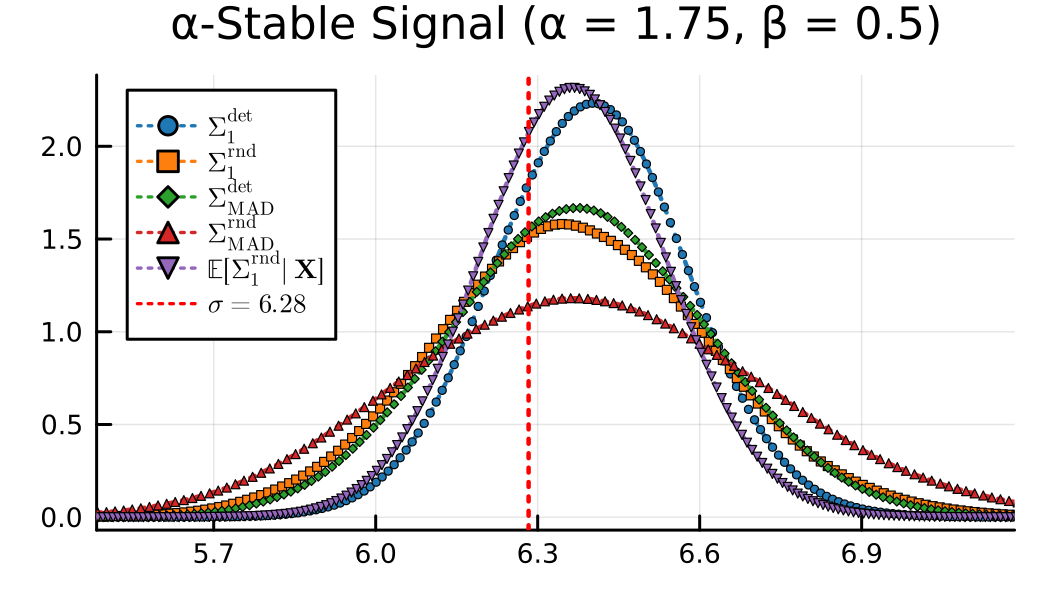}
\includegraphics[width=0.48\textwidth]{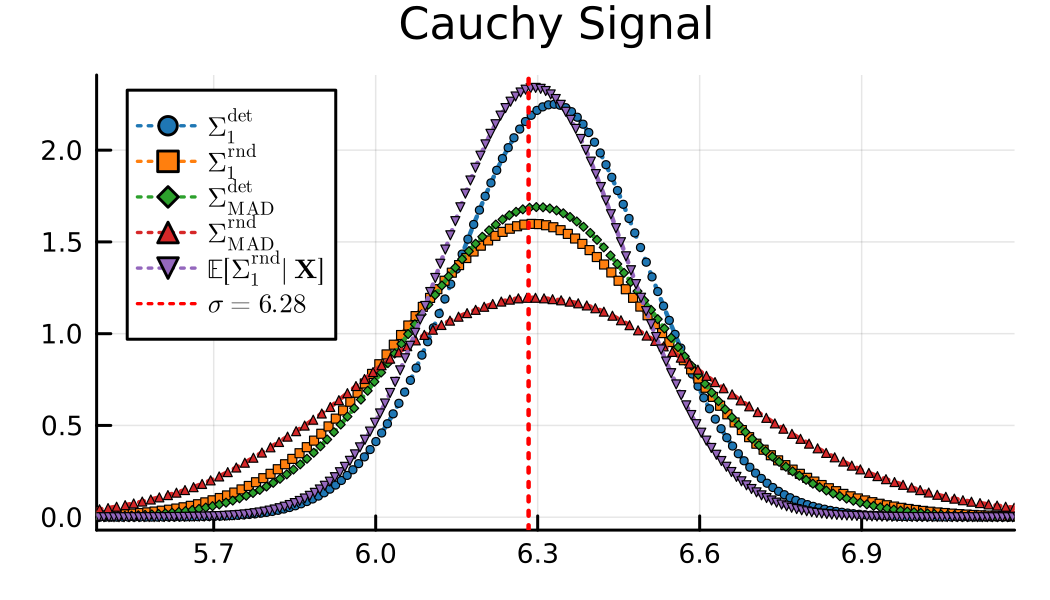}
\caption{Empirical sampling distributions (frequency histograms) of the scale estimators $\Sigma_1$ and $\Sigma_\MAD$ of $M = 10^5$ Monte Carlo trials with sample size $n = 10^4$. The vertical red dashed line indicates the true scale $\sigma=2\pi$.}
\label{fig:sampling_distributions}
\end{figure}

\begin{table}[ht]
\centering
\begin{tabular}{|c|c|c|c|c|c|c|c|c|c|}
\hline
Signal 
    & Stat.
    & $\Sigma_1^\Det$
    & $\Sigma_1^\rnd$
    & $\E[\Sigma_1^\rnd|\bm{X}]$
    & $\Sigma_2^\Det$
    & $\Sigma_2^\rnd$
    & $\E[\Sigma_2^\rnd|\bm{X}]$
    & $\Sigma_\MAD^\Det$
    & $\Sigma_\MAD^\rnd$
\\[3pt] \hline
\multirow{2}{*}{$\alpha$-stable} 
    & Bias
    & 4.18\%
    & {\color{red}3.67\%}
    & {\color{red}3.70\%}
    & 4.22\%
    & 4.08\%
    & 4.1\%
    & {\color{red}3.54\%}
    & {\color{red}3.66\%}
\\[3pt] \cline{2-10}
    & SD
    & {\color{red}5.40\%}
    & 7.66\%
    & {\color{red}5.13\%}
    & {\color{red}4.98\%}
    & 7.03\%
    & {\color{red}5.01\%}
    & 7.36\%
    & 10.4\%
\\[3pt] \hline 
\multirow{2}{*}{Cauchy} 
    & Bias
    & 1.23\%
    & {\color{red}0.75\%}
    & {\color{red}0.77\%}
    & 1.32\%
    & 1.18\%
    & 1.20\%
    & {\color{red}0.83\%}
    & {\color{red}0.95\%}
\\[3pt] \cline{2-10}
    & SD
    & {\color{red}5.35\%}
    & 7.64\%
    & {\color{red}5.10\%}
    & {\color{red}5.28\%}
    & 7.18\%
    & {\color{red}5.28\%}
    & 7.35\%
    & 10.4\%
\\[3pt] \hline
\end{tabular}

\vspace{6pt}
\caption{The table shows the Monte Carlo bias and standard deviation, corresponding to the graphs displayed in Figure~\ref{fig:sampling_distributions}. The ``significantly smaller'' values of each row are high-lighted in red. The Monte Carlo estimation of $\E[\Sigma_1^\rnd|\bm{X}]$ was done by averaging $K=100$ copies of $\Sigma_1^\rnd$ with independent draws for $\bm\psi^\rnd$ for each sample of $\bm{X}$.}
\label{tab:sampling_distributions}
\end{table}

\noindent The information found in Table~\ref{tab:sampling_distributions} suggests that the overall best estimator is the Monte Carlo estimate of $\E[\Sigma_r^\rnd|\bm{X}]$, with the only downside being the complexity of its estimation, which is typically $K$ times more computationally expensive than all other estimators. The biases of all estimators are quite similar, with a slight edge in favour of the estimators $\Sigma_r^\rnd$ and MAD. On the other hand, the standard deviations of $\Sigma_r^\Det$ and $\E[\Sigma_r^\rnd|\bm{X}]$ are approximately a factor of $1/\sqrt{2}$ smaller than those of all other estimators (except $\Sigma_\MAD^\rnd$, whose standard deviation is twice as large), as observed in Figure~\ref{fig:sampling_distributions}. Moreover, all estimators appear to have a slight positive bias or skew, which is somewhat expected in estimators of positive quantities. This leaves the open question of whether there exists some function $\varphi$ for which, for instance, $\E[\varphi(\Sigma_r^\rnd)|\bm{X}]$ has a reduced bias and approximately the same standard deviation. (For instance, this could be attempted in practice by combining $\Sigma_1^\rnd$, $\Sigma_1^\Det$ and $\Sigma_\MAD^\Det$ using the control variates method, but obtaining theoretical guarantees appears to be a hard problem.) In addition, we observe all estimators to have a better performance in the presence of a Cauchy signal, which is extremely heavy tailed but of almost finite variation $p\approx 1$, in contrast to the $\alpha$-stable signal with finite mean but rougher paths $p\approx \alpha=1.75$.\\

The previous figure and table highlight a core structural distinction within our framework: the choice between the deterministic proxy $\bm\psi^\Det$ and the random proxy $\bm\psi^\rnd$. While our theoretical results include concentration inequalities and asymptotic analyses for the terms arising in such concentration inequalities, it is of interest to understand the fluctuations of $\Sigma_r^\rnd$, given $\bm{X}$, which tends to fluctuate (as a function of the auxiliary sample $\bm\xi$) about $\Sigma_r^\Det$. To illustrate this phenomenon, for a single sample of $\bm{X}$, we plot both $\Sigma_r^\Det$ and several samples of $\Sigma_r^\rnd$ (corresponding to $K=10^5$ draws of $\bm\psi^{\rnd}$). The empirical results of this conditional profiling are displayed in Figure~\ref{fig:proxy_fluctuations}.

\begin{figure}[ht]
\centering
\includegraphics[width=.95\textwidth]{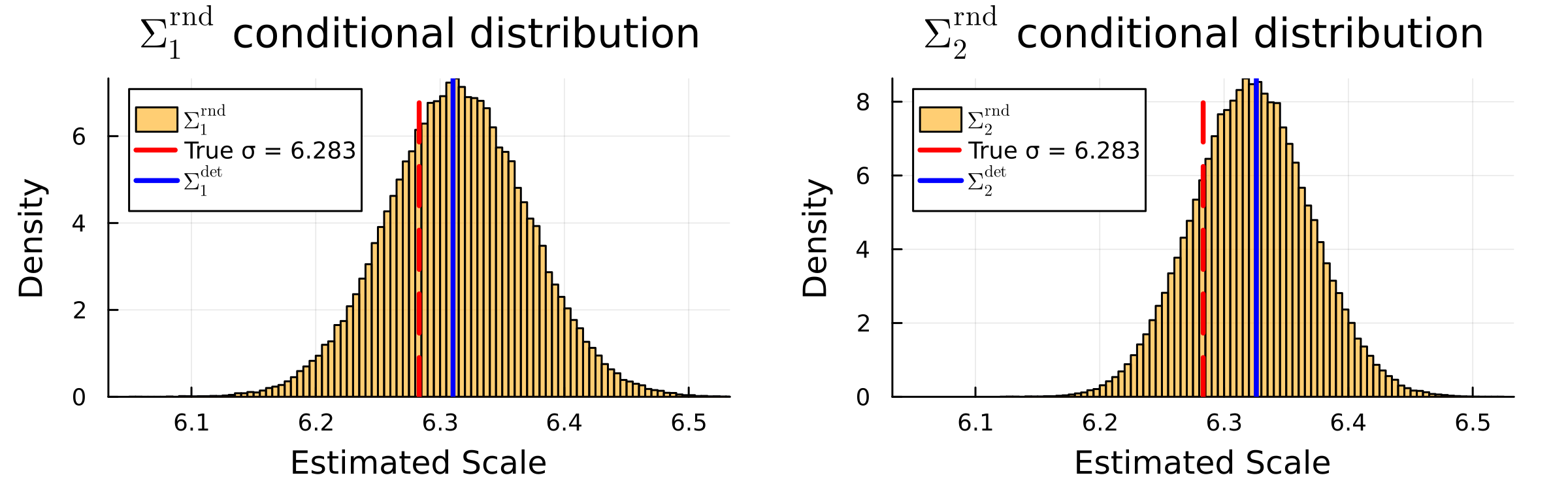}
\caption{Conditional empirical distributions of the random proxy estimator $\Sigma_r^\rnd$ under $K = 10^5$ auxiliary resamples, conditioned on a single realization of $\bm{X}$ with Cauchy signal $\bm{Y}$. In both panels, the orange histograms represent the continuous conditional fluctuation of the random proxy estimates, the solid blue vertical line marks the estimate $\Sigma_r^\Det$, and the dashed red vertical line indicates the true baseline scale $\sigma = 2\pi$.}
\label{fig:proxy_fluctuations}
\end{figure}

\subsection{Asymptotic Rates of Convergence and Robustness}
\label{subsec:asymptotic_rates}

Our next goal is to numerically verify the theoretical bounds of Corollaries~\ref{cor:fbm}-\ref{cor:stable-mad}. We evaluate the asymptotic efficiency and convergence rates of the proposed estimators across increasing sample horizons. This experiment serves as a direct quantitative verification of the sharpness of the non-asymptotic concentration bounds derived in prior sections. Specifically, we analyse how the rate of decay of the mean absolute estimation error is affected by the signal activity and the noise correlation structure.\\

We run $M=100$ Monte Carlo simulations of independent trials per configuration. The true noise scale parameter is set to $\sigma = 2\pi$ and the system is evaluated on the dyadic sequence of sample sizes $n \in \{2^7, \dots, 2^{20}\}$. We consider two signal and noise configurations as follows. In both configurations $\bm{Y}$ is the (scaled) sum of a fractional Gaussian noise with Hurst index $1/\gamma$ and an independent sequence of i.i.d. $\gamma$-stable random variables, both independent of $\bm{Z}$. In the first example, $\bm{Z}$ is itself a fractional Gaussian noise with index $H>1/\gamma$ (as in Corollaries~\ref{cor:fbm} and~\ref{cor:fbm-mad}). In the second example, $\bm{Z}$ is an independent sequence of i.i.d. symmetric (i.e., zero skewness) $\alpha$-stable random variables (as in Corollaries~\ref{cor:stable} and~\ref{cor:stable-mad}) with stability index $\alpha=1/H>\gamma$.\\

We track the Mean Absolute Error (MAE) of the scale estimators across all grid sizes. For the first example we consider the estimators $\Sigma_1^\Det$, $\E[\Sigma_1^\rnd|\bm{X}]$ and $\Sigma_\MAD^\Det$ (other estimators are excluded for visibility). On the other hand, for the second example, since $\Psi^{-1}$ is not easily computable, we consider only $\E[\Sigma_1^\rnd|\bm{X}]$ and $\E[\Sigma_\MAD^\rnd|\bm{X}]$, estimated using $K=100$ independent copies of the random proxy $\bm\psi^\rnd$. To extract the empirical asymptotic rates of convergence, we map the results onto a log-log coordinate space and execute an Ordinary Least Squares (OLS) linear regression mapping (using the largest sample sizes for which the Monte Carlo estimates appear to have reached linearity in Figure~\ref{fig:convergence_rates}):
\begin{equation}
\log(\mathrm{MAE}_n) 
= \beta_0 + \beta_1 \log(n) + \ve_n,
\end{equation}
where $\mathrm{MAE}_n$ denotes the mean absolute error of a given estimator with sample size $n$ and the OLS scaling coefficient $\beta_1$ reveals the empirical rate exponent. The results are summarized in Table~\ref{tab:convergence_rates} and illustrated in Figure~\ref{fig:convergence_rates}.

\begin{table}[ht]
\centering
\begin{tabular}{|c|c|c|c|c|c|c|c|}
\hline
Noise
    & $H=1/\alpha$
    & UB for $\Sigma_1$
    & UB for $\Sigma_\MAD$
    & $\Sigma_1^\Det$
    & $\E[\Sigma_1^\rnd|\bm{X}]$
    & $\Sigma_\MAD^\Det$
    & $\E[\Sigma_\MAD^\rnd|\bm{X}]$
\\ \hline
\multirow{3}{*}{fGn} 
    & $H=.3$
    & .469
    & .265
    & .581
    & .578
    & .543
    & --
\\ \cline{2-8}
    & $H=.5$
    & .269
    & .152
    & .523
    & .566
    & .507
    & --
\\ \cline{2-8}
    & $H=.7$
    & .069
    & .039
    & .165
    & .196
    & .156
    & --
\\ \hline 
\multirow{3}{*}{$\alpha$-stable} 
    & $\alpha=1.8$
    & .214
    & .121
    & --
    & .414
    & --
    & .434
\\ \cline{2-8}
    & $\alpha=1.6$
    & .144
    & .082
    & --
    & .219
    & --
    & .236
\\ \cline{2-8}
    & $\alpha=1.4$
    & .055
    & .031
    & --
    & .071
    & --
    & .082
\\ \hline
\end{tabular}
\vspace{6pt}
\caption{The table shows the Monte Carlo bias decay rate as well as the theoretical bounds (derived via Corollaries~\ref{cor:fbm}-\ref{cor:stable-mad}) for this rate of decay. More precisely, for each applicable estimator $\hat\sigma$, the table shows the best Monte Carlo fit power $\wp$ such that $|\E[\hat\sigma]-\sigma|\approx Cn^{-\wp}$ as a function of the sample size $n$.}
\label{tab:convergence_rates}
\end{table}

\begin{figure}[htbp]
    \centering
    \includegraphics[width=.49\textwidth]{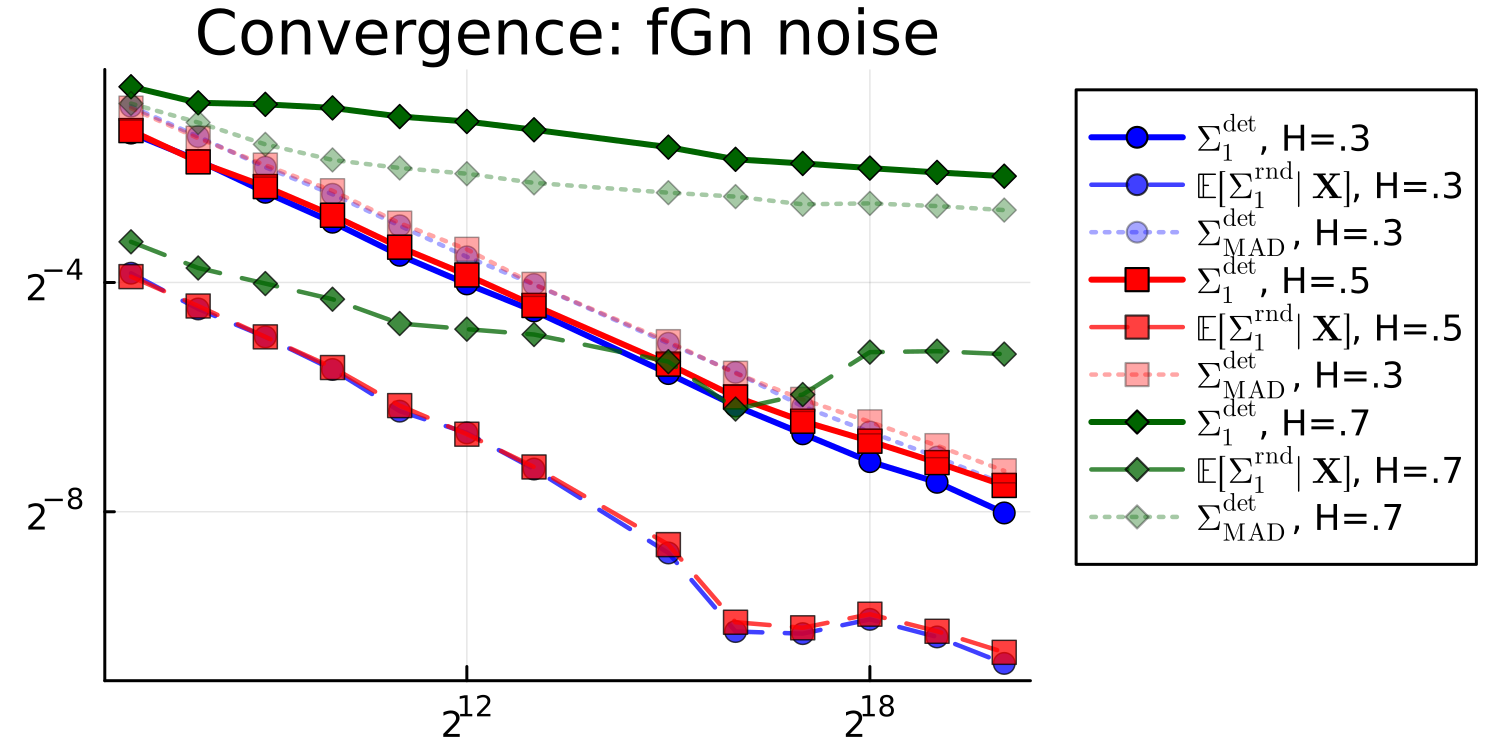}\includegraphics[width=.49\textwidth]{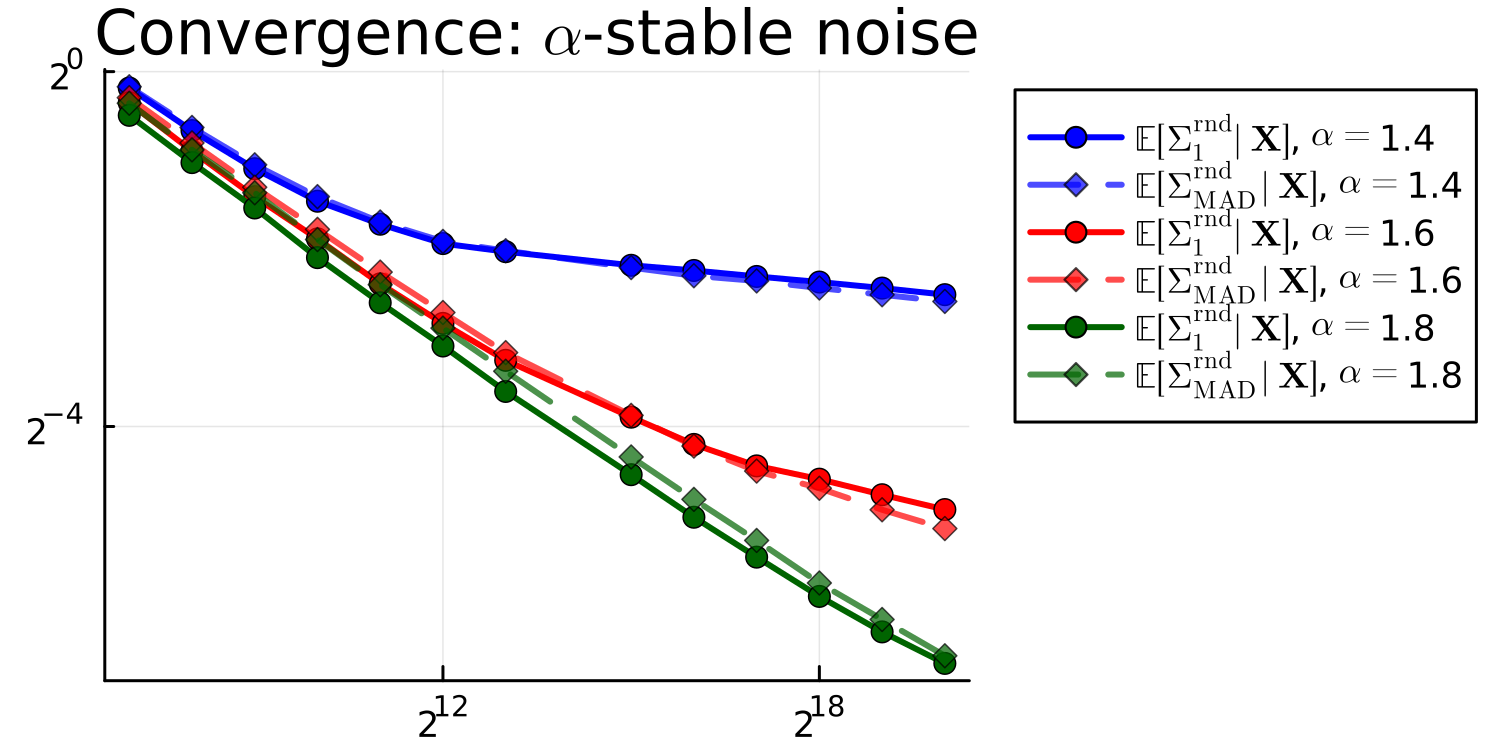}
    \caption{Log-log convergence trajectories of the Mean Absolute Error (MAE) as a function of the dyadic sample dimension $n \in \{2^7,\ldots, 2^{20}\}$ under a baseline noise scale $\sigma = 2\pi$ and $\gamma$-stable signal with stability index $\gamma=1.3$ and skewness $\beta=0.5$. The left (resp. right) panel displays the rates for the case where the noise is a fractional Gaussian noise with index $H$ (resp. i.i.d. $\alpha$-stable with index $\alpha=1/H$) with solid, dashed and dotted (resp. solid and dashed) lines respectively corresponding to $\Sigma^\Det_1$, $\E[\Sigma^\rnd_1|\bm{X}]$ and $\Sigma^\Det_\MAD$ (resp. $\E[\Sigma^\rnd_1|\bm{X}]$ and $\E[\Sigma^\rnd_\MAD|\bm{X}]$).}
    \label{fig:convergence_rates}
\end{figure}

The results presented in Table~\ref{tab:convergence_rates} and Figure~\ref{fig:convergence_rates} demonstrate that all estimators outperform the theoretical bounds established in Section~\ref{sec:sde_applications}. This validates our theoretical guarantees and reinforces the conjecture that the estimators are more efficient than what our current theory predicts. 

\section{Conclusion}
\label{sec:conclusion}

\noindent In this manuscript, we have developed a unified, fully non-parametric mathematical framework for estimating the scale of additive noise in high-dimensional and high-frequency settings. By mapping the additive structure entirely into the spatial domain of order statistics, we have strictly circumvented the catastrophic loss of sparsity induced by classical linear filters, such as the discrete wavelet transform \cite{wavelet_shrinkage, mallat1999wavelet}, and bypassed the fragility of temporal thresholding techniques typically used in stochastic differential equations \cite{MR2770909, MR4140022}.\\

\noindent  Our core theoretical achievement lies in the formulation of generalized $L^r$-weighted concentration inequalities, which elegantly and deterministically decouple the unnormalized structural variation of the sparse signal from the empirical-to-proxy discrepancy of the latent noise. To control this discrepancy without assuming light tails or independence, we introduced the variance-stabilizing methodology of density-matched weights. Whether generating proxies deterministically via the inverse cumulative distribution function or randomly via independent simulation, this weighting mechanism aggressively penalizes boundary singularities, annihilating exploding Jacobians associated with heavy-tailed reference laws~\cite{MR4028181}. \\

\noindent  Consequently, we rigorously proved that our spatial order-statistic estimators that reach or approach the optimal $\Oh(n^{-1/2})$ convergence speed under i.i.d. noise and even under heavily correlated Gaussian noise, wherein we control long-term memory intrinsic to fractional temporal correlation via Gebelein's inequality \cite{MR7220, MR2387368}. This geometric decoupling mathematically guarantees robust, threshold-free volatility recovery across a vast landscape of corrupted physical and econometric systems, maintaining the high finite-sample breakdown resilience championed in classical robust statistics \cite{MR161415, hampel1986robust}.

\appendix

\section{Proofs of auxiliary lemmas and main results}
\label{app:Riemann-Poisson}

This appendix contains the proofs of Lemmas~\ref{lem:Riemann-sums} and~\ref{lem:Poisson-sums} and Theorems~\ref{thm:weak_lim_proxy_det},~\ref{thm:weak_lim_proxy_rand} and~\ref{thm:Kol-Gauss-dependent}.

\begin{proof}[Proof of Lemma~\ref{lem:Riemann-sums}]
By symmetry, we may assume that $f_n$, $f$ and $g$ vanish on $[1/2,1)$. Also, by~\cite[Thm.~1.5.3]{MR1015093}, we may take $g$ to be non-increasing, after replacing it by an integrable non-increasing upper bound in $\RV_0^{-\alpha_0}$ if necessary.\\

\noindent Fix $\ve>0$ and let $0<\delta<1/2$. We split $(0,1/2]$ into the boundary part $B_\delta\coloneqq(0,\delta)$ and the compact bulk $C_\delta\coloneqq[\delta,1/2]$. For two Lebesgue integrable functions $h_1$ and $h_2$, set
\[
\Delta_n(h_1,h_2)
\coloneqq
\left|
\frac{1}{n}
\sum_{k=1}^{\lceil n/2\rceil}
h_1\left(\frac{k}{n+1}\right)
-
\int_0^{1/2}h_2(x)\dd x
\right|.
\]
We write $\Delta_n^A\coloneqq\Delta_n(f_n,f)$ and, for the two pieces of the interval,
\[
\Delta_{n,\delta}^B
\coloneqq
\Delta_n(f_n\1_{B_\delta},f\1_{B_\delta}),
\qquad
\Delta_{n,\delta}^C
\coloneqq
\Delta_n(f_n\1_{C_\delta},f\1_{C_\delta}).
\]

Since $|f_n|\leq g$ and $g$ is non-increasing, we have
\[
\begin{aligned}
\Delta_{n,\delta}^B
&\leq
\frac{n+1}{n}
\sum_{k=1}^{\lceil\delta n\rceil-1}
\int_{(k-1)/(n+1)}^{k/(n+1)}g(x)\dd x
+
\int_{B_\delta}g(x)\dd x\\
&\leq
\frac{n+1}{n}\int_0^\delta g(x)\dd x
+
\int_0^\delta g(x)\dd x .
\end{aligned}
\]

On the other hand,
\[
\begin{aligned}
\Delta^C_{n,\delta}
&\leq
\frac{1}{n}
\sum_{k=\lceil\delta n\rceil}^{\lceil n/2\rceil}
\left|
f_n\left(\frac{k}{n+1}\right)
-
f\left(\frac{k}{n+1}\right)
\right| +
\left|
\frac{1}{n}
\sum_{k=\lceil\delta n\rceil}^{\lceil n/2\rceil}
f\left(\frac{k}{n+1}\right)
-
\int_\delta^{1/2}f(x)\dd x
\right|.
\end{aligned}
\]
The first term tends to zero as $n$ tends to infinity, because $f_n$ converges uniformly to $f$ on $C_\delta$. The second term also tends to zero, since $f$ is locally Riemann integrable. Therefore,
\[
\limsup_{n\to\infty}\Delta_n^A
\leq
\lim_{\delta\da 0}
\limsup_{n\to\infty}
\left(
\Delta_{n,\delta}^B+\Delta_{n,\delta}^C
\right)
=0,
\]
which proves the result.
\end{proof}

\begin{proof}[Proof of Lemma~\ref{lem:Poisson-sums}]
Note that $V_{n,k} \sim \Beta(k, n-k+1)$, so
\[
\frac{1}{n}\sum_{k=1}^n\E[f_n(V_{n,k})]
= \int_0^1 f_n(x) \sum_{j=1}^n \binom{n-1}{j-1} x^{j-1}(1-x)^{n-j} \dd x
=\int_0^1 f_n(x)\dd x.
\]
Thus, the result follows from Lebesgue's dominated convergence theorem.
\end{proof}

\begin{proof}[Proof of Theorem~\ref{thm:weak_lim_proxy_det}]
Let $\bm{V}=\{V_k;\,k\in\br{n}\}$ be an i.i.d. sample of standard uniform random variables. Then $\bm{Z}=\Psi^{-1}(\bm{V})$ and $\bm{Z}^\ua=\Psi^{-1}(\bm{V}^\ua)$. Let $v_k\coloneqq k/(n+1)$, so that $\psi_k=\Psi^{-1}(v_k)$. Since
\[
\lim_{u\da 0}\min\{\kappa(u),\kappa(1-u)\}=\infty,
\]
we may use~\cite[Cor.~4.3.1]{MR815960} to construct standard Brownian bridges $\beta_n$ such that, in probability,
\begin{equation}
\label{eq:empirical-to-bb}
\limsup_{n\to\infty}
\sup_{\frac{1}{n+1}\le u\le \frac{n}{n+1}}
\frac{|\mathcal{V}_n(u)-\beta_u|}{\kappa(u)\sqrt{u(1-u)}}=0,
\end{equation}
where $\mathcal{V}_n(u)\coloneqq \sqrt{n}(V^\ua_{\lceil un\rceil}-u)$ is the associated empirical quantile process.\\

\noindent Fix $0<\delta<1/2$. Under our assumptions, $\Psi^{-1}$ is continuously differentiable on $[\delta,1-\delta]$ with positive derivative $\eta$. Thus, by~\eqref{eq:empirical-to-bb} and the mean value theorem, the process
\[
\mathcal{Z}_n(u)
\coloneqq
\sqrt{n}\bigl(Z^\ua_{\lceil un\rceil}-\Psi^{-1}(u)\bigr)
\]
converges almost surely and uniformly on
$[\delta,1-\delta]$ to $\beta\eta$. Hence, $\widetilde{\mathcal{Z}}_n(u)\coloneqq\sqrt{n}(Z^\ua_{\lceil un\rceil}-\psi_{\lceil un\rceil})$ converges to $\beta\eta$ since, by the mean value theorem, we have
\begin{align*}
\sup_{\delta\le u\le 1-\delta}
|\widetilde{\mathcal{Z}}_n(u)-\mathcal{Z}_n(u)|
&=\sqrt{n}
\sup_{\delta\le u\le 1-\delta}
\Pigl|\Psi^{-1}(u)-\Psi^{-1}\Pigl(\frac{\lceil un\rceil}{n+1}\Pigr)\Pigr|  \le\frac{1}{\sqrt{n}}
\sup_{\delta\le u\le 1-\delta}|\eta(u)|,
\end{align*}
which is finite and tends to zero as $n$ tends to infinity. Therefore,
$\widetilde{\mathcal Z}_n$ also converges almost surely and uniformly to
$\beta\eta$ on $[\delta,1-\delta]$.\\

\noindent Denote $\varpi=\omega\circ\Psi^{-1}$ and define the following linear functionals on $D[\delta,1-\delta]$
\[
\mathcal{I}^\delta_n:
f\mapsto
\frac{1}{n}\sum_{k=\lceil n\delta\rceil}^{\lceil n(1-\delta)\rceil}\varpi(v_k)|f(v_k)|^r 
\quad\quad\quad\quad\quad
\mathcal{I}^\delta:f\mapsto\int_{\delta}^{1-\delta} \varpi(u)|f(u)|^r\dd u.
\]
Then, since $\widetilde{\mathcal{Z}}_n$ converges uniformly to $\beta\eta$ on $[\delta,1-\delta]$, Lemma~\ref{lem:Riemann-sums} implies that with probability one
\[
S_n^\delta
\coloneqq\mathcal{I}_n^\delta(\widetilde{\mathcal{Z}}_n)
\to\mathcal{I}^\delta(\beta\eta)
=\int_\delta^{1-\delta}\varpi(u)\eta(u)^r|\beta_u|^r\dd u\eqqcolon S^\delta,
\]
as $n$ tends to infinity. The case where $\omega\circ\Psi^{-1}$ is supported on a compact subset of $(0,1)$ follows easily. Hence, it remains to consider the other situation.\\

\noindent To complete the proof, it remains to extend the previous weak limit to the whole interval $(0,1)$. Define  
\[
S_n\coloneqq \lim_{\delta\da 0}S_n^{\delta}
\quad\quad\text{and}\quad\quad
S\coloneqq \lim_{\delta\da 0}S^{\delta},
\]
where $S$ is well defined by pathwise monotone convergence. Moreover, $S<\infty$ with probability one, since by assumption
\[
\E[S]
=
\E[|N(0,1)|^r]
\int_0^1
\varpi(u)\eta(u)^r[u(1-u)]^{r/2}\dd u
\le
\E[|N(0,1)|^r]
\int_0^1\varphi(u)\dd u
<\infty.
\] Thus, it remains to show that
\[
\lim_{\delta\da 0}\limsup_{n\to\infty}\p[|S_n-S_n^\delta|>\ve]=0,
\]
for every $\ve>0$.\\

\noindent 
By~\eqref{eq:empirical-to-bb}, the law of the iterated logarithm for a Brownian bridge and the condition on $\kappa$,
\begin{equation}
\label{eq:M_V}
M_n
\coloneqq
\sup_{\frac{1}{n+1}\le u\le \frac{n}{n+1}}
\frac{|\mathcal V_n(u)|}
{\kappa(u)\sqrt{u(1-u)}}
=
\Oh_\p(1).
\end{equation}

\noindent Applying the mean value theorem to $\Psi^{-1}$, we find $W_{k}$ between $V^\ua_{k}$ and $v_{k}$ such that
\begin{equation*}
Z^\ua_{k} - \psi_k 
=\Psi^{-1}(V^\ua_{k}) 
    - \Psi^{-1}(v_{k}) 
= \eta(W_{k}) (V^\ua_{k} - v_{k}).
\end{equation*}
Next, by~\cite[Ineq.~1~\&~2,~p.~415]{MR3396731}, for every sufficiently small $\ve>0$, there exists a constant $0<c_\ve<1$ such that the event
\[
\mathcal{E}_{n,\ve}
\coloneqq
\big\{
c_\ve v_k
\leq
V_k^\ua
\leq
c_\ve^{-1}v_k,
\quad
c_\ve(1-v_k)
\leq
1-V_k^\ua
\leq
c_\ve^{-1}(1-v_k),
\quad
\text{for all }k\in\br n
\big\}.
\]
satisfies $\inf_{n\in\N} \p[\mathcal{E}_{n,\ve}] 
\ge 1 - \ve. $ Since $\eta$ varies regularly at $0$ and $1$, and $W_k$ lies between $v_k$ and $V^\ua_k$, Potter's bounds~\cite[Thm.~1.5.6]{MR1015093} imply that, on $\mathcal{E}_{n,\ve}$,
\[
\max_{k\in\br{n}}
\max
\left\{
\left|\frac{\eta(W_k)}{\eta(v_k)}\right|,
\left|\frac{\eta(v_k)}{\eta(W_k)}\right|
\right\}
\le C_\ve,
\]
where $C_\ve$ is a deterministic constant depending only on $\eta$ and its indices of regular variation.\\

\noindent Thus, on $\mathcal{E}_{n,\ve}$, we obtain 
\[
\sqrt{n}|Z^\ua_k-\psi_k|
\le M_{\mathcal{V}}\,C_\ve\,
\eta(v_k)\,\kappa(v_k)\,\sqrt{v_k(1-v_k)},
\]
for $k\in\br{n}$. The assumption on
\[
\varphi(u)\coloneqq \varpi(u)[\eta(u)\kappa(u)\sqrt{u(1-u)}]^r
\]
and Lemma~\ref{lem:Riemann-sums} yield $T_n^\delta\to T^\delta$, where
\begin{align*}
T_n^\delta
\coloneqq
\frac{1}{n}\sum_{k=1}^n
\varphi(v_k)\1_{\{v_k<\delta\}\cup\{v_k>1-\delta\}},
\quad\quad\quad\quad\quad
T^\delta
\coloneqq
\int_{(0,\delta)\cup(1-\delta,1)}
\varphi(u)\dd u<\infty.
\end{align*}
Hence, on $\mathcal{E}_{n,\ve}$, after relabelling the constant $C_\ve$ if necessary,
\[
S_n-S_n^\delta
=
n^{\frac{r}{2}-1}
\sum_{k=1}^n
w_k|Z^\ua_k-\psi_k|^r
\1_{\{v_k<\delta\}\cup\{v_k>1-\delta\}} \le T_n^\delta M_{\mathcal{V}}^rC_\ve^r.
\]
Moreover, by monotone convergence, $T^\delta\to 0$ as $\delta\da 0$. Hence, for every $\ve'>0$, we have
\[
\p[|S_n-S_n^\delta|>\ve']
\le
\p\bigl[
T_n^\delta M_{\mathcal V}^rC_\ve^r
>\ve' \bigr]
+\ve.
\]
Taking limits as $n$ tends to infinity, $\delta\da 0$ and $\ve\da 0$, completes the proof, since we obtain
\[
\limsup_{\delta\da 0}
\limsup_{n\to\infty}
\p][|S_n-S_n^\delta|>\ve']=0.\qedhere
\]
\end{proof}

\begin{proof}[Proof of Theorem~\ref{thm:weak_lim_proxy_rand}]
We follow a similar strategy as in the proof of Theorem~\ref{thm:weak_lim_proxy_det}. Because $\bm{Z}$ and $\bm{\xi}$ are independent i.i.d. sequences, their respective uniform empirical quantile processes, $\bm{V}_n$ and $\bm{U}_n$, are independent. Each of these processes can be analysed as before. In particular, for any $\delta>0$, on $D[\delta,1-\delta]$ we have
\[
\sqrt{n}(Z^\ua_{\lceil nu \rceil} - \xi^\ua_{\lceil nu \rceil}) 
\stackrel{Law}{\rightarrow} (\beta_u - \tilde\beta_u)\eta(u) 
\eqd \sqrt{2} \beta_u\eta(u),
\]
as $n$ tends to infinity, where $\tilde\beta$ is an independent standard Brownian bridge.\\

\noindent Let $v_k\coloneqq k/(n+1)$, for $k\in\br{n}$, as before. Since $\xi^\ua_{\lceil nu \rceil}$ converges to $\Psi^{-1}(u)$ uniformly on $[\delta,1-\delta]$ with probability one, and $\omega$ is continuous and bounded, $\omega(\xi^\ua_{\lceil nu \rceil})$ converges to $\omega(\Psi^{-1}(u))$ uniformly on $[\delta,1-\delta]$ with probability one. By the continuous mapping theorem and Lemma~\ref{lem:Riemann-sums}, we deduce 
\[
n^{\frac{r}{2}-1}\sum_{k=1}^n 
\omega(\xi^\ua_k)|Z^\ua_k-\xi_k^\ua|^r
\1_{\{\delta\le v_k\le 1-\delta\}}
\stackrel{Law}{\rightarrow}
2^{r/2}\int_\delta^{1-\delta} 
\omega(\Psi^{-1}(u)) \eta(u)^r|\beta_u|^r\dd u.
\]
As in the proof of Theorem~\ref{thm:weak_lim_proxy_det}, the case where $\varpi$ is supported on a compact subset of $(0,1)$ follows easily. Hence, it remains to consider the other situation, for which we only need to bound the tail terms, again denoted by $S_n-S_n^\delta$.\\

\noindent Let $B_\delta\coloneqq(0,\delta)\cup(1-\delta,1)$. Using the inequality
\[
|a-b|^r \le 2^{r-1}(|a|^r+|b|^r),
\]
valid for $a,b\in\R$ and $r\ge 1$, we obtain
\[
|S_n-S_n^\delta|
\le 2^{r-1} \left[ T_n^{(\delta,1)} + T_n^{(\delta,2)} \right],
\]
where
\begin{align*}
T_n^{(\delta,1)}
&\coloneqq 
n^{\frac{r}{2}-1}
\sum_{v_k\in B_\delta} 
\omega(\xi^\ua_k)|Z^\ua_k-\Psi^{-1}(v_k)|^r,\\
T_n^{(\delta,2)}
&\coloneqq 
n^{\frac{r}{2}-1}
\sum_{v_k\in B_\delta} 
\omega(\xi^\ua_k)|\xi_k^\ua-\Psi^{-1}(v_k)|^r.
\end{align*}
It suffices to show that
\begin{equation}
\label{eq:Tj_delta}
\lim_{\delta\da 0}\limsup_{n\to\infty}
\p[T_n^{(\delta,j)}>\ve]=0,
\end{equation}
for every $\ve>0$ and $j=1,2$.\\

\noindent Consider the same probability space of~\cite[Cor.~4.3.1]{MR815960}, constructed for both $\bm{Z}$ and $\bm\xi$, and let $M$ be the largest of the finite random variables $M_{\mathcal{V}}$ in~\eqref{eq:M_V} corresponding to $\bm{Z}$ and~$\bm\xi$. Given a sufficiently small $\ve>0$, by~\cite[Ineq.~1~\&~2,~p.~415]{MR3396731}, there exists a constant $0<c_\ve<1$ such that the events
\begin{align*}
\mathcal{E}^V_{n,\ve}
&\coloneqq
\big\{
c_\ve v_k \le V^\ua_k \le c_\ve^{-1}v_k,
\quad
c_\ve(1-v_k)
\le 1-V^\ua_k
\le c_\ve^{-1}(1-v_k),
\quad\text{for all }k\in\br n
\big\},\\
\mathcal{E}^U_{n,\ve}
&\coloneqq
\big\{
c_\ve v_k
\le U^\ua_k
\le c_\ve^{-1}v_k,
\quad
c_\ve(1-v_k)
\le 1-U^\ua_k
\le c_\ve^{-1}(1-v_k),
\quad
\text{for all }k\in\br n
\big\},
\end{align*}
satisfy
\[
\p[\mathcal{E}^V_{n,\ve}]
=
\p[\mathcal{E}^U_{n,\ve}]
\ge 1-\ve,
\]
for all $n\in\N$. Then, as in the proof of Theorem~\ref{thm:weak_lim_proxy_det}, on the event $\mathcal{E}^V_{n,\ve}$ the random variables $|Z^\ua_k-\Psi^{-1}(v_k)|$ are bounded by
\[
n^{-1/2}M C_\ve \eta(v_k)\kappa(v_k)\sqrt{v_k(1-v_k)},
\]
for $k\in\br{n}$. Similarly, on $\mathcal{E}^U_{n,\ve}$, the variables $|\xi^\ua_k-\Psi^{-1}(v_k)|$ satisfy the same bound.\\

\noindent Thus, on $\mathcal{E}^V_{n,\ve}\cap\mathcal{E}^U_{n,\ve}$, Potter's bound~\cite[Thm.~1.5.6]{MR1015093}, applied at $0$ and $1$ to the regularly varying function $\varpi\ge\omega\circ\Psi^{-1}$, gives, for some $K_\ve\ge 1$,
\[
T_n^{(\delta,1)}
\le 
M^r K_\ve
\frac{1}{n}
\sum_{v_k\in B_\delta} 
\varpi(v_k)\big[\kappa(v_k)\eta(v_k)\sqrt{v_k(1-v_k)}\big]^r.
\]
By the hypothesis on
\[
\varphi(u)\coloneqq
\varpi(u)[\kappa(u)\eta(u)\sqrt{u(1-u)}]^r
\]
and Lemma~\ref{lem:Riemann-sums}, the right-hand side converges to
\[
M^r K_\ve
\int_{B_\delta}\varphi(u)\dd u.
\]
as $n$ tends to infinity. Hence, we may establish~\eqref{eq:Tj_delta} for $j=1$ as in the proof of Theorem~\ref{thm:weak_lim_proxy_det}. The proof for $j=2$, that is, the control of $T_n^{(\delta,2)}$, is completely analogous, using the corresponding high-probability event. This completes the proof.
\end{proof}

\begin{proof}[Proof of Theorem~\ref{thm:Kol-Gauss-dependent}]
We will bound $\|E_n\|_\infty$, where 
$$E_n(x)\coloneqq n^{-1}\sum_{i=1}^n (\1_{\{W_i \le x\}} - \Phi(x))$$ 
is the empirical process. Expanding the indicator function in $L^2(\R, \Phi)$ using Hermite polynomials $H_m$,
\[
\1_{\{W_i \le x\}} - \Phi(x) 
= \sum_{m=1}^\infty \frac{J_m(x)}{m!} H_m(W_i),
\]
where the coefficients are given by $J_m(x) = \E[\1_{\{W \le x\}} H_m(W)] = -\Phi'(x)H_{m-1}(x)$ for any standard normal $W$. Substituting this expansion into $E_n(x)$, we decompose the empirical process into its leading rank-1 projection $T_{n}(x)$ and a remainder series $R_n(x)$ encompassing all higher-order chaos terms ($m \ge 2$):
\[
E_n = T_n + R_n,
\quad 
T_n(x)\coloneqq-\Phi'(x) \frac{1}{n} \sum_{i=1}^n W_i,
\quad R_n(x)\coloneqq\sum_{m=2}^\infty \frac{J_m(x)}{m!} \Bigg( \frac{1}{n} \sum_{i=1}^n H_m(W_i) \Bigg).
\]

\noindent Taking the supremum over $x$ yields:
\[
\sup_{x \in \R} |T_{n}(x)| \le \big(\sup_{x \in \R} \Phi'(x)\big) \Bigg| \frac{1}{n} \sum_{i=1}^n W_i \Bigg| = \frac{1}{\sqrt{2\pi}} \Bigg| \frac{1}{n} \sum_{i=1}^n W_i \Bigg|.
\]
Note that $\E[(\sum_{i=1}^n W_i)^2] = \sum_{i=1}^n\sum_{j=1}^n \Cov[W_i, W_j] \le D_n$. Thus, we have
\[
\sup_{x\in\R} |T_n(x)| 
= \Oh_{L^2}\big(n^{-1}\sqrt{D_n}\big).
\]

\noindent The orthogonality condition 
$$\E[H_m(W_i) H_k(W_j)] = \Cov[W_i, W_j]^m \cdot \1_{\{m=k\}}m!.$$ 
togheter with the fact that $|\Cov[W_i, W_j]| \le 1$, yields the identity
\[
\Var\Bigg[\frac{1}{n} \sum_{i=1}^n H_m(W_i)\Bigg] 
= \frac{m!}{n^2} \sum_{i=1}^n \sum_{j=1}^n \Cov[W_i, W_j]^m 
\le \frac{m! D_n}{n^2}.
\]
Hence, for any $x < y$, the variance of $R_n(y) - R_n(x)$ is 
\begin{align*}
\E[(R_n(y) - R_n(x))^2] 
&= \sum_{m=2}^\infty \frac{(J_m(y) - J_m(x))^2}{(m!)^2} \mathrm{Var}\Bigg[\frac{1}{n} \sum_{i=1}^n H_m(W_i)\Bigg] \\
&\le \frac{D_n}{n^2} \sum_{m=2}^\infty \frac{(J_m(y) - J_m(x))^2}{m!}.
\end{align*}
By Parseval's identity, the sum over all $m \ge 1$ is exactly the variance of $\1_{\{x < W \le y\}}$, so
\[
\E[(R_n(y) - R_n(x))^2]
\le \frac{D_n}{n^2}
\sum_{m=1}^\infty \frac{(J_m(y) - J_m(x))^2}{m!} 
= \frac{D_n}{n^2}\Var(\1_{\{x < W \le y\}}) 
\le \frac{D_n}{n^2}[\Phi(y) - \Phi(x)].
\]

\noindent Write
\[
\rho(k)\coloneqq\Cov[W_1,W_{1+k}].
\]
Fix $\varepsilon\in(0,1/3)$. Since $\rho(k)$ tends to zero, there exists an integer $q\geq1$ such that
\[
|\rho(k)|\leq\varepsilon
\]
for every $k\geq q$. For $a\in\{1,\ldots,q\}$, set
\[
R_{n,a}(x)
\coloneqq
\sum_{\ell=2}^{\infty}
\frac{J_\ell(x)}{\ell!}
\bigg(
\frac1n
\sum_{\substack{1\leq i\leq n\\ i\equiv a\!\!\!\pmod q}}
H_\ell(W_i)
\bigg).
\]
Then
\[
R_n=\sum_{a=1}^q R_{n,a}.
\]
Since $q$ is fixed, it is enough to prove that $nD_n^{-1/2}R_{n,a}$ is tight for every $a=1,\dots,q$. We fix such an $a$ and, to simplify the notation, write $R_n$ in place of $R_{n,a}$ throughout the rest of the argument. A prime on a summation sign (e.g., $\sum_{i=1}^n{}^\prime$) indicates that the summation is restricted to this fixed residue class modulo $q$.\\

\noindent We now control the remainder term $R_n$. To this end, we prove that the sequence $n D_n^{-1/2}R_n$ is tight. We proceed through a modification of Billingsley's argument, following the tightness strategy of Campese, Nourdin and Nualart~\cite{CampeseNourdinNualart}. More precisely, we will show that there exists a constant $C>0$ such that, for every $x\le y\le z$,
\begin{align}\label{eq:billingsgoal}
&\E\left[
|R_n(y)-R_n(x)|^2
|R_n(z)-R_n(y)|^2
\right]\notag\\
&\qquad\le
C\frac{D_n^2}{n^4}
\left[
\bigl(\Phi(z)-\Phi(x)\bigr)^{3/2}
+
\frac{\Phi(z)-\Phi(x)}{D_n}
\right].
\end{align}
The first term is the usual one appearing in Billingsley's criterion, while the second one will be treated separately using the fact that $D_n\geq n$. Writing
\[
G_{x,y}(w)
\coloneqq
\1_{\{x<w\le y\}}
-
\bigl(\Phi(y)-\Phi(x)\bigr)
-
\bigl(J_1(y)-J_1(x)\bigr)w,
\]
we have
\[
R_n(y)-R_n(x)
=
\frac1n\sum_{i=1}^{n}{}' G_{x,y}(W_i),
\]
where $G_{x,y}$ has Hermite rank at least two. Therefore,
\begin{align}\label{eq:refdecomp}
\E\!\left[
|R_n(y)-R_n(x)|^2
|R_n(z)-R_n(y)|^2
\right]=
\frac1{n^4}
\sum_{i,j,k,\ell=1}^{n}\!\!\!{}'
\E\!\left[
G_{x,y}(W_i)
G_{x,y}(W_j)
G_{y,z}(W_k)
G_{y,z}(W_\ell)
\right].
\end{align}
We split the above sum according to the coincidence pattern of the indices. More precisely, let
\begin{align*}
\mathcal I_4
&\coloneqq
\{(i,j,k,\ell): i,j,k,\ell \text{ are all distinct}\},\\
\mathcal I_{2,1,1}
&\coloneqq
\{(i,j,k,\ell): \text{exactly three indices are distinct}\},\\
\mathcal I_{2,2}
&\coloneqq
\{(i,j,k,\ell): \text{exactly two distinct indices, each appearing twice}\},\\
\mathcal I_{3,1}
&\coloneqq
\{(i,j,k,\ell): \text{exactly two distinct indices, one appearing three times}\},\\
\mathcal I_1
&\coloneqq
\{(i,j,k,\ell): i=j=k=\ell\}.
\end{align*}

We will the following technical estimate, in order to implement the estimations from \cite{MR471045}
\begin{equation}
\label{eq:rho-square-Dn}
n\sum_{k=0}^{n-1}\rho(k)^2
\le
3D_n.
\end{equation}
In order to show this, denote $m=\lfloor n/2\rfloor$, we have
\[
n\sum_{k=0}^{m}\rho(k)^2
\le
n\sum_{k=0}^{m}|\rho(k)|
\le
D_n.
\]
Moreover, the Cauchy-Schwarz inequality yields
\[
\bigg(\sum_{k=m+1}^{n-1}\rho(k)^2\bigg)^2
\le
\E\bigg[\bigg(\sum_{k=m+1}^{n-1}\rho(k)W_{k+1}\bigg)^2\bigg]
\le
\sum_{k,\ell=m+1}^{n-1}|\rho(k)\rho(\ell)\rho(|k-\ell|)|.
\]
Since $|k-\ell|\le m$ whenever $k,\ell>m$, we obtain
\[
\sum_{k,\ell=m+1}^{n-1}|\rho(k)\rho(\ell)\rho(|k-\ell|)|
\le
\bigg(1+2\sum_{h=1}^{m}|\rho(h)|\bigg)\sum_{k=m+1}^{n-1}\rho(k)^2.
\]
Hence
\[
n\sum_{k=m+1}^{n-1}\rho(k)^2
\le
2D_n,
\]
which proves~\eqref{eq:rho-square-Dn}.\\

\noindent The proof of Proposition~3.1 in~\cite{MR471045} also shows that the constant in Lemma~4.5 is bounded by a constant depending only
on $p$ and $\varepsilon$, multiplied by the product of the $L^2$-norms of the functions involved. Moreover, for every $r\geq2$,
\[
\|G_{x,y}\|_{L^r}
\leq
C_r\bigl(\Phi(y)-\Phi(x)\bigr)^{1/r}.
\]
For $(i,j,k,\ell)\in\mathcal I_4$, the expectation appearing in \eqref{eq:refdecomp} involves four functions of Hermite rank at least two. Therefore, Lemma~4.5 of~\cite{MR471045}, together with~\eqref{eq:rho-square-Dn}, shows that the contribution of these indices is bounded by
\[
C\frac{D_n^2}{n^4}
\bigl(\Phi(y)-\Phi(x)\bigr)
\bigl(\Phi(z)-\Phi(y)\bigr)
\le
\frac{C}{4}\frac{D_n^2}{n^4}
\bigl(\Phi(z)-\Phi(x)\bigr)^2.
\]
\noindent For $(i,j,k,\ell)\in\mathcal I_{2,1,1}$, exactly one index is repeated. Up to a relabeling, the corresponding expectations are of the forms
\[
\E\!\left[
G_{x,y}(W_i)^2G_{y,z}(W_j)G_{y,z}(W_k)
\right],
\qquad
\E\!\left[
G_{y,z}(W_i)^2G_{x,y}(W_j)G_{x,y}(W_k)
\right],
\]
or
\[
\E\!\left[
G_{x,y}(W_i)G_{y,z}(W_i)
G_{x,y}(W_j)G_{y,z}(W_k)
\right],
\]
where $i,j,k$ are pairwise distinct. In each case, the two functions corresponding to the indices appearing only once have Hermite rank at least two. Therefore, Lemma~4.5 of~\cite{MR471045}, together with~\eqref{eq:rho-square-Dn}, yields
\[
C\frac{D_n^2}{n^4}
\|G_{x,y}^2\|_{L^2}
\|G_{y,z}\|_{L^2}^2.
\]
An application of H\"older's inequality gives
\[
C\frac{D_n^2}{n^4}
\|G_{x,y}\|_{L^4}^2
\|G_{y,z}\|_{L^2}^2
\le
C\frac{D_n^2}{n^4}
\bigl(\Phi(y)-\Phi(x)\bigr)^{1/2}
\bigl(\Phi(z)-\Phi(y)\bigr).
\]
Similarly, the second type contributes at most
\[
C\frac{D_n^2}{n^4}
\|G_{y,z}^2\|_{L^2}
\|G_{x,y}\|_{L^2}^2
\le
C\frac{D_n^2}{n^4}
\bigl(\Phi(y)-\Phi(x)\bigr)
\bigl(\Phi(z)-\Phi(y)\bigr)^{1/2}.
\]
For the third type, Lemma~4.5 gives
\[
C\frac{D_n^2}{n^4}
\|G_{x,y}G_{y,z}\|_{L^2}
\|G_{x,y}\|_{L^2}
\|G_{y,z}\|_{L^2}.
\]
By H\"older's inequality argument,  this term becomes bounded by
\[
C\frac{D_n^2}{n^4}
\bigl(\Phi(y)-\Phi(x)\bigr)^{3/4}
\bigl(\Phi(z)-\Phi(y)\bigr)^{3/4}.
\]
The contribution of $\mathcal I_{2,1,1}$ is then bounded by
\[
C\frac{D_n^2}{n^4}
\left[
\bigl(\Phi(y)-\Phi(x)\bigr)^{1/2}
\bigl(\Phi(z)-\Phi(y)\bigr)
+
\bigl(\Phi(y)-\Phi(x)\bigr)
\bigl(\Phi(z)-\Phi(y)\bigr)^{1/2}
\right].
\]
From the analysis above it follows that the previous expression is at most
\[
C\frac{D_n^2}{n^4}
\bigl(\Phi(z)-\Phi(x)\bigr)^{3/2}.
\]
\noindent For $(i,j,k,\ell)\in\mathcal I_{2,2}$, the corresponding expectations are, up to a relabeling, of the forms
\[
\E\!\left[
G_{x,y}(W_i)^2G_{y,z}(W_j)^2
\right]
\]
and
\[
\E\!\left[
G_{x,y}(W_i)G_{y,z}(W_i)
G_{x,y}(W_j)G_{y,z}(W_j)
\right],
\]
for $i\neq j$. Observe that
\[
\E\!\left[G_{x,y}(W)^2\right]
\le
\Phi(y)-\Phi(x),
\qquad
\E\!\left[|G_{x,y}(W)|^4\right]
\le
C\bigl(\Phi(y)-\Phi(x)\bigr),
\]
and analogously for $G_{y,z}$. Hence, by Gebelein's inequality,
\[
\begin{aligned}
\E\!\left[
G_{x,y}(W_i)^2G_{y,z}(W_j)^2
\right]
&\le
\bigl(\Phi(y)-\Phi(x)\bigr)
\bigl(\Phi(z)-\Phi(y)\bigr)\\
&\quad+
C|\Cov[W_i,W_j]|
\sqrt{
\bigl(\Phi(y)-\Phi(x)\bigr)
\bigl(\Phi(z)-\Phi(y)\bigr)
}.
\end{aligned}
\]
The same bound holds, in absolute value, for the second type of expectation, by applying Gebelein's inequality to $G_{x,y}G_{y,z}$. Summing over $\mathcal I_{2,2}$, using $D_n\ge n$, and absorbing the finite number of possible coincidence patterns into the constant, we obtain
\[
\begin{aligned}
&\frac{1}{n^4}
\sum_{(i,j,k,\ell)\in\mathcal I_{2,2}}^{\prime}
\left|
\E\!\left[
G_{x,y}(W_i)
G_{x,y}(W_j)
G_{y,z}(W_k)
G_{y,z}(W_\ell)
\right]
\right|\\
&\qquad\le
C\frac{D_n^2}{n^4}
\left[
\bigl(\Phi(z)-\Phi(x)\bigr)^2
+
\frac{\Phi(z)-\Phi(x)}{D_n}
\right].
\end{aligned}
\]
\noindent For $(i,j,k,\ell)\in\mathcal I_{3,1}$, the corresponding expectations are, up to a relabeling, of the forms
\[
\E\!\left[
G_{x,y}(W_i)^2G_{y,z}(W_i)G_{y,z}(W_j)
\right]
\quad
\quad
\quad
\quad
\quad
\quad\E\!\left[
G_{x,y}(W_i)G_{y,z}(W_i)^2G_{x,y}(W_j)
\right],
\]
for $i\neq j$. Since $G_{x,y}$ and $G_{y,z}$ are centered, Gebelein's inequality and H\"older's inequality give
\[
\begin{aligned}
&\left|
\E\!\left[
G_{x,y}(W_i)^2G_{y,z}(W_i)G_{y,z}(W_j)
\right]
\right|\\
&\qquad\le
C|\Cov[W_i,W_j]|
\bigl(\Phi(y)-\Phi(x)\bigr)^{1/4}
\bigl(\Phi(z)-\Phi(y)\bigr)^{3/4}.
\end{aligned}
\]
The other type of expectation satisfies the analogous bound with the two increments interchanged. Hence, summing over $\mathcal I_{3,1}$, we obtain
\[
C\frac{D_n}{n^4}\bigl(\Phi(z)-\Phi(x)\bigr)
=
C\frac{D_n^2}{n^4}
\frac{\Phi(z)-\Phi(x)}{D_n}.
\]

\noindent Finally, for $(i,j,k,\ell)\in\mathcal I_1$, we have $i=j=k=\ell$. By H\"older's inequality and the fourth-moment estimates used above,
\[
\begin{aligned}
\E\!\left[
G_{x,y}(W_i)^2G_{y,z}(W_i)^2
\right]
&\le
\|G_{x,y}(W)\|_{L^4}^2
\|G_{y,z}(W)\|_{L^4}^2\\
&\le
C\sqrt{
\bigl(\Phi(y)-\Phi(x)\bigr)
\bigl(\Phi(z)-\Phi(y)\bigr)
}\le
C\bigl(\Phi(z)-\Phi(x)\bigr).
\end{aligned}
\]
Since $\mathcal I_1$ contains at most $n$ elements and $D_n\geq n$, its contribution is bounded by
\[
C\frac{n}{n^4}\bigl(\Phi(z)-\Phi(x)\bigr)
\le
C\frac{D_n^2}{n^4}
\frac{\Phi(z)-\Phi(x)}{D_n}.
\]
\noindent Combining the estimates for the five coincidence patterns, we obtain~\eqref{eq:billingsgoal}. Define
\[
\widetilde R_n(u)
\coloneqq
nD_n^{-1/2}R_n\bigl(\Phi^{-1}(u)\bigr),
\]
and set $\widetilde R_n(0)=\widetilde R_n(1)=0$. Then, for every
$0\leq u\leq v\leq w\leq1$,
\[
\begin{aligned}
&\E\left[
|\widetilde R_n(v)-\widetilde R_n(u)|^2
|\widetilde R_n(w)-\widetilde R_n(v)|^2
\right]
\leq
C\big(
(w-u)^{3/2}
+
(w-u)/D_n
\big).
\end{aligned}
\]
Moreover, the second-moment estimate obtained above gives
\[
\E\left[
|\widetilde R_n(v)-\widetilde R_n(u)|^2
\right]
\leq v-u.
\]

Let
\[
\mathcal T_n\coloneqq\{k/n:\,k=0,\ldots,n\}.
\]
For $u\leq v\leq w$ in $\mathcal T_n$, either $u=w$, in which case the corresponding increments vanish, or $w-u\geq n^{-1}$. Since $D_n\geq n$, we have
\[
\frac{w-u}{D_n}
\leq
\frac{w-u}{n}
\leq
(w-u)^2
\leq
(w-u)^{3/2}.
\]
Consequently,~\eqref{eq:billingsgoal} gives
\[
\E\left[
|\widetilde R_n(v)-\widetilde R_n(u)|^2
|\widetilde R_n(w)-\widetilde R_n(v)|^2
\right]
\leq
C(w-u)^{3/2},
\]
for $u\leq v\leq w$ in $\mathcal T_n$.  For $u\in[0,1)$, set $\pi_n(u)\coloneqq \lfloor nu\rfloor/n$ with  $\pi_n(1)\coloneqq1,$ and define
\[
\widetilde R_n^\circ(u)
\coloneqq
\widetilde R_n(\pi_n(u)).
\]
The usual Billingsley argument, applied on the grid $\mathcal T_n$, together with the second-moment estimate obtained above, shows that $\widetilde R_n^\circ$ is tight in $D[0,1]$.\\

It remains to compare $\widetilde R_n$ and $\widetilde R_n^\circ$. The same decomposition used to prove~\eqref{eq:billingsgoal}, with the same increment in all four factors, gives
\[
\E\left[
|\widetilde R_n(v)-\widetilde R_n(u)|^4
\right]
\leq
C\left[
(v-u)^{3/2}
+
\frac{v-u}{D_n}
\right],
\]
for $0\leq u\leq v\leq1$. Hence, by Markov's inequality,
\[
\begin{aligned}
&\p\left[
\max_{0\leq k<n}
\left|
\widetilde R_n\left(\frac{k+1}{n}\right)
-
\widetilde R_n\left(\frac{k}{n}\right)
\right|>\eta
\right]\\
&\qquad\leq
\frac{C}{\eta^4}
\sum_{k=0}^{n-1}
\left[
n^{-3/2}
+
\frac{1}{nD_n}
\right]
\leq
\frac{C}{\eta^4}
\left[
n^{-1/2}
+
D_n^{-1}
\right],
\end{aligned}
\]
which converges to zero.\\

\noindent Using $J_1(x)=-\Phi'(x)$, we may write
\[
\widetilde R_n(u)
=
\frac{1}{\sqrt{D_n}}
\sum_{i=1}^{n}{}'
\left[
\1_{\{\Phi(W_i)\leq u\}}
-u
+\Phi'(\Phi^{-1}(u))W_i
\right].
\]
The function
\[
u\longmapsto\Phi'(\Phi^{-1}(u)),
\qquad 0<u<1,
\]
extends continuously to $[0,1]$ by assigning value zero at the endpoints. Set
\[
\omega_n
\coloneqq
\sup_{\substack{u,v\in[0,1]\\|u-v|\leq1/n}}
\left|
\Phi'(\Phi^{-1}(u))
-
\Phi'(\Phi^{-1}(v))
\right|.
\]
Then $\omega_n\to0$. By the monotonicity of
\[
u\longmapsto
\sum_{i=1}^{n}{}'\1_{\{\Phi(W_i)\leq u\}},
\]
for $k/n\leq u\leq(k+1)/n$ we have
\[
\begin{aligned}
\left|
\widetilde R_n(u)
-
\widetilde R_n\left(\frac{k}{n}\right)
\right|
&\leq
\left|
\widetilde R_n\left(\frac{k+1}{n}\right)
-
\widetilde R_n\left(\frac{k}{n}\right)
\right|+
\frac{2}{\sqrt{D_n}}
+
2\omega_n
\left|
\frac{1}{\sqrt{D_n}}
\sum_{i=1}^{n}{}'W_i
\right|.
\end{aligned}
\]
Moreover,
\[
\E\left[
\left|
\frac{1}{\sqrt{D_n}}
\sum_{i=1}^{n}{}'W_i
\right|^2
\right]
\leq1.
\]
It follows that
\[
\sup_{u\in[0,1]}
\left|
\widetilde R_n(u)
-
\widetilde R_n^\circ(u)
\right|
\cip0.
\]
Since $\widetilde R_n^\circ$ is tight, the sequence $\widetilde R_n$ is tight in $D[0,1]$.\\

\noindent Returning to the original notation, we have proved that $nD_n^{-1/2}R_{n,a}$ is tight for every $a\in\{1,\ldots,q\}$. Since
\[
R_n=\sum_{a=1}^q R_{n,a}
\]
and $q$ is fixed, we obtain
\[
\sup_{x\in\R}|R_n(x)|
\leq
\sum_{a=1}^q\sup_{x\in\R}|R_{n,a}(x)|
=
\Oh_{\p}\left(n^{-1}\sqrt{D_n}\right).
\]
Finally, applying the triangle inequality to the initial decomposition yields:
\[
\sup_{x \in \R} |E_n(x)| \le \sup_{x \in \R} |T_{n}(x)| + \sup_{x \in \R} |R_n(x)| 
= \Oh_{\p}\big(n^{-1}\sqrt{D_n}\big).\qedhere
\]
\end{proof}
\noindent \textbf{Acknowledgements}\\
Arturo Jaramillo Gil was supported by
the grant CBF2023-2024-2088. Jorge Gonz\'alez C\'azares was supported by DGAPA-PAPIIT grant 36-IA104425 and EPSRC grant EP/V009478/1.

\bibliographystyle{plain}
\bibliography{ref}

\end{document}